\documentclass[10pt]{amsart}
\usepackage[bookmarksopen,bookmarksdepth=2]{hyperref}
\usepackage{etex}
\usepackage{graphicx}
\usepackage{amssymb,amsmath,amsfonts,amsthm,epsfig,amscd,stmaryrd}
\usepackage{stmaryrd}
\usepackage[all,cmtip,poly]{xy}
\usepackage{a4wide}
\usepackage{color}

\def\NwithzeroB{\mathbf{N}}
\def\NwithzeroA{\NwithzeroB}
\def\NwithoutzeroA{\mathbf{N}_{> 0}}

\def\PSL{\mathrm{PSL}}

\def\wh{\widetilde{h}}
\def\WM{\widetilde{M}}

\def\ab{\mathrm{ab}}
\def\Gbox{G^{\square}}
\def\Gbox{\mathrm{S}G}
\def\Hcong{H^{1,\mathrm{cong}}}
\def\Gammah{\widehat{\Gamma}}
\def\Zhat{\widehat{\Z}}
\def\Htw{\widetilde{H}}
\def\sss{\mathrm{ss}}

\def\SSS{\S}
\def\SSSS{\S}
\def\SSSSs{\S\S}
\def\TS{\S}

\def\llp{(\kern-0.15em{(}}
\def\rrp{)\kern-0.15em{)}}

\renewcommand\setminus{\smallsetminus}
\def\spec{\mathrm{Spec} \,}

\def\HH{\mathcal{H}}

\def\wGamma{\widetilde{\Gamma}}
\def\wPhi{\widetilde{\Phi}}
\def\wwPhi{\widetilde{\Psi}}

\def\wF{\widetilde{F}}
\def\wpsi{\widetilde{\psi}}
\def\wOm{\widetilde{\Omega}}
\def\PSO{\mathrm{PSO}}
\def\wr{\widetilde{r}}
\def\ws{\widetilde{s}}
\def\wt{\widetilde{t}}

\def\PSU{\mathrm{PSU}}

\def\wH{\widetilde{H}}
\def\wgamma{\widetilde{\gamma}}

\usepackage{graphicx}
\usepackage{diagrams}

\newmuskip\pFqmuskip

\newcommand*\pFq[6][8]{
  \begingroup 
  \pFqmuskip=#1mu\relax
  \mathcode`\,=\string"8000
  \begingroup\lccode`\~=`\,
  \lowercase{\endgroup\let~}\pFqcomma
  {}_{#2}F_{#3}{\left[\genfrac..{0pt}{}{#4}{#5};#6\right]}%
  \endgroup
}
\newcommand{\pFqcomma}{\mskip\pFqmuskip}

\def\T{\mathbf{T}}

\def\Ext{\mathrm{Ext}}

\def\F{\mathbf{F}}
\def\Z{\mathbf{Z}}

\def\Hom{\mathrm{Hom}}
\def\R{\mathbf{R}}

\def\Q{\mathbf{Q}}

\def\GL{\mathrm{GL}}
\def\Qbar{\overline{\Q}}

\def\bi{\mathbf{i}}
\def\bx{\mathbf{x}}
\def\bk{\mathbf{k}}

\newif\iffinalrun
\iffinalrun
\else
 \fi

\iffinalrun
  \newcommand{\need}[1]{}
  \newcommand{\mar}[1]{}
\else
  \newcounter{margcount}
  \newcommand{\need}[1]{{\tiny *** #1}}
  \newcommand{\mar}[1]{\stepcounter{margcount}\marginpar{\raggedright\tiny \themargcount.FIXME #1}}
\fi

\renewcommand\mathbb{\mathbf}

\newtheorem{theorem}[subsubsection]{Theorem}
\newtheorem{cor}[subsubsection]{Corollary}
\newtheorem{proposition}[subsubsection]{Proposition}

\newtheorem{question}[subsubsection]{Question}
\newtheorem{example}[subsubsection]{Example}

\newtheorem{thm}[subsubsection]{Theorem}
\newtheorem{lemma}[subsubsection]{Lemma}

\newtheorem{df}[subsubsection]{Definition}
\newtheorem{note}[subsubsection]{Notation}
\newtheorem{corr}[subsubsection]{Corollary}

\theoremstyle{remark}
\newtheorem{remark}[subsubsection]{Remark}

\numberwithin{equation}{subsection}
\def\numequation{\addtocounter{subsubsection}{1}\begin{equation}}

\counterwithin{figure}{equation}

\let\oldfigure\figure
\let\endoldfigure\endfigure
\renewenvironment{figure}
  {\refstepcounter{equation}\oldfigure}
  {\endoldfigure}

\usepackage[bookmarksopen,bookmarksdepth=2]{hyperref}
\usepackage{tikz}
\usepackage{tikz-cd}
\usepackage{mathtools}

\def\C{\mathbf{C}}

\def\Gal{\mathrm{Gal}}

\def\SL{\mathrm{SL}}

\newmuskip\pFqmuskip

\def\H{\mathbf{H}}

\title{The Unbounded Denominators Conjecture}

\author[F. Calegari]{Frank Calegari}
 \email{fcale@math.uchicago.edu} \address{The University of Chicago,
5734 S University Ave,
Chicago, IL 60637, USA}

\author[V. Dimitrov]{Vesselin Dimitrov}
 \email{vesselin.dimitrov@gmail.com} \address{Department of Mathematics, 
 California Institute of Technology,
Pasadena, CA 91125, USA}

\author[Y. Tang]{Yunqing Tang}
\email{yungqing.tang@berkeley.edu} \address{Department of Mathematics, University of California, Berkeley, Evans Hall, Berkeley, CA 94720, USA}

\thanks{F.C. was supported in part by NSF Grant DMS-2001097. Y.T. was supported in part by NSF grant DMS-2231958 and a Sloan Research Fellowship. Some of the work was done when Y.T. was at CNRS and Universit\'e Paris-Saclay from February 2020 to June 2021 and at Princeton University from July 2021 to June 2022.}

\begin{document}

\begin{abstract}
We prove the unbounded denominators conjecture in the theory of noncongruence modular forms for finite index subgroups of $\SL_2(\Z)$. Our result includes also Mason's generalization of the original conjecture to the setting of vector-valued modular forms, thereby supplying a new path to the congruence property in rational conformal field theory. The proof involves a new arithmetic holonomicity bound of a potential-theoretic flavor, together with Nevanlinna second main theorem, the congruence subgroup property of $\SL_2(\Z[1/p])$, and a close description of the Fuchsian uniformization $D(0,1)/\Gamma_N$ of the Riemann surface $\C \setminus \mu_N$. 
\end{abstract}

\maketitle

{\footnotesize
\tableofcontents
}

\section{Introduction}

We prove the following:

\begin{theorem}[Unbounded Denominators Conjecture] \label{theorem:main}
Let~$N$ be any positive integer, and let
$f(\tau) \in \Z  \llbracket q^{1/N} \rrbracket$ for~$q = \exp(\pi i \tau)$ be a holomorphic
function on the upper half plane~$\H$. Suppose there exists an integer~$k$ and a finite index
 subgroup~$\Gamma \subset \SL_2(\Z)$ such that
$$f \left( \frac{a \tau + b}{c \tau + d}\right) = (c \tau + d)^k f(\tau), \quad \quad
\textrm{for all } \left( \begin{matrix} a & b \\ c & d \end{matrix} \right) \in \Gamma,$$
and suppose that~$f(\tau)$ is meromorphic at the cusps,
that is, locally extends to a meromorphic function near every cusp in the compactification
of~$\H/\Gamma$.
Then~$f(\tau)$ is a modular form for a congruence subgroup of~$\SL_2(\Z)$.
\end{theorem}

The contrapositive of this statement is equivalent to the following, which explains the name of the conjecture: if~$f(\tau) \in \Q \llbracket q^{1/N} \rrbracket$ is a  
modular form which is \emph{not}
modular for some congruence subgroup, then the coefficients of~$f(\tau)$ have \emph{unbounded denominators}.
The corresponding statement remains true if one replaces~$\Q$ by any number field (see Remark~\ref{remark:numberfield}).

Let~$\lambda(\tau)$ be the modular lambda function (Legendre's parameter):
\numequation \label{Legendre}
\frac{\lambda(\tau)}{16} =
\left( \frac{\eta(\tau/2) \eta(2 \tau)^2}{\eta(\tau)^3} \right)^8 =
q \prod_{n=1}^{\infty} \left( \frac{1 + q^{2n}}{1 + q^{2n-1}} \right)^8 = q -8 q^2 + \cdots 
\end{equation}
with~$q = e^{\pi i \tau}$ and~$\eta(\tau/2) = q^{1/24} \prod_{n=1}^{\infty} (1 - q^n)$. (Historic conventions force one to use~$q$ for both~$e^{\pi i \tau}$ and~$e^{2 \pi i \tau}$ --- we use the first
choice unless we expressly state otherwise.)  On replacing the weight~$k$ form $f$ by the 
weight zero form $f(\tau)(\lambda(\tau)/16\eta(\tau/2)^2)^k$,
we may (and do) assume that~$k=0$. The function~$f$ is then an \emph{algebraic} function of~$\lambda$, with branching only at the three punctures $\lambda = 0, 1, \infty$ of the modular curve $Y(2) \cong \mathbb{P}^1 \setminus \{0, 1, \infty\}$. Thus another reading of our result states that the Bely\u{\i} maps (\'etale coverings) 
$$
\pi  : U \to \mathbb{CP}^1 \setminus \{0, 1, \infty\} := \mathrm{Spec} \, \C[\lambda, 1/\lambda, 1/(1-\lambda)]  = Y(2)
$$
possessing a formal Puiseux branch in $\Z\llbracket  \lambda(\tau/m)/16 \rrbracket \otimes \C$ for some $m \in \NwithoutzeroA$ are exactly the congruence  coverings $Y_{\Gamma} = \H / \Gamma \to \H / \Gamma(2) = Y(2)$, with $\Gamma$ ranging over all \emph{congruence} subgroups of $\Gamma(2)$.
The reverse implication is a theorem of Shimura~\cite{ShimuraEichler} (presented in his book as~\cite[Theorem~3.52]{Shimura}), and reflects the fact that the~$q$-expansions of eigenforms
on congruence subgroups are determined by their Hecke eigenvalues (see also~\cite[\SSS~1.2]{Katz349}). 

We refer the reader to Atkin and Swinnerton-Dyer~\cite{ASD} for the roots of the unbounded denominators conjecture, and to Birch's article~\cite{Birch} as well as to Long's survey~\cite[\SSS~5]{Long} for an introduction to this problem and its history. For the vector-valued generalization, see~\SSSS~\ref{Mason vvmf} and its references below.  The cases of  relevance to the partition and correlation functions of rational
conformal field theories (of which the tip of the iceberg is the example~(\ref{cube root of j}) discussed below)
were resolved in a string of works~\cite{DongRen, DongLinNg, SommerhauserZhu, NgSchauenburg, Xu, Bantay, Zhu, AndersonMoore}, by the modular tensor categories method. Some further sporadic cases of the unbounded denominators conjecture have been settled by mostly {\it ad hoc} means~\cite{FioriFranc, FrancMason, LiLong, KL1, KL2}.

To give some simple examples, the integrality property $\sqrt[8]{1-x} \in \Z\llbracket x/16 \rrbracket$ corresponds to the fact that the modular form $(\lambda/16)^{1/8} = q^{1/8} \prod_{n=1}^{\infty} (1 + q^{2n})(1 + q^{2n-1})^{-1}$ and the affine Fermat curve $x^8 + y^8 = 1$ are  congruence; whereas a simple non-example~\cite[\SSS~5.5]{Long} is the affine Fermat curve $x^n + y^n = 1$ for $n \notin \{1, 2, 4, 8\}$, for which the fact that its Fuchsian group is a noncongruence arithmetic group is detected arithmetically by the calculation $\sqrt[n]{1-x} \notin \Z\llbracket x/16 \rrbracket \otimes \C$.  This recovers a classical theorem of Klein~\cite[page~534]{FrickeKlein}.
To include 
an example related to
two-dimensional rational conformal field theories, consider
the  following function  (with~$q=e^{2 \pi i \tau}$):
\numequation \label{cube root of j}
j(\tau)^{1/3} = q^{-1/3} \, \frac{1 + 240 \sum_{n=1}^{\infty} \sigma_3(n) q^n}{\prod_{n=1}^{\infty} (1-q^n)^8} =  q^{-1/3} (1 + 248q + 4124q^2 + 34752q^3 + \cdots). 
\end{equation}
The resulting Fourier coefficients 
are closely linked to the dimensions of the irreducible representations of the  exceptional
Lie group~$E_8(\C)$, and in particular they are integers. To be more precise: the  modular function $j^{1/3}$ coincides with the graded dimension of the
 level one highest-weight representation
of the affine Kac--Moody algebra $E_8^{(1)}$; see Gannon's book~\cite[\SSS~0.5 and~\SSS~3.2.3]{Gannon} 
for a broad view on this topic and
its relation to mathematical physics.
The unbounded denominators conjecture (Theorem~\ref{theorem:main}) now implies that~$j^{1/3}$ 
must be a modular function on a
\emph{congruence} subgroup.   (Strictly speaking,
since~$j^{1/3}$ is a Laurent series rather than a power series, one applies Theorem~\ref{theorem:main}  to~$j^{1/3} \cdot \Delta$
and then divides through by~$\Delta$.)
One readily confirms that~$j^{1/3}$ is a  Hauptmodul for
the level~$3$ subgroup which is the kernel of the composite~$\PSL_2(\Z) \rightarrow \PSL_2(\F_3) = A_4 \rightarrow \Z/3\Z$.
One final example is the function
\numequation
\label{eqn:DeligneExample}
h:=\frac{\lambda(\tau)(1 - \lambda(\tau))}{16} = \left( \frac{\eta(\tau/2) \eta(2 \tau)}{\eta(\tau)^2} \right)^{24}
\end{equation}
of level~$\Gamma^0(2) \supset \Gamma(2)$; 
 here the complete list of $n$ for which $h^{1/n}$ is either congruence modular or has bounded denominators are the divisors of~$24$. 
 The claim that~$h^{1/n}$ has bounded denominators for~$n|24$ is apparent from the product formula in~\autoref{eqn:DeligneExample},
 and the claim that~$h^{1/n}$ does not have bounded denominators for~$n \nmid 24$ is an elementary exercise. 
  We can directly compute when is~$h^{1/n}$ congruence, as follows. 
 By Kummer theory, the extension~$\C(h^{1/n})$ of the function field~$\C(h)$ of~$\H/\Gamma^0(2)$ is Galois with Galois group~$\Z/n\Z$.
 Since~$h$ is nonvanishing
 on~$\H$, this extension is unramified away from the cusps, and so gives rise to a homomorphism~$\Gamma^0(2) \rightarrow \Z/n\Z$;
 the function~$h^{1/n}$ is modular for the kernel~$\Gamma$ of this homomorphism. Now~$h^{1/n}$
 is congruence if and only if the latter map factors through the congruence completion~$\widehat{\Gamma^0(2)}$ of~$\Gamma^0(2)$; here,~$\Gamma^0(2)$  is considered
 as a subgroup of~$\PSL_2(\Z)$.
 In fact, one may   compute that the abelianization of~$\Gamma^0(2)$ is~$\Z \oplus \Z/2\Z$
 whereas the abelianization of~$\widehat{\Gamma^0(2)}$ is~$\Z/24 \Z \oplus \Z/2\Z$. The other~$\Z/2\Z$ extension corresponds to the congruence modular form~$\sqrt{1-64 h} = 1 - 2 \lambda$.
 
In a similar vein pertaining to the examples from the representation theory of vertex operator algebras, we prove in our closing \SSSS~\ref{vvmf} the natural generalization of 
Theorem~\ref{theorem:main} to components of vector-valued modular forms for $\SL_2(\Z)$, in particular resolving --- in a sharper form, in fact --- Mason's unbounded denominators conjecture~\cite{MasonConjectured, KohnenMason} on generalized modular forms.

\subsection{A sketch of the main ideas} \label{intro sketch}  Our proof of Theorem~\ref{theorem:main} follows a broad Diophantine analysis path known in the literature 
(see~\cite{BostGerms, BostAlg} or~\cite[Chapter~10]{BostBook}) as the \emph{arithmetic algebraization} method. 
\subsubsection{The Diophantine principle}
The most basic antecedent of these ideas is the following easy lemma:

\begin{lemma} \label{easylemma}
A power series~$f(x) = \sum_{n=0}^\infty a_n x^n \in \Z\llbracket x \rrbracket$ which  defines a holomorphic function on~$D(0,R)$ for some~$R > 1$ is a polynomial.
\end{lemma}

Lemma~\ref{easylemma} follows  upon combining the following two observations, fixing some~$1>\eta>R^{-1}$:
\begin{enumerate}
\item
The coefficients $a_n$ are either~$0$ or else~$\geq 1$ in magnitude.
\item 
The Cauchy integral formula gives a uniform upper bound $|a_n| = o(\eta^n)$.
\end{enumerate}
We shall refer to the first inequality as a \emph{Liouville lower bound}, following its use by Liouville in his proof of  the lower bound~$|\alpha - p/q| \gg 1/q^n$ 
for algebraic numbers~$\alpha\neq p/q$ of degree~$n\ge 1$.
We shall refer to the second inequality as a \emph{Cauchy upper bound},
following the example above where it comes from an application of the Cauchy integral formula.
The first nontrivial generalization of Lemma~\ref{easylemma} was
 \'Emile Borel's theorem~\cite{Borel}. Dwork famously used a  $p$-adic generalization of Borel's theorem in his $p$-adic analytic proof of the rationality of the zeta function of an algebraic  variety over a finite field (see Dwork's account in the book~\cite[Chapter~2]{DworkGerottoSullivan}). The simplest nontrivial statement of Borel's theorem is 
 that an integral formal power series $f(x) \in \Z\llbracket x \rrbracket$ must  already be a rational function as soon as it has a \emph{meromorphic}  representation as a quotient of two convergent complex-coefficients power series on some disc $D(0,R)$ of a radius $R > 1$.  
 The subject of arithmetic algebraization blossomed at the hands of many authors, including most prominently Carlson, P\'olya, Robinson, Salem, Cantor, D. $\&$ G. Chudnovsky, 
 Bertrandias, Zaharjuta, Andr\'e, Bost, Chambert-Loir~\cite{ACL,BCL,Amice}, \cite[\SSS~I.5]{Andre}, \cite[\SSS~VIII]{AndreG}. A simple milestone that we further develop in our~\SSSS~\ref{holonomy abstraction} is Andr\'e's algebraicity criterion~\cite[Th\'eor\`eme~5.4.3]{Andre}, stating in a particular case that an integral formal power series $f(x) \in \Z\llbracket x \rrbracket$ is algebraic as soon as the two formal functions $x$ and $f$ admit a simultaneous analytic uniformization --- that means an analytic map $\varphi : (D(0,1),0) \to (\C,0)$ such that the 
 composition
 $f(\varphi(z)) \in \C\llbracket z \rrbracket$ of holomorphic function germs
also converges on the full disc $D(0,1)$, and such that~$\varphi$ is sufficiently large in terms of  \emph{conformal size}, namely: $|\varphi'(0)| > 1$.  For example, for any integer~$m$, the algebraic power series $f = (1 - m^2 x)^{1/m} \in \Z\llbracket x \rrbracket$ admits the simultaneous analytic uniformization~$x = \varphi(z) =   (1-e^{Mz}) m^{-2}$ and~$f = f(\varphi(z)) = e^{Mz/m}$, where the conformal size $|\varphi'(0)| = M/m^2$ can clearly be made
arbitrarily large by making a suitable choice of~$M$.

A common theme of all these generalizations of Lemma~\ref{easylemma} is that they   come down to a tension between
a Liouville lower bound and a Cauchy upper bound. For example, in the  proof of Borel's theorem (\cite[Ch~5.3]{Amice}), the Liouville lower bound is applied
not to the coefficients~$a_n$ themselves but rather to Hankel determinants~$\det | \alpha_{i,j}|$ with~$\alpha_{i,j}=a_{i+j+n}$. 
To consider a more complicated example (much closer in both spirit and in details to our own analysis), to prove Andr\'e's algebraicity criterion~\cite[Th\'eor\`eme~5.4.3]{Andre},
one wants to prove that certain powers of a formal function~$f(\mathbf{x}) \in \mathbf{Z}\llbracket \mathbf{x} \rrbracket$ are linearly dependent over the polynomial ring $\Z[\mathbf{x}]$. (It will be advantageous to consider
functions in several complex variables~$\mathbf{x} = (x_1,\ldots,x_d)$.)  The idea is now to consider a certain~$\Z[\mathbf{x}]$ linear combination~$F(\mathbf{x})$ of  powers 
of~$f(\mathbf{x})$ chosen such that they vanish to high order at~$\mathbf{0}$ but yet the~$\Z[\mathbf{x}]$ coefficients~$p(\mathbf{x})$ are  themselves not too complicated ---
the existence of such a choice follows from the classical Siegel's lemma. Now the Liouville lower bound is applied to a lowest order  non-zero coefficient of~$F(\mathbf{x}) \in \mathbf{Z}\llbracket \mathbf{x}\rrbracket$.
Note that such a coefficient must exist or else the equality~$F(\mathbf{x}) = 0$ realizes~$f(\mathbf{x})$ as algebraic. 
The Cauchy upper bound  in this case once again
follows by an application of the Cauchy integral formula.

In our setting, the Liouville lower bound ultimately  comes down to the integrality (``bounded denominators'') hypothesis on the Fourier coefficients of $f(\tau)$, while the 
Cauchy upper bound
comes down to studying the mean growth behavior $m(r,\varphi) := \int_{|z| = r} \log^+{|\varphi|} \, \mu_{\mathrm{Haar}}$ of the largest (universal covering) analytic map $\varphi : D(0,1) \to \C \setminus \mu_N$ avoiding the $N$-th roots of unity. These are clearly distinguished in our abstract arithmetic algebraization work of~\SSSS~\ref{holonomy abstraction} as the steps~\eqref{LLB} and~\eqref{CUP}, respectively. (We also refer to~\eqref{LB} and~\eqref{CB}, resp.~(\ref{Liouville}) and~(\ref{the upper bound}), in our alternative treatments.)
Our Theorem~\ref{abstract form} is effectively a quantitative refinement of Andr\'e's algebraicity criterion to take into account the degree of algebraicity over $\Q(x)$, and still more precisely a certain holonomy rank over $\Q(x)$. Foreshadowing a key technical point (to be discussed in more detail later in the introduction), our Cauchy upper bound is given in terms
of a mean (integrated) growth term rather than a supremum term, and  this  improvement is essential to our approach.

\subsubsection{Modularity and simultaneous uniformizations of $f$ and $\lambda$}
Let us now explain the relevance of arithmetic holonomy rank bounds to the unbounded denominators conjecture.
After reducing to weight $k = 0$ as above, the functions $f = f(\tau)$ and $x := \lambda(\tau) / 16 \in q + q^2 \Z\llbracket q\rrbracket$ are algebraically dependent and share both (we assume) the property of integral Fourier coefficients at the  cusp $q = 0$. Let us assume for the purpose of this sketch that $f(\tau) \in \Z\llbracket q \rrbracket$ with~$q = e^{\pi i \tau}$, i.e. that the cusp $i\infty$ has width dividing~$2$. Then the formal inverse series expansion 
$$
q 
= x + 8x^2 + 84x^3 + 992 x^4 + \cdots \in x
+ x^2 \Z\llbracket x\rrbracket
$$
	 of (\ref{Legendre}) has integer coefficients, expressing the identity $\Z\llbracket q \rrbracket = \Z\llbracket x \rrbracket$ of formal power series rings, and that formal substitution turns our integral Fourier coefficients hypothesis into an \emph{algebraic} power series with integer coefficients: henceforth in this introductory sketch we switch to writing, by a mild and harmless notational abuse, simply $f(x) \in \Z\llbracket x \rrbracket$ in place of $f(\tau)$ and $\lambda(q)$ in place of~$\lambda(\tau)$. In the general case of arbitrary cusp width, which we need anyhow for the inner workings of our proof even if one is ultimately interested in the $\Z\llbracket q \rrbracket$ case, we will only have $f \in \Z[1/N]\llbracket x \rrbracket$ when we write out $f(\tau)$ as a power series in $x := \sqrt[N]{\lambda/16}$ to accommodate the Puiseux series --- but there is still a hidden integrality property which we can exploit. That leads to some mild technical nuance with the power series~(\ref{normal}) --- think of $t = q^{1/N}$,
$x(t) = \sqrt[N]{\lambda(t^N)/16}$ and $p(x) = x^N$ --- in our refinement~(\ref{abstracted bound}) of Andr\'e's theorem. 

 The complex analysis enters by way of a linear ODE  in the following way. To start with, we have, just {\it by fiat}, the simultaneous analytic uniformization of the two functions $x := \lambda/16$ and $f$ by the complex unit $q$-disc $|q| < 1$.  In this way, 
the tautological  choice $\varphi(z):= \lambda(z)/ 16$ turns our algebraic power series $f(x) \in \Z\llbracket x \rrbracket$
into a boundary case (unit conformal size $\varphi'(0) = 1$) of Andr\'e's criterion. Another boundary case, but this time transcendental and incidentally demonstrating the sharpness of the qualitative Andr\'e algebraicity criterion even in the {\it a priori} holonomic situation (see~\cite[Appendix, A.5]{Andre} for a discussion), 
is provided by the Gauss hypergeometric function
\numequation
\label{whatisFone}
F(x) := \pFq{2}{1}{1/2,1/2}{1}{16x} = \sum_{n=0}^{\infty} \binom{2n}{n}^2 x^n \in \Z\llbracket x \rrbracket,
\end{equation}
whose unit-radius simultaneous analytic uniformization with $x = \lambda /16$  is given again by the analytic $q$ 
coordinate, and the classical Jacobi formula 
\numequation
\label{whatisFtwo}
F(\tau) =  \pFq{2}{1}{1/2,1/2}{1}{\lambda(q)} = \Big( \sum_{n \in \Z} q^{n^2} \Big)^2
\end{equation}
which transforms this hypergeometric series into
a weight one modular form for the congruence group~$\Gamma(2)$.  The existence of such transcendental $\Z\llbracket x \rrbracket$ holonomic functions on $\C \setminus \{0,1/16\}$ recovers---by Andr\'e's algebraicity criterion---a classical ``$1/16$ theorem'' of Carath\'eodory~\cite[(412.8) on page~198]{Caratheodory}. (See also Goluzin~\cite[\SSS~III.1, Theorem~1]{Goluzin}.)

\subsubsection{A finite local monodromy leads to an overconvergence}\label{sec_overconv}

It turns out, and this is the key to our method and already answers Andr\'e's question in~\cite[Appendix, A.5]{Andre}, 
that a different choice of $\varphi(z)$ allows one to arithmetically distinguish between these two cases (algebraic and transcendental), 
and to have the algebraicity of $f(x)$ recognized by Andr\'e's Diophantine criterion by way of an ``overconvergence.''
Suppose that~$f(\tau)$  is a holomorphic modular function and~$F(\tau)$ is a holomorphic modular form ---
concretely, let us take~$F(\tau)$ to be the theta series of  equation~(\ref{whatisFtwo}), ---
and let~$f(x)$ and~$F(x)$ respectively denote these functions as functions of~$x = \lambda/16$, so~$F(x)$
is given by equation~(\ref{whatisFone}).
The common feature of these two functions $f(x)$ and $F(x)$  --- coming respectively out of modular forms
of weights~$0$ and~$1$ --- is that  they both vary 
holonomically in $x \in \C \setminus \{ 0, 1/16 \}$: they satisfy linear ODEs with coefficients in $\Q[x]$ and no singularities\footnote{With nontrivial local monodromy. The precise definition is in~\ref{holonomy ring}.}
apart from the three punctures $x = 0, 1/16, \infty$ of $Y(2) = \H / \Gamma(2)$. The difference feature is that their respective local monodromies around $x = 0$ are finite for the case of $f(x)$ (a quotient of $\Z/N$, with the order $N$ equal to the lowest common multiple of the cusp widths, or Wohlfahrt level~\cite{Wohlfahrt} of $f(x)$); and infinite for the case of $F(x)$ (isomorphic to $\Z$, corresponding  more  particularly to the fact that this particular hypergeometric function acquires a $\log{x}$ term after an analytic continuation around a small circle enclosing $x = 1/16$). 
If now we perform the variable change $x \mapsto x^N$, redefaulting to $x := \sqrt[N]{\lambda(q^N)/16}$, that resolves the $N$-th root ambiguity 
in the formal Puiseux branches of $f(x)$ at $x = 0$, and the resulting algebraic power series $f(x^N) \in \Z\llbracket x^N \rrbracket \subset \Z\llbracket x \rrbracket$ has turned holonomic on $\mathbb{P}^1 \setminus \{16^{-1/N}\mu_N, \infty\}$: singularities only at $16^{-1/N} \mu_N \cup \{\infty\}$ (but not at $x = 0$: this key step of exploiting arithmetic algebraization is the same as in Ihara's arithmetic connectedness theorem~\cite[Theorem~1]{Ihara}, which together with Bost's extension~\cite{Bost} to arithmetic Lefschetz theorems have in equal measure been inspirational for our whole approach to the unbounded denominators conjecture).  Since $\lambda : D(0,1) \to \C \setminus \{1\}$ has 
fiber $\lambda^{-1}(0) = \{0\}$, the function $\varphi(z):= \sqrt[N]{\lambda(z^N)/16} : D(0,1) \to  \C \setminus 16^{-1/N}\mu_N$ is still 
holomorphic on the unit disc $|z| < 1$, and  under this tautological choice, both functions $f(x^N)$ and $F(x^N)$ continue to be at the borderline of Andr\'e's algebraicity criterion: $|\varphi'(0)| = 1$. 

But if instead of the tautological simultaneous uniformization we take 
$$\varphi : D(0,1) \to \C \setminus 16^{-1/N} \mu_N$$ to be the universal covering map (pointed at $\varphi(0) = 0$), then  either by a direct computation with monodromy, or by Cauchy's analyticity theorem on the solutions of linear ODEs with analytic coefficients and no singularities in a disc, we have both function germs $x := \varphi(z)$ and $f(x) := f(\varphi(z))$ holomorphic, hence convergent, on the full unit disc $D(0,1)$. In contrast, now $F(\varphi(z))$ converges only up to the ``first'' nonzero fiber point in $\varphi^{-1}(0) \setminus \{0\}$, giving a certain radius rather smaller than $1$. We must have the strict lower bound $|\varphi'(0)| > 1$, because the preceding unit-radius holomorphic map $\sqrt[N]{\lambda(z^N)/16} : D(0,1) \to \C \setminus 16^{-1/N}\mu_N$ has to factorize \emph{properly} through the universal covering map. Indeed in Theorem~\ref{exactconformal}, using an explicit description by hypergeometric functions of the multivalued inverse of the universal covering map of $\C \setminus \mu_N$ based on Poincar\'e's ODE approach~\cite{Hempel} to the uniformization of Riemann surfaces, we find an exact formula for this uniformization radius in terms of the Euler Gamma function.\footnote{Andr\'e pointed out to us that this explicit formula has  previously been obtained by Kraus and Roth, see~\cite[Remark~5.1]{KrausRoth}. See also~\cite[\SSS~III.1]{Goluzin}.} 
Hence the algebraicity of $f(x)$ gets witnessed by Andr\'e's criterion; and the formal new result that we get already at this opening stage (see Theorem~\ref{ODE form}) is that \emph{any} integral formal power series solution $f(x) \in \Z\llbracket x \rrbracket$ to a linear ODE $L(f) = 0$ without singularities on $\mathbb{P}^1 \setminus \{0,1/16,\infty\}$ is in fact algebraic as soon as the linear differential operator $L$ has a finite local monodromy $\Z/N$ around the singular point $x = 0$. More than this: the quantitative Corollary~\ref{holonomy}  proves that the totality of such $f(x) \in \Z\llbracket x \rrbracket$ at a given $N$ span a finite-dimensional $\Q(x)$-vector space, and gives an upper bound on its dimension as a function of the Wohlfahrt level parameter $N$. Now since a (noncongruence) counterexample $f(\tau) \in \Z\llbracket q \rrbracket$ to Theorem~\ref{theorem:main} would not 
exist on its own
 but spawn a whole sequence $f(p \tau) \in \Z\llbracket q \rrbracket$ of  $\Q(x)$-linearly independent counterexamples at growing Wohlfahrt level~$N \mapsto Np$, 
our idea is to measure up the supply of these putative (fictional) counterexamples alongside the congruence supply at a gradually increasing level until together they break the quantitative bound~(\ref{abstracted bound}) supplied by our arithmetic holonomy Theorem~\ref{abstract form}. 

\subsubsection{The dimension bound can be leveraged with growing level~$N$}
 We have the congruence supply of dimension $[\Gamma(2) : \Gamma(2N)] \gg N^3$, and then 
 as a glance at our shape~(\ref{dimensionbound}) of holonomy rank bound readily reveals, it seems a fortuitous piece 
of luck that the conformal size (Riemann uniformization radius at $0$) of our relevant Riemann surface $\C \setminus 16^{-1/N} \mu_N$ turns out to have the matching asymptotic form $1 + \zeta(3) / (2N^3) + O(N^{-5})$.  
We ``only'' have to prove that the numerator (growth) term in the holonomy rank bound~(\ref{dimensionbound}) inflates at a slower rate than our extrapolating putative counterexamples $f(\tau) \mapsto f(p \tau)$! 

The meaning of the requisite inflation rate is clarified in~\SSSS~\ref{serresection}, with Proposition~\ref{selfcontained}
and Remark~\ref{expansion rate}. It turns out that the logarithmically inflated holonomy rank (dimension) bound 
by $O(N^3\log{N})$ is sufficient for the desired proof by contradiction (but an $O(N^{3 + 1 / \log{\log{N}}})$ or worse form of bound would not suffice); and this is what we ultimately prove. Getting to this degree of precision
creates however some additional challenges. A straightforward elaboration of Andr\'e's original argument in~\cite[Criterium VIII~1.6]{AndreG}, taking the number of variables $d \to \infty$  and  involving the $\sup_{|z| = r}\log{ |\varphi|}$ growth term of  {\it loc.cit.} in place of our mean (integrated) growth term $m(r,\varphi)$ (see~\SSSS~\ref{Nevone} and~\SSSS~\ref{integratedgrowthbound}), leads quite easily to an $O(N^5)$ dimension bound; and 
 by a further work explicitly with the cusps of the Fuchsian uniformization $D(0,1) / \Gamma_N \cong \C \setminus \mu_N$, and an appropriate Riemann map precomposition, it is possible to further reduce that down to an $O(N^4)$. See Remark~\ref{Renggli note}. This does not suffice to conclude the proof. Going further requires an intrinsic improvement into Andr\'e's dimension bound itself: the reduction of the supremum term to the integrated term in the numerator of~(\ref{dimensionbound}). 

We  give three proofs of this improvement, all being based on the same auxiliary construction scheme of~\SSSS~\ref{auxiliary construction}.  Our default treatment~\SSSSs~\ref{auxiliary construction},  \ref{extrapolate}, \ref{canonical decomposition} is   based on Nevanlinna's canonical factorization of meromorphic functions of bounded characteristic. Additionally, we also include in~\SSSS~\ref{initial argument} our original argument  based on equidistribution ideas, and a simplified alternative path~\SSSS~\ref{psh approach} proposed to us by Andr\'e and based on plurisubharmonicity and a lexicographic induction. 
The former variation has  a 
potential-theoretic flavor familiar from the proof of Bilu's  theorem~\cite{Bilu} (see also~\SSSS~\ref{potential sketch}), but it is 
 in the cross-variables $d \to \infty$ asymptotic aspect and hence different than the well-established link (see~\cite{Bost,BCL,BostGerms}) of arithmetic algebraization to adelic potential theory.

\subsubsection{Nevanlinna theory for Fuchsian groups}  \label{Nevone}
Everything is thus reduced to establishing a uniform integrated growth bound of the form
\numequation \label{growth key}
m(r,F_N^N) := \int_{|z| = r} \log^+{|F_N^N|} \, \mu_{\mathrm{Haar}} = O\Big(\log{\frac{N}{1-r}}\Big),
\end{equation}
 where $N \geq 2$ and $F_N : D(0,1) \to \C \setminus \mu_N$ is the universal covering map based at $F_N(0) = 0$. 
Heuristically this is supported by the idea that the renormalized function
$F_N(q^{1/N})^N$ ``converges'' in some sense to the modular lambda function $\lambda(q)$, as $N \to \infty$. 
These functions do indeed converge as~$q$-expansions as~$N \rightarrow \infty$ on any ball around the origin of radius strictly
less than~$1$.
The problem is that this convergence is not in any way uniform as $r \to 1$, but we need to use~(\ref{growth key}) with
a radius as large as $r = 1 - 1/(2N^3)$. The growth of the map $F_N$ is governed by the growth of the cusps of the $(N, \infty, \infty)$ triangle  Fuchsian group $\Gamma_N$, and studying these directly, for instance by comparing them to the cusps of the limit $(\infty, \infty, \infty)$ triangle group $\Gamma(2)$, proves to be difficult. 

Surprisingly perhaps, we are instead able in~\SSSS~\ref{nevanlinna} to prove the requisite mean growth bound~(\ref{growth key}) on the 
abstract grounds of Nevanlinna's value distribution theory for general meromorphic functions.  For any universal covering map $F : D(0,1) \to \C \setminus \{a_1, \ldots, a_N\}$ of a sphere with  $N+1 \geq 3$ punctures, one has the mean growth asymptotic $m(r, F) = \int_{|z| = r} \log^+{|F|} \, \mathrm{\mu}_{\mathrm{Haar}} \sim \frac{1}{N-1} \log{\frac{1}{1-r}}$ under $r \to 1^-$, 
providing extremal examples of Nevanlinna's defect inequality with $N+1$ full deficiencies on the disc~\cite[page~272]{Nevanlinna}. Contrast this with the
qualitatively exponentially larger growth behavior $\sup_{|z| = r} \log{|F|} \asymp \frac{1}{1-r}$ of the crude supremum term.  
In our particular situation of $\{a_1, \ldots, a_N\} = \mu_N$ for the puncture points, we are able to exploit 
the fortuitous relation $ \prod_{i=1}^N (x-a_i) \sum_{i=1}^N \frac{1}{x-a_i} = Nx^{N-1}$ particular
to the partial fractions decomposition~\eqref{choice of p} to get to the uniformity precision of~(\ref{growth key}) 
with the method of the logarithmic derivative in Theorem~\ref{growth term}.

\begin{remark}[Big $O$ and small $o$ notation, $\NwithzeroA$ and $\NwithoutzeroA$ ] \label{rem:notation}
We use big $O$ and small $o$ notation throughout in their usual way.
We also use Vinogradov's $\ll$ notation which is completely synonymous with the big $O$ notation, that is, $f \ll g$ has
the same meaning as $f = O(g)$. Both of these notations mean that, with respect to some implicit variables, 
the inequality~$f \le Cg$ holds for all values of these variables sufficiently close to some implicit limit.
We call (any suitable choice of) $C$ the \emph{implicit constant}, and whenever we want to stress what either the implicit
variables or implicit limits are in the notation, these are included as subscripts on either~$o$, $O$, or~$\ll$.
We shall use~$\NwithzeroB = \{0,1,2, \ldots\}$ to denote the natural numbers with zero,
and~$\NwithoutzeroA$  to denote the positive integers.
\end{remark}

\section{The arithmetic holonomicity theorem} \label{holonomy abstraction}

Our proof relies on the following dimension bound which is an extension of
Andr\'e's arithmetic algebraicity criterion~\cite[Th\'eor\`eme~5.4.3]{Andre}. We state and prove our result here in a
particular case suited to our needs, beginning with the abstract form. We denote by $\mathcal{O}(\overline{D(0,1})) \subset \C \llbracket z \rrbracket$ the ring of holomorphic function germs that converge on some open neighborhood of the closed unit disc $|z| \leq 1$. Throughout our paper, we will use the notation
$$
\T  := \{ e^{2 \pi i \theta} \, : \, \theta \in [0,1) \} \subset \C^{\times}
$$
for the unit circle,  the Cartesian power
$$
\T^d := \{ (e^{2\pi i \theta_1}, \ldots, e^{2\pi i \theta_d}) \, : \, \theta_1, \ldots, \theta_d \in [0,1) \} \subset \mathbb{G}_m^d(\C)
$$
for the unit  $d$-torus, and
$$
\mu_{\mathrm{Haar}} := d\theta_1 \cdots d\theta_d
$$
for the normalized Haar measure of this compact group.

\begin{thm} \label{abstract form}
Consider the following data: 
\begin{itemize} 
\item[(i)] 
a nonconstant rational function $p(x) \in \Q(x) \setminus \Q$ without pole at $x=0$, 
\item[(ii)] a formal power series 
\numequation \label{normal}
x(t) \in t +  t^2 \Q\llbracket t \rrbracket
\end{equation}
 pulling back $p$ into an integral coefficients power series $x^* p := p(x(t)) \in \Z \llbracket t \rrbracket$  in the new variable $t$,
\item[(iii)]  and a holomorphic mapping $\varphi:  \overline{D(0,1)} \to \C$ taking $\varphi(0) = 0$ with $|\varphi'(0)| > 1$, and pulling back $p$ into a holomorphic function $\varphi^* p \in \mathcal{O}(\overline{D(0,1)})$  on some  neighborhood of the closed unit disc.
\end{itemize}
Suppose the formal power series $f_1, \ldots, f_m \in \Q\llbracket x \rrbracket$ are $\Q(p(x))$-linearly independent and satisfy the following integrality and analyticity properties like in~(ii) and~(iii):
$$
x^*f_1, \ldots, x^*f_m \in \Z \llbracket t \rrbracket, \quad \textrm{and} 
\quad \varphi^* f_1, \ldots, \varphi^* f_m \in \mathcal{O}(\overline{D(0,1)}).
$$
Then $f_1, \ldots, f_m \in \Q \llbracket x \rrbracket$ are algebraic (i.e., all $f_i \in \overline{\Q(x)}$), and
\numequation \label{abstracted bound}
 m  \leq  e  \cdot \frac{  \int_{\T} \log^+{|p \circ \varphi|} \, \mu_{\mathrm{Haar}}}{\log{|\varphi'(0)|}},
  \end{equation}
where $e = 2.718\ldots$ is Euler's constant. 
\end{thm}

The novel point of the bound~\eqref{abstracted bound} is the integrated term in the numerator instead of a supremum term.  It is critical for our proof of the unbounded denominators conjecture to have the numerator in~\eqref{abstracted bound}, which measures the growth of $\varphi$, expressed as a Nevanlinna characteristic function (or, equivalently in the holomorphic case that we consider, a mean proximity function).  

This abstract dimension bound~\eqref{abstracted bound} will be used more concretely as a holonomy rank bound. 
To state the relevant corollary, let  us introduce an algebra of holonomic power series with integral coefficients
and restricted singularities.

\begin{df} \label{holonomy ring}
For $U \subset \C$ an open subset, $R \subset \C$ a subring with fraction field $F := \mathrm{Frac}(R)$, and $x(t) \in t \Q\llbracket t \rrbracket$ a
formal power series,
 we define $\mathcal{H}(U,x(t),R)$ to be the ring  of  formal
power series $f(x) \in F\llbracket x \rrbracket$ whose $t$-expansion $f(x(t)) \in R\llbracket t \rrbracket$, and such that there exists a nonzero 
linear differential operator $L$ over $\overline{\Q}(x)$ with $L(f) = 0$ and having a trivial local monodromy around all of its singular  points that belong to $U$. 
%\medskip

Further, we let $\mathcal{V}(U, x(t), R)$ to be the $F(x)$-vector space spanned by $\mathcal{H}(U,x(t),R)$. 

%\medskip

For $x(t) = t$, we more simply denote the $R[x]$-algebra $\mathcal{H}(U, t, R)$  by $\mathcal{H}(U,R)$ and the $F(x)$-vector space $\mathcal{V}(U,t,R)$ by $\mathcal{V}(U,R)$.
\end{df}

Here by \emph{trivial local monodromy around $x  = \alpha$} we mean that there exist a complex  neighborhood $U_{\alpha} \ni \alpha$ and meromorphic functions $g_1,\dots, g_n \in \mathcal{M}(U_{\alpha})$ on $U_{\alpha}$, where $n$ is the order of $L$,  such that $g_1,\dots, g_n$ form a $\C$-basis of the solution space of $L(f)=0$ on $U_\alpha \setminus \{\alpha\}$. This is the case if $x = \alpha$ is not a singular point of $L$.
An example at a singular point $x = 0$ include $L_n = x \frac{d}{dx} - n$ for $n \in \Z \setminus \{0\}$, of solution space $\ker{L_n} = \C \cdot x^n$; this is meromorphic (but not holomorphic) when~$n < 0$.  

% there are two kinds of typical examples. Firstly, the meromorphic function $1/x$ has   $L_0 := x \frac{d}{dx} + 1$ for its minimal differential operator,  and every homogeneous $L$ fulfilling $L(1/x) = 0$ has to factorize through $L_0$ and is hence necessarily singular at $x = 0$. And secondly, the holomorphic function $x$ lies in the kernels of both operators $L_1 := x \frac{d}{dx} - 1$ and $L_2 := (d/dx)^2$, and hence it has no minimal linear homogeneous ODE --- but technically the equation $L_1(f) = 0$ is singular at $x = 0$ while having the one-dimensional solution space $\C \cdot x$.   

Our holonomy bound is now a straightforward combination of Theorem~\ref{abstract form} and Cauchy's analyticity theorem on the solutions of linear differential equations with analytic coefficients. 

\begin{cor} \label{holonomy} 
Let $0 \in U \subset \C$ be an open subset containing the origin. If the uniformization radius of the pointed
Riemann surface $(U,0)$ is strictly greater than $1$, then the algebra $\mathcal{V}(U,\Z)$ is finite-dimensional as a $\Q(x)$-vector space. 

More precisely, let $p(x) \in \Q(x) \setminus \Q$ be a non-constant  rational function without poles in $U$, 
and let
$\varphi(z) : \overline{D(0,1)} \to U$ be  a holomorphic map taking $\varphi(0) = 0$ with $|\varphi'(0)| > 1$. 
 If
\numequation 
x(t) \in t +  t^2 \Q\llbracket t \rrbracket
\end{equation}
has $p(x(t)) \in \Z\llbracket t \rrbracket$, 
then
 the following dimension bound holds on $\mathcal{V}(U,x(t),\Z)$ over
$\Q(p(x))$:
  \numequation \label{dimensionbound}
\dim_{\Q(p(x))} \mathcal{V}(U,x(t),\Z)   \leq e  \cdot \frac{  \int_{\T} \log^+{|p \circ \varphi|} \, \mu_{\mathrm{Haar}}}{\log{|\varphi'(0)|}}.
  \end{equation}

\end{cor}

\begin{proof}
 The pulled-back space $\varphi^*\mathcal{H}(U,  x(t) , \Z) \subset \varphi^* \mathcal{H}(U,\C)$
 lies in the ring of formal
power series fulfilling linear differential equations with analytic coefficients and no singularities with nontrivial local monodromies on the closed
disc $\overline{D(0,1)}$. Hence, for any such function $f \in \mathcal{H}(U,  x(t) , \Z)$, there exists a nonzero $g(x)\in \Q[p(x)] \setminus \{0\}$ such that for any singular point $\alpha\in U$ of the linear operator $L$ in Definition~\ref{holonomy ring}, and for any local solution $h(x)$ of~$L(h) = 0$ in a small punctured neighborhood of $\alpha$, the product function $g(x)h(x)$ is holomorphic at $x = \alpha$. 
(The singularities of~$L$ all occur at algebraic points.) 
Cauchy's theorem then gives that $\varphi^*(gf)$ is a holomorphic function on $\overline{D(0,1)}$, and we conclude by Theorem~\ref{abstract form}.
\end{proof}

Theorem~\ref{abstract form} is modeled on Andr\'e's Diophantine approximation method~\cite[\SSS~VIII]{AndreG}, \cite[\SSS~5]{Andre}. We include as many as three proofs, all sharing a common basic framework~\SSSS~\ref{auxiliary construction} and relying crucially on a $d \to \infty$ limit for the number of auxiliary variables in the auxiliary function constructed by Lemma~\ref{Siegel} below. Our original treatment was based on equidistribution and is in~\SSSSs~\ref{auxiliary construction}, \ref{initial argument},  and an alternative approach proposed to us by Andr\'e and based on plurisubharmonicity is in~\SSSSs~\ref{auxiliary construction}, \ref{psh approach}.
 Firstly we give a shorter proof based on  Nevanlinna's canonical factorization~\SSSS~\ref{canonical decomposition} and the following intermediate form of Theorem~\ref{abstract form}. 

\begin{lemma}  \label{meromorphic form}
In the setting of Theorem~\ref{abstract form}, consider furthermore an arbitrary holomorphic function $h : \overline{D(0,1)} \to \C$ with $h(0) = 1$. Then
\numequation  \label{meromorphic bound}
m \leq e \, \frac{ \max \big\{  \sup_{\T}{\log{|h|}}, \, \sup_{\T}{\log{|h \cdot \varphi^* p|}}  \big\}  }{\log{|\varphi'(0)|}},
\end{equation}
and $f_1, \ldots, f_m \in \overline{\Q(x)} \cap \Q\llbracket x \rrbracket$. 
\end{lemma}

\begin{remark}
We will find in~\SSSS~\ref{canonical decomposition} that the bound in Theorem~\ref{abstract form} is equal to the infimum of the bounds in Lemma~\ref{meromorphic form} across all choices of the \emph{holomorphic multiplier} function~$h$. Therefore, in this form, Lemma~\ref{meromorphic form} 
is in fact equivalent to our main Theorem~\ref{abstract form}; but it turns out convenient to approach the statement in this intermediate form. 
On the other hand, 
Remark~\ref{large numbers} sketches a strengthened form of the lemma. 
\end{remark}

 For a complete proof of the unbounded denominators conjecture, we invite the reader on a first pass to proceed directly to~\SSS~\ref{main plan} after~\SSS~\ref{canonical decomposition}.

\subsection{The auxiliary construction} \label{auxiliary construction}

We will make a use of a Diophantine approximation construction in a high number $d \to \infty$ of variables $\mathbf{x} := (x_1, \ldots, x_d)$. We will write 
$$
\mathbf{x^j} := x_1^{j_1} \cdots x_d^{j_d}, \quad   p(\mathbf{x}) := (p(x_1), \ldots, p(x_d)). 
$$

Since $\varphi$ maps $(D(0,1),0)$ to~$(\C,0)$ with nonzero derivative, the inverse function theorem gives a positive radius $\rho > 0$ such that
\numequation  \label{inverse function theorem}
\varphi \, : \, \varphi^{-1}(D(0,\rho))_0 \xrightarrow{\cong} D(0,\rho)
\end{equation}
is an analytic isomorphism from the connected component $\varphi^{-1}(D(0,\rho))_0$
of~$\varphi^{-1}(D(0,\rho))$  which contains the element~$0$.

\begin{lemma} \label{Siegel}
 Let $d, \alpha \in \NwithoutzeroA$ and $\kappa \in (0,1)$ be parameters.
Asymptotically
  in $\alpha \to \infty$ as $d$ and $\kappa$ are held fixed, there exists a \emph{nonzero} $d$-variate formal function $F(\mathbf{x})$ of the form
\numequation 
 \label{theform}  F(\mathbf{x}) =     \sum_{\substack{  \mathbf{i} \in \{1, \ldots, m\}^d   \\ \mathbf{k} \in \{0, \ldots, D-1\}^{d} }}  a_{\mathbf{i,k}}\, p( \mathbf{x})^{\mathbf{k}}  \, \prod_{s=1}^d f_{i_s}(x_s) 
\in \Q\llbracket \mathbf{x} \rrbracket \setminus \{0\},
\end{equation}
vanishing to order at least $\alpha$ at $\mathbf{x = 0}$,
   with
  \begin{enumerate}
    \item  \label{degrees} 
$$
  D \leq   \frac{1}{(d!)^{1/d}} \frac{1}{m} \Big( 1+\frac{1}{\kappa} \Big)^{\frac{1}{d}} \alpha + o(\alpha);
$$ 
     \item \label{heights} all $a_{\mathbf{i,k}} \in \Z$ are integers 
 bounded in absolute value by $\exp\big( \kappa C\alpha +   o(\alpha) \big)$
     for some constant $C \in \R$ depending only on %$\max_{u \notin U}{1/|u|}$ -- this was in the holonomic version
the radius $\rho$ from~\eqref{inverse function theorem}     and on the degree and height of the rational function $p(x) \in \Q(x)$.
\end{enumerate}
\end{lemma}

\begin{proof}
We expand our sought-for formal function in (\ref{theform}) into a formal power series in $\Q\llbracket \mathbf{x} \rrbracket$ and solve $\binom{\alpha + d}{d} \sim \alpha^{d}/ d!$ linear equations in the $(mD)^d$ free parameters $  a_{\mathbf{i,j}}$. To begin with, we show that in the formal inverse function expansion, the integrality condition $p(x(t)) \in \Z\llbracket t \rrbracket$ entails $x(t) \in t + (t^2/M) \Z\llbracket t/M \rrbracket$ with some $M \in \NwithoutzeroA$ bounded in terms of the degree and height of the rational function $p(x)$.  
 
Here are the details on the construction of~$M \in \NwithoutzeroA$. Set $y := b(p(x) - p(0))/c$ with~$b \in \Z \setminus \{0\}$ and~$c \in \NwithoutzeroA$ chosen so that $y \in \left( x^N+ x^{N+1} \Q\llbracket x \rrbracket \right) \cap \Q(x)$ for some~$N \in \NwithoutzeroA$. Formally, we have a Puiseux series branch expansion $x = x(y) \in \zeta y^{1/N} + y^{1/N} \Qbar\llbracket y^{1/N} \rrbracket$, where~$\zeta^N = 1$. Eisenstein's theorem~\cite[\S~11.4]{BombieriGubler} supplies an $M_1 \in \NwithoutzeroA$ (depending on $p(x)$) for which $x(y) \in \overline{\Z}\llbracket y^{1/N}/M_1 \rrbracket$. On the other hand, the binomial expansion gives $(1+u)^{1/N} = \sum_{n=0}^{\infty} \binom{1/N}{n} u^n \in \Z\llbracket u/N^2\rrbracket$, by a simple denominator estimate. With our assumptions~$p(x(t)) \in \Z\llbracket t \rrbracket$ and  $x(t) \in t + t^2 \Q\llbracket t \rrbracket$ implying $y(t) =b (p(x(t))-p(0))/c \in c^{-1}\Z\llbracket t \rrbracket \cap \left( x^N + x^{N+1} \Q\llbracket x \rrbracket \right) = c^{-1}\Z\llbracket t \rrbracket \cap \left( t^N + t^{N+1} \Q\llbracket t \rrbracket \right) = t^N + c^{-1} t^{N+1} \Z\llbracket t \rrbracket$ and hence $\zeta y(t)^{1/N} =t \left( 1 + b_1t/c + b_2t^2/c +  b_3t^3/c + \cdots \right)^{1/N}$ with~$\zeta^N = 1$ and
 some integers~$b_1, b_2, \ldots \in \Z$, the binomial expansion gives~$\zeta y(t)^{1/N} \in \Z \llbracket t/(cN^2)\rrbracket$, and therefore 
 $x(t) = x(y(t)) \in \overline{\Z}\llbracket  y(t)^{1/N}/M_1 \rrbracket
\subseteq  \overline{\Z}\llbracket  t/(cN^2M_1) \rrbracket$. Coupled with~$x(t) \in t + t^2 \Q\llbracket t \rrbracket$, this supplies the requisite formula
$x(t) \in t + (t^2/M) \Z\llbracket t/M \rrbracket$ with $M := c^2 N^4 M_1^2$.
 
Now the inverse series also has $t(x) \in x + (x^2/M) \Z \llbracket x/M \rrbracket$, and so $f_i(x(t)) \in \Z \llbracket t \rrbracket$ entails $f_i(x) \in \Z \llbracket x/M \rrbracket$ for all $i = 1,\ldots, m$. Furthermore, by~\eqref{inverse function theorem}, every power series $f_i(x) \in \Q\llbracket x \rrbracket$ is convergent on the archimedean disc $|x| < \rho$. The result then follows from the classical Siegel lemma~\cite[Lemma~2.9.1]{BombieriGubler}, with $e^C := M/\rho$ and 
the degree parameter choice
$$
D  \sim  \frac{1}{m(d!)^{1/d}} \Big( 1+ \frac{1}{\kappa} \Big)^{\frac{1}{d}} \alpha,
$$
that brings in a Dirichlet exponent $\sim \kappa$ as $\alpha \to \infty$. 

Since the formal functions $f_1, \ldots, f_m \in \Q\llbracket x\rrbracket$ are linearly independent over $\Q(p(x))$, an easy induction argument on the dimension $d$ shows that $\{f_\bi\}_{\bi \in \{1, \ldots, m\}^d}$ are linearly independent over $\Q(p(\bx))$. For the step of this induction, simply note that a non-zero element $Q(x_1,\ldots,x_{d+1}) \in \Q(p(x_1),\ldots,p(x_{d+1})) \setminus \{0\}$ specializes to a non-zero element~$Q(\bx,c) \in \Q(p(\bx)) \setminus \{0\}$ for all but finitely many arguments~$c \in \Q$ under setting~$x_{d+1} := c$, and so a putative relation in the~$d+1$ variables~$(\bx, x_{d+1})$ specializes to a relation in the~$d$ variables~$\bx = (x_1,\ldots,x_d)$. 

At this point, having established the $\Q(p(\bx))$-linear independence of the constituent functions~$f_{\bi}$, the property $F \not\equiv 0$ follows since at least one $a_{\mathbf{i,j}} \neq 0$ in the form~\eqref{theform}. 
\end{proof}

\subsection{Extrapolation and proof of Lemma~\ref{meromorphic form}} \label{extrapolate}
We consider the nonzero formal function
\numequation \label{cleared up}
H(\mathbf{z}) := h(z_1)^D \cdots h(z_d)^D \cdot F(\varphi(z_1), \ldots, \varphi(z_d)) \in  \C \llbracket \mathbf{z} 
\rrbracket \setminus \{0\}.
\end{equation}
By construction, it vanishes at $\mathbf{z=0}$ to order at least $\alpha$, and it is holomorphic in a neighborhood of the closed unit polydisc because all the split-variables constituents 
$$
\varphi^*p; \quad  \varphi^* f_1, \ldots, \varphi^* f_m \in {\mathcal{O}(\overline{D(0,1)})}.
$$
Let $\beta \geq \alpha$ be the exact order of vanishing of $F(\mathbf{x}) \in \Q \llbracket \mathbf{x} \rrbracket \setminus \{0\}$ at $\mathbf{x=0}$, and consider $c \, \mathbf{x^n}$ any nonzero monomial of that lowest order $\beta = |\mathbf{n}|$. Since $x(t) \in t + t^2 \Q\llbracket t \rrbracket$, the term $c \, \mathbf{t^n}$ is a lowest order monomial in the formal power series   $F(x(\mathbf{t})) \in \Z \llbracket \mathbf{t} \rrbracket$, and so $c \in \Z \setminus \{0\}$. Thus  we have the Liouville lower bound: 
\numequation \label{LLB}
|c| \geq 1.
\end{equation}
On the other hand, \eqref{cleared up} and the normalizations $h(z) \in 1 + z\, \C \llbracket z \rrbracket$  and $\varphi(z) \in \varphi'(0) z + z^2 \, \C \llbracket z \rrbracket$ exhibit $c \varphi'(0)^{\beta} \, \mathbf{z^n}$ as  a lowest order monomial in $H(\mathbf{z})$. Since the $\mathbf{z^n}$ coefficient is also computed by Cauchy's integral formula $  \int_{\T^d} \frac{ H(\mathbf{z}) }{ \mathbf{z^n}  } \, \mu_{\mathrm{Haar}}(\mathbf{z})$, we have the Cauchy upper bound: 
\numequation \label{CUP}
|c| \cdot |\varphi'(0)|^{\alpha}    \leq |c| \cdot |\varphi'(0)|^{\beta} \leq \sup_{\T^d} |H|. 
\end{equation}
To estimate the last supremum under the asymptotic $\alpha \to \infty$ for  fixed $d$ and $\kappa$, we note that~\eqref{cleared up} expands from~\eqref{theform}
into a~$\Z$-linear combination of $(mD)^d = \exp\big( o(\alpha) \big)$ terms of the form
$$
 \prod_{j=1}^{d} h(z_j)^{D-k_j} \prod_{j=1}^{d} (h(z_j)  \cdot  p(\varphi(z_j)) )^{k_j} \cdot f_{i_1}(\varphi(z_1)) \cdots 
f_{i_d}(\varphi(z_d)),
$$
for some $k_1, \ldots, k_d \in \{0,\ldots,D-1\}$, and with coefficients bounded in magnitude by
the quantity $\exp \big( \kappa C \alpha + o(\alpha)  \big)$.
Every such term is bounded in magnitude on $\T^d$ by
$$
e^{\kappa C \alpha + o(\alpha)} \cdot \max\big\{ \sup_{\T} |h|, \, \sup_{\T} |h \cdot \varphi^*p| \big\}^{dD} \cdot \max_{1 \leq i \leq m} \sup_{\T} |\varphi^* f_i|^d. 
$$
By the triangle inequality, we have in the $\alpha \to \infty$ asymptotic---with respect to a fixed $d$---the supremum bound
\begin{equation*}
\sup_{\T^d} \log{|H|}    \leq  dD  \cdot \max\big\{ \sup_{\T} \log{|h|}, \, \sup_{\T} \log{|h \cdot \varphi^*p|} \big\} + \kappa C \alpha + o(\alpha).
\end{equation*}
Combining with~\eqref{LLB} and~\eqref{CUP}, we get the asymptotic bound
$$
\alpha \log{|\varphi'(0)|}  \leq  \frac{d}{(d!)^{1/d}} \Big( 1 + \frac{1}{\kappa} \Big)^{\frac{1}{d}} \cdot \frac{\alpha}{m}   \max\big\{ \sup_{\T} \log{|h|}, \, \sup_{\T} \log{|h \cdot \varphi^*p|} \big\} + \kappa C \alpha +  o(\alpha)
$$
as $\alpha \to \infty$ with respect to the other parameters. 

This proves  the dimension bound
$$
m \leq  \inf_{ \substack{  d \in \NwithoutzeroA  \\ 0 < \kappa < (\log{|\varphi'(0)|})/|C| }  } \left\{  \frac{ \frac{d}{(d!)^{1/d}} \Big( 1 + \frac{1}{\kappa} \Big)^{\frac{1}{d}} \cdot  \max\big\{ \sup_{\T} \log{|h|}, \, \sup_{\T} \log{|h \cdot \varphi^*p|} \big\}}{{\log{|\varphi'(0)| - \kappa C}} }   \right\}
$$
contingent on the denominator being positive.
Lemma~\ref{meromorphic form} now follows by firstly letting $d \to \infty$ and then $\kappa \to 0$, and observing that in that limit
$$
  \frac{d}{(d!)^{1/d}} \Big( 1 + \frac{1}{\kappa} \Big)^{\frac{1}{d}} \to e \qquad \textrm{ while } \qquad \kappa C \to 0,
$$
by Stirling's asymptotic and the key point that the constant $C$ depends only 
on~$P$ and~$\varphi$  but not on either $d$ or $\kappa$.

 The algebraicity of $f_i$ follows {\it a fortiori} 
 by the finite dimension bound~\eqref{abstracted bound},  since all powers of $f_i$ satisfy $x^*f_i^N\in \Z\llbracket t \rrbracket$ and $\varphi^*f_i^N \in \mathcal{O}(\overline{D(0,1)})$, for any $N\in \NwithoutzeroA$.   \qed

\subsection{Canonical factorization and proof of Theorem~\ref{abstract form}} \label{canonical decomposition}
At this point Theorem~\ref{abstract form} comes as the immediate combination of Lemma~\ref{meromorphic form} and the following classical lemma of Nevanlinna. 

\begin{lemma}[Nevanlinna~\cite{Nevanlinna}]  \label{Beurling Nevanlinna}
Consider a holomorphic function $g : \overline{D(0,1)} \to \C$, and let~$\varepsilon > 0$.  Then there exists a quotient representation
$$
g = \frac{hg}{h},
$$
where $h : \overline{D(0,1)} \to \C$ is holomorphic with
$$
h(0) = 1 \quad \textrm{and}  \quad 
\max \big\{ \sup_{\T} \log{|h|}, \, \sup_{\T} \log{|hg|} \big\} \leq \int_{\T} \log^+{|g|} \, \mu_{\mathrm{Haar}} + \varepsilon.
$$
\end{lemma}

\begin{proof} 
This is in~\cite[\SSS~VII.1.4, Theorem on p.~187]{Nevanlinna} or~\cite[\SSS~VII.5]{Goluzin}, in the more general setting of meromorphic maps~$g(z)$; with the corresponding statement replacing~$\int_{\T} \log^+{|g|} \, \mu_{\mathrm{Haar}}$ by the full Nevanlinna characteristic~\S~\ref{integratedgrowthbound} of~$g$. We present the
argument for the reader's convenience, sticking to the holomorphic case of our statement. 
The statement is, of course, trivial for the zero function; we assume~$g \not\equiv 0$.  
Let~$a_1, \ldots, a_k$ be the finitely many zeros in the closed unit disc   $\overline{D(0,1)} = \{ |z| \leq 1 \}$
of the nonzero meromorphic function~$g : \overline{D(0,1)} \to \C$. (The latter, we recall, means by definition that~$g$ is meromorphic on some open neighborhood of the closed disc; hence the finiteness of the set of zeros that lie in the closed disc.) Let~$n_i \in \NwithzeroA$ be the multiplicity of the zero~$a_i$. The \emph{Blaschke product}
$$
B(z) := \prod_{i=1}^{k} \left( \frac{z - a_i}{ 1 - \overline{a_i} z}  \right)^{n_i} \, : \, D(0,1) \to D(0,1)
$$
is a holomorphic self-map~$D(0,1) \to D(0,1)$ of the unit disc that preserves its boundary, and
$$
G(z) := \frac{g(z)}{B(z)} \in \mathcal{O}^{\times}(D(0,1))
$$
 is a \emph{functional unit} on the open disc: a 
nowhere vanishing holomorphic function on~$D(0,1)$.   The function~$\log{|G|} : D(0,1) \to \R$ 
is therefore harmonic, and so the Poisson kernel formula --- see~\eqref{Poisson extended} below for a review --- together with the canonical decomposition $\log = \log^+ - \log^-$ into positive and negative parts gives the quotient representation
$$
G(z) = \frac{  \exp \Big(  -  \int_{|z|=r} \log^+{ \frac{1}{|G(w)|} } \cdot \frac{w+z}{w-z} \, \mu_{\mathrm{Haar}}(w) \Big) }{
 \exp \Big(  -  \int_{|z|=r} \log^+{ |G(w)| } \cdot \frac{w+z}{w-z}  \, \mu_{\mathrm{Haar}}(w) \Big) }
 $$
 on a neighborhood of the closed disc~$|z| \leq r$, for every~$r < 1$. Therefore, taking~$r \to 1^-$, we have quotient
 representation on the open disc~$D(0,1)$: 
\numequation \label{functional unit case canonical factorization} 
\begin{aligned}
G(z) = \frac{  \exp \Big(  -  \int_{\T} \log^+{ \frac{1}{|G(w)|} } \cdot \frac{w+z}{w-z} \, \mu_{\mathrm{Haar}}(w) \Big) }{
 \exp \Big(  -  \int_{\T} \log^+{ |G(w)| } \cdot \frac{w+z}{w-z}  \, \mu_{\mathrm{Haar}}(w) \Big) }
  =  \frac{  \exp \Big(  -  \int_{\T} \log^+{ \frac{1}{|g(w)|} } \cdot \frac{w+z}{w-z} \, \mu_{\mathrm{Haar}}(w) \Big) }{
 \exp \Big(  -  \int_{\T} \log^+{ |g(w)| } \cdot \frac{w+z}{w-z}  \, \mu_{\mathrm{Haar}}(w) \Big) }. 
 \end{aligned}
\end{equation}
In this factorization, the top and bottom both are holomorphic functions of $z \in D(0,1)$, and they both are bounded in absolute value by $\leq 1$, because the Poisson kernel satisfies $\Re\Big( \frac{w+z}{w-z} \Big) > 0$ for $1 = |w| > |z|$. Furthermore, the bottom in~\eqref{functional unit case canonical factorization} takes the value $ \exp \Big( -  \int_{\T} \log^+{|g|} \, \mu_{\mathrm{Haar}} \Big)$ at~$z = 0$. Therefore, on the open disc~$D(0,1)$, the definition 
$$
h(z) := \exp\left( \int_{\T} \log^+{|g|} \, \mu_{\mathrm{Haar}}    -   \int_{\T} \log^+{ |g (w)| } \cdot \frac{w+z}{w-z}  \, \mu_{\mathrm{Haar}}(w)  \right)
$$
fulfills~$h(0) = 1$ and~$\sup_{D(0,1)} |h| \leq  \int_{\T} \log^+{|g|} \, \mu_{\mathrm{Haar}}$, 
but then also~\eqref{functional unit case canonical factorization} taken with~$g = BG$ gives 
\begin{equation*}
\begin{aligned}
& \quad
\sup_{D(0,1)} \left\{ \log{ |hg|}  \right\} = \sup_{D(0,1)} \left\{ \log{ |hBG|} \right\}  \leq \sup_{D(0,1)} \left\{ \log{  |hG| } \right\} \\
&  = \sup_{|w| < 1} \left\{ \int_{\T} \log^+{|g|} \, \mu_{\mathrm{Haar}}    -    \int_{\T} \log^+{ \frac{1}{|g(w)|} } \cdot \frac{w+z}{w-z} \, \mu_{\mathrm{Haar}}(w) \right\} \\
& \leq  \int_{\T} \log^+{|g|} \, \mu_{\mathrm{Haar}}.
\end{aligned}
\end{equation*}

This constructs the desired quotient representation except on the open disc~$D(0,1)$ rather than on a neighborhood of the closed disc. 
%\cite[\SSS~VII.1.4, (1.5'), (1.7), (1.8)]{Nevanlinna}
The full statement bootstraps from this by the following limiting argument; this is where the~$\varepsilon > 0$ emerges in the statement of the theorem. 
Taking a small enough $\delta>0$ (to be chosen at the end in dependence on $\varepsilon$) such that $g$ is a holomorphic function on $\overline{D(0, 1+\delta)}$,  
we apply the preceding to the holomorphic function $\widetilde{g}(z):=g((1+\delta)z)$ on $\overline{D(0,1)}$. We obtain a holomorphic function $\widetilde{h} : D(0,1) \to \C$ such that $\widetilde{h}(0)=1$ and
\[\max \left\{\sup_{ D(0,1)} \log |\widetilde{h}|, \sup_{D(0,1)} \log |\widetilde{h}\widetilde{g}| \right\} \leq \int_{\T} \log^+{|\widetilde{g}|} \, \mu_{\mathrm{Haar}}.\]
Define $h(z):=\widetilde{h}(z/(1+\delta))$,  a holomorphic function on $D(0,1+\delta) \supset \overline{D(0,1)}$ with $h(0)=1$. Then
\[\max \big\{ \sup_{\T} \log{|h|}, \, \sup_{\T} \log{|hg|} \big\} \leq \int_{\T} \log^+{|\widetilde{g}|} \, \mu_{\mathrm{Haar}} = \int_{|z|=1+\delta } \log^+{|g(z)|} \, \mu_{\mathrm{Haar}}.\]
We get what we want upon choosing~$\delta = \delta(\varepsilon) > 0$ small enough to have $\displaystyle \int_{|z|=1+\delta } \log^+{|g(z)|} \, \mu_{\mathrm{Haar}} \leq \int_{\T} \log^+{|g|} \, \mu_{\mathrm{Haar}} + \varepsilon$. 
\end{proof}

\begin{remark} \label{large numbers} 
Conversely, for any holomorphic map $h : \overline{D(0,1)} \to \C$ with $h(0) = 1$, we have the lower bound
$$
\max \big\{ \sup_{\T} \log{|h|}, \, \sup_{\T} \log{|hg|} \big\}  \geq  \int_{\T} \log^+{|g|} \, \mu_{\mathrm{Haar}},
$$
as one sees immediately from integrating the pointwise identity
$$
\max\{ \log{|h|}, \log{|hg|} \}  =  \log{|h|} + \log^+{|g|}
$$
over $\T$ and using $\int_{\T} \log{|h|} \geq \log{|h(0)|} = 0$ from subharmonicity. This shows the necessity of the $\varepsilon$ in
Lemma~\ref{Beurling Nevanlinna}. It also shows that Theorem~\ref{abstract form}---our final goal of the current~\SSSS~\ref{holonomy abstraction},
which at this point is fully proved---is in fact equivalent with the intermediate form Lemma~\ref{meromorphic form}.  

On the other hand, with a bit more work based on the Law of Large Numbers, we could restrict the auxiliary construction~\eqref{theform} to only admit those exponent vectors $\bk \in \{0, \ldots, D-1\}^{d}$ that have $\sum_{i=1}^{d} k_i/D$ concentrated around the expectation $d/2$. With such a variant of Lemma~\ref{Siegel}, the same argument leads to 
the finer bound
\begin{equation*} 
m \leq (e/2) \,  \inf_{h: \, h(0) = 1} \Big\{ \frac{   \sup_{\T}{\log{|h|}} +  \sup_{\T}{\log{|h \cdot \varphi^* p|}}   }{\log{|\varphi'(0)|}} \Big\},
\end{equation*}
where the infimum is taken over all holomorphic mappings $h : \overline{D(0,1)} \to \C$ subject to the
normalizing constraint $h(0) = 1$.
We will not need this improvement here. 
\end{remark}

\subsection{A first alternative proof} \label{psh approach} In this section, we complete an idea proposed to us by Andr\'e as an alternative to our original proof of  Theorem~\ref{abstract form} (itself recounted in~\SSSS~\ref{initial argument} further down), based on  plurisubharmonicity and a lexicographic induction instead of on Cauchy's formula. 
We invite the reader at this point to skip ahead  directly to \SSSS~\ref{main plan} on a first pass, as the arithmetic holonomy bound~\eqref{dimensionbound}---the algebraization ingredient that we need for the unbounded denominators conjecture---has already been proved. 

\subsubsection{Lemma on the lexicographically lowest coefficient}\label{lem_lexi}
 The  extrapolation step will now be based on the following analytic lemma, to be applied with $G(\mathbf{z}) = F(\varphi(z_1), \ldots, \varphi(z_d))$, where $F(\mathbf{x})$ is our auxiliary function from Lemma~\ref{Siegel}. The lemma reflects the plurisubharmonic property of the multivariable complex functions of the form $\log{|H(\mathbf{z})|}$ with $H(\mathbf{z})$ holomorphic, used inductively on the number of variables $d$. %This approach was suggested to us by Andr\'e as an alternative to our original proof.

\begin{lemma}  \label{lowest lexi}
Consider a function $G(\mathbf{z}) \in \C \llbracket \mathbf{z} \rrbracket \setminus \{0\}$ holomorphic on the closed unit polydisc $\{\mathbf{z} \, : \, \max_{i=1}^d |z_i|\leq 1\}$, and let $c \, \mathbf{z^n}$ be the lexicographically minimal monomial. Then
\numequation \label{lexi bound}
 \log{|c|}   \leq  \int_{\T^d} \log{|G|} \, \mu_{\mathrm{Haar}}. 
\end{equation}

\begin{proof}
We induct on the number of variables $d$. For $d = 1$, the bound \eqref{lexi bound} follows directly from Jensen's formula, or from the subharmonic property of the function $u(z) := \log{|z^{-n}G(z)|}$, which entails
$$
\log{|c|}  = u(0)  \leq \int_{\T} u \, \mu_{\mathrm{Haar}}  = \int_{\T} \log{|G|} \, \mu_{\mathrm{Haar}}. 
$$ 
The last equality uses that the functions $u(z) = \log{|z^{-n}G(z)|}$ and $\log{|G(z)|}$ have the same restriction on the unit circle $\T$.

For the induction step, we write $\mathbf{z} = (z_1, \mathbf{z}')$ and
$$
G(\mathbf{z}) = z_1^{n_1} H(\mathbf{z}), 
$$
where $H \in \C\llbracket \mathbf{z} \rrbracket$ is holomorphic   by our lexicographic minimality assumption. For any fixed $\mathbf{z}' \in \T^{d-1}$, by the same argument as the $d=1$ case above, we have
\numequation \label{frozen}
\log{|H(0, \mathbf{z}')|} \leq \int_{\T} \log{|H(z_1, \mathbb{z}')|} \, \mu_{\mathrm{Haar}}(z_1) = \int_{\T} \log{|G(z_1, \mathbf{z}')|} \, \mu_{\mathrm{Haar}}(z_1).
\end{equation}
By assumption, the lexicographically minimal monomial in $H(0,\mathbf{z}') \in \C \llbracket \mathbf{z}' \rrbracket$ is equal to $c \, \mathbf{z}'^{\mathbf{n}'}$, where $\mathbf{n} = (n_1, \mathbf{n}')$. Therefore the induction hypothesis gives
\numequation
\log{|c|}  \leq \int_{\mathbf{T}^{d-1}} \log{| H(0, \mathbf{z}') |} \, \mu_{\mathrm{Haar}} (\mathbf{z}').
\end{equation}
We complete the induction by integrating the inequality \eqref{frozen} over $\mathbf{z}' \in \T^{d-1}$. 
\end{proof}
\end{lemma}

\subsubsection{Extrapolation and first alternative proof of Theorem~\ref{abstract form}}  \label{extrapolation part}
We apply Lemma~\ref{lowest lexi} to the $\varphi$-pullback of our $d$-variate auxiliary function: 
\numequation
G(z_1, \ldots, z_d) := F(  \varphi(z_1), \ldots, \varphi(z_d)  )  \in \C \llbracket \mathbf{z} 
\rrbracket \setminus \{0\}.
\end{equation}
This is holomorphic in a neighborhood of the closed unit polydisc, because all the split-variables constituents 
$$
\varphi^*p; \quad  \varphi^* f_1, \ldots, \varphi^* f_m \in {\mathcal{O}(\overline{D(0,1)})}.
$$

Thus, with $c \, \mathbf{z^n}$ the lexicographically lowest monomial in $G(\mathbf{z})$, we get from equation~\eqref{lexi bound} and Lemma~\ref{Siegel} our Cauchy upper bound:
\numequation
\begin{split}
 \label{CB}
\log{|c|}   \leq \int_{\T^d} \log{|F(\varphi(z_1), \ldots, \varphi(z_d)) |} \, 
\mu_{\mathrm{Haar}} \\  
\leq  dD  \int_{\T} \log^+{|p \circ \varphi|}
\, \mu_{\mathrm{Haar}}  +  \kappa C \, \alpha + o(\alpha), 
\end{split}
\end{equation}
asymptotically as $\alpha \to \infty$ with regard to the other parameters. Here, we used the pointwise triangle inequality bound
$$
\log{|F(x_1,\ldots,x_d)|} \leq D \sum_{i=1}^d \log^+{|p(x_i)|}
+ \kappa C \alpha + o(\alpha)
$$
for $x_i := \varphi(z_i)$ (note that the sum in~\eqref{theform} is comprised 
of $(mD)^d = \exp(o(\alpha))$ terms), and integrated this pointwise bound over the unit polycircle $\mathbf{z} \in \T^d$. 

The Liouville lower bound comes down to  the integrality property 
\numequation        \label{spacing}
F(x(t_1), \ldots, x(t_d)) \in \Z\llbracket \mathbf{t} \rrbracket
\end{equation}
inherited from our respective assumptions
$$
x^*p; \quad x^* f_1, \ldots, x^* f_m \in \Z\llbracket t \rrbracket
$$
on the split-variables constituents in~\eqref{theform}. Given our normalizations $x(t) \in t + t^2 \Q\llbracket t \rrbracket$ and $\varphi(z)  \in \varphi'(0) z + z^2 \C \llbracket z \rrbracket$, the lexicographically lowest term of $G(\mathbf{z}) \in \C \llbracket \mathbf{z} \rrbracket$ is equal to $\varphi'(0)^{\beta}$ times the lexicographically lowest term of $F(x(\mathbf{t})) \in \Z\llbracket \mathbf{t} \rrbracket$, where $\beta := |\mathbf{n}| = n_1 + \cdots + n_d \geq \alpha$ is the common total degree of these lexicographically lowest terms in $G(\mathbf{z})$ and $F(x(\mathbf{t}))$. By~\eqref{spacing}, this entails that the nonzero coefficient
$$
c \in \varphi'(0)^{\beta} \, \Z \setminus \{0\},
$$
and hence {\it a fortiori} that
\numequation \label{LB}
\log{|c|} \geq \beta \log{|\varphi'(0)|} \geq \alpha \log{|\varphi'(0)|}.
\end{equation}

We get our requisite dimension bound~\eqref{abstracted bound} on combining the degree bound~(\ref{degrees}) of Lemma~\ref{Siegel} with the Cauchy upper bound~\eqref{CB} and the Liouville lower bound~\eqref{LB}, and letting firstly $\alpha \to \infty$, then $d \to \infty$, and finally $\kappa \to 0$.

\medskip

This completes another proof of Theorem~\ref{abstract form}.   \qed

\subsection{A second alternative proof}   \label{initial argument}

The remainder of~\SSSS~\ref{holonomy abstraction} presents our original argument for Theorem~\ref{abstract form}, with the thought that it could still be useful for other settings including potential theory (see~\ref{potential sketch}). Like~\SSSS~\ref{extrapolate} and unlike~\SSSS~\ref{psh approach}, it is based on the leading order jet rather than the overall lexicographically lowest monomial in $F(\mathbf{x})$, and on the pointwise Cauchy integral formula instead of on plurisubharmonicity. Contrastingly to both, it employs a cross-variables equidistribution idea. 

\subsubsection{Equidistribution} \label{equidistribution} We start out the same way as with Lemma~\ref{Siegel}, but now aim to extrapolate based directly on the pointwise Cauchy bound.  The key idea here is that upon substituting $x_j = \varphi(z_j)$ into (\ref{theform}), the $d \to \infty$ equidistribution on the circle of the 
uniform independent and identically distributed points $z_1, \ldots, z_d$ will normally get the constituent monomials in (\ref{theform})  to grow at most at the integrated exponential rate of $dD \int_{\T} \log^+{|p \circ \varphi| }\, \mu_{\mathrm{Haar}}$. %It is this Monte Carlo (randomized numerical integration) principle that makes possible the integrated growth term --- as opposed to the cruder and rather more straightforward supremum growth term, compare to~\cite[VIII~1.6]{AndreG} ---  in our holonomy rank bound (\ref{dimensionbound}).  Having $ \int_{\T} \log^+{|p \circ \varphi|} \, \mu_{\mathrm{Haar}}$ instead of $\sup_{|z| = 1} \log{|p \circ \varphi|} $ in (\ref{dimensionbound}) is a critical step in our proof of the unbounded denominators conjecture. 
%To carry out this program, we employ in \SSSS~\ref{extrapolation part} Andr\'e's method~\cite{AndreG,Andre} of extrapolating with the function
%\numequation \label{the norm}
%F(x(\mathbf{t})) := F(x(t_1), \ldots, x(t_d)) \in \Z\llbracket \mathbf{t} \rrbracket \setminus \{0\},
%\end{equation}
%whose integrality of coefficients follows from Lemma~\ref{Siegel} and our defining 
%assumptions that $p(x) \in \Z\llbracket t \rrbracket$ while all $f_i(x(t)) \in \Z\llbracket t \rrbracket$. This integrality is the key to the Liouville lower bound. 
The problem with directly applying the Cauchy bound as in~\cite[VIII~1.6]{AndreG} is that it involves a pointwise upper bound on the intervening functions $| p(\varphi(\mathbf{z}))^{\mathbf{k}}|$ on the unit polycircle $\mathbf{z} \in \T^d$, and while the Monte Carlo heuristic applies on the majority of $\T^d$ under $d \to \infty$, with a probability tending to $1$ roughly speaking at a rate exponential in $-d$ (this follows by Hoeffding's concentration inequality with~\eqref{Erdos-Turan} below), the peaks at the biased part of $\T^d$ get overwhelmingly large, and a direct extrapolation with~\eqref{theform} in this way  still only leads to a dimension bound with $\sup_{|z| = 1} \log{|p \circ \varphi|} $. 

To improve the supremum term to the mean term $\int_{\T} \log^+{|p \circ \varphi|} \, \mu_{\mathrm{Haar}}$, we dampen the size at the peaks by firstly multiplying (\ref{theform}) by a suitably chosen power $V(\mathbf{z})^M$ of the Vandermonde polynomial 
\numequation \label{Vandermonde}
V(\mathbf{z}) := \prod_{i< j}(z_i - z_j) = \det{ \begin{bmatrix} 1 & z_1 & z_1^2 & \cdots & z_1^{d-1} \\ 1 & z_2 & z_2^2 & \cdots & z_2^{d-1} \\
\vdots & \vdots & \vdots & \cdots & \vdots \\
1 & z_d & z_d^2 & \cdots & z_d^{d-1}\end{bmatrix} } \in \Z [z_1, \ldots, z_d] \setminus \{0\}.
\end{equation}
By applying the Hadamard volume inequality to the Vandermonde determinant in (\ref{Vandermonde}), we recover the following classical result of Fekete, crucial for the present approach. 

\begin{lemma}[Fekete] \label{Fekete}
The supremum of $|V(\mathbf{z})| = \prod_{1 \leq i < j \leq d} |z_i-z_j|$ over the unit polycircle $\mathbf{z} \in \T^d$ is equal to $d^{d/2}$, with equality if and only if the points $z_1, \ldots, z_d$ are the vertices of a regular $d$-gon. 
\end{lemma}
 
The idea for sifting out the equidistributed tuples $(z_1, \ldots, z_d)$ is the following. If the points $z_1, \ldots, z_d$ are poorly distributed in the uniform measure of the circle,  the quantity $|V(\mathbf{z})|$ is uniformly exponentially small in $-d^2$ (Lemma~\ref{non-equidistribution} below). This plays off against the $d^{d/2} = \exp(o(d^2))$ bound of Lemma~\ref{Fekete} to sift out the equidistributed points in our pointwise upper bound in the Cauchy integral formula when we extrapolate in \SSSS~\ref{extrapolation part} above. 
Liouville's Diophantine lower bound still succeeds like in Andr\'e~\cite[\SSS5]{Andre}, thanks to the chain rule and the integrality of the expansion~(\ref{Vandermonde}), but at the Cauchy upper bound  we are now aided by the fact that $V( \mathbf{z})^M$ is extremely small (an exponential in $-M d^2$, see (\ref{damping grace})) at the peaks of the pointwise Cauchy bound, where the point $(z_1, \ldots,z_d)$ is poorly distributed, while still not too large (subexponential in $M d^2$, thanks to Lemma~\ref{Fekete}) uniformly throughout the whole polycircle~$\T^d$.

In the remainder of the current subsection, we spell out the notion of `well-distributed' and `poorly distributed', and supply the key equidistribution property for the numerical integration step. The following is the standard notion of discrepancy theory. 

\begin{df}
The (normalized, box) {\bf discrepancy function} $D : \T^d \to (0,1]$ is the supremum over all circular arcs
$I \subset \T$ of the defect between the normalized arc length of $I$ and the proportion of points falling inside $I$: 
$$
D(z_1, \ldots, z_d) := \sup_{I \subset \T} \big| \mu_{\mathrm{Haar}}(I) - \frac{1}{d}  \# \{ i \, : \, z_i \in I \}  \big|.
$$
\end{df}

We also recall the basic properties of the total variation functional on the circle. In our situation, all that we need is that $\log^+{|h|}$ is of bounded variation 
for an arbitrary $C^1$ function $h : \T \to \R$. Then Koksma's estimate permits us to integrate numerically. All of this can be alternatively phrased 
in the qualitative language of  weak-$*$ convergence.

\begin{df}
The {\bf total variation} $V(g)$ of a function $g : \T \to \R$ is the supremum over all partitions $0 \leq \theta_1 < \cdots < \theta_n < 1$ of $\sum_{j=1}^{n-1} | g(e^{2\pi \sqrt{-1} \theta_{j+1}}) - g(e^{2\pi \sqrt{-1} \theta_j}) |$. 
\end{df}

Thus, for $g \in C^1(\T)$, we have the simpler formula 
\numequation \label{total variation}
V(g) = \int_{\T} |g'(z)| \, \mu_{\mathrm{Haar}}(z), \quad g \in C^1(\T). 
\end{equation}

We have $V(\log^+{|h|}) < \infty$ for $h \in C^1(\T)$, and Koksma's inequality (see for example Drmota--Tichy~\cite[Theorem~1.14]{DrmotaTichy}):
\numequation \label{Koksma}
\Big| \frac{1}{d} \sum_{j=1}^d g(z_j) - \int_{\T} g \, \mu_{\mathrm{Haar}} \Big| \leq V(g) D(z_1, \ldots, z_d).
\end{equation}
In practice the discrepancy function is conveniently estimated by the  Erd\"os--Tur\'an  inequality (cf. Drmota--Tichy~\cite[Theorem~1.21]{DrmotaTichy}):
\numequation \label{Erdos-Turan}
D(z_1, \ldots, z_d) \leq 3 \Big(  \frac{1}{K+1} + \sum_{k=1}^K \frac{1}{k} \Big| \frac{z_1^k + \cdots + z_d^k}{d} \Big| \Big), \quad \textrm{for all } K \in \NwithzeroA,
\end{equation}
 in terms of the power sums. 
 Here we note in passing that, by (\ref{Erdos-Turan}) and the Chernoff tail bound or the Hoeffding concentration inequality (see, for example, Tao~\cite[Theorem~2.1.3 and Ex.~2.1.4]{Tao}), 
 we have that for any fixed $\varepsilon > 0$, the probability 
of the event $D(z_1, \ldots, z_d) \geq \varepsilon$ decays to $0$ exponentially in $-d$ as $d \to \infty$. This last remark has purely a heuristic value for our next step, and is not used in the estimates in itself (but rather shows that these estimates are sharp).  

Thus we introduce another parameter $\varepsilon > 0$, which in the end will be let to approach $0$ but only after $d \to \infty$, and we divide the points $\mathbf{z} \in \T^d$ into two groups
according to whether $D(z_1, \ldots, z_d) < \varepsilon$ (the well-distributed points) or $D(z_1 \ldots, z_d) \geq \varepsilon$ (the poorly distributed points). For the well-distributed group  we use Koksma's inequality (\ref{Koksma}), and for the poorly distributed group we take advantage of the overwhelming damping force of the Vandermonde factor. 

The following is essentially Bilu's equidistribution theorem~\cite{Bilu}, in a mild disguise.

\begin{lemma} \label{non-equidistribution} 
There  are functions $c(\varepsilon) > 0$ and $d_0(\varepsilon) \in \R$  such that, for every $\varepsilon \in (0,1]$, if $d \geq d_0(\varepsilon)$  and $(z_1, \ldots, z_d) \in \T^d$ is a $d$-tuple with discrepancy $D(z_1, \ldots, z_d) \geq \varepsilon$, then
\numequation \label{damping grace}
|V(z_1, \ldots, z_d)| = \prod_{1 \leq i < j \leq d} |z_i - z_j| < e^{-c(\varepsilon) d^2}. 
\end{equation}
\end{lemma}

\begin{proof} 
Since the qualitative result suffices for our purposes here, we give a soft proof based on compactness. The following argument borrows from
Bombieri and Gubler's exposition~\cite[page~103]{BombieriGubler} of Bilu's equidistribution theorem. The contrapositive of the requisite statement
is the existence of an~$\varepsilon \in (0,1]$ with
\[\liminf_{d\rightarrow \infty} \left\{ \inf_{\mathbf{z} \in \T^d, D(\mathbf{z}) \geq \varepsilon} \frac{1}{d^2} \sum_{1 \leq i < j \leq d} \log{\frac{1}{|z_i-z_j|}} \right\} \leq 0.\]
(If this quantity is strictly positive for all~$\varepsilon \in (0,1]$, then define~$c(\varepsilon) > 0$ to be that quantity.)
Hence, arguing for the contradiction, we suppose 
that there is an $\varepsilon \in (0,1]$ and an infinite sequence $(z_1^{(d)}, \ldots, z_d^{(d)}) \in \T^d$ such that
\numequation \label{non-positive energy}
\lim_{d \to \infty} \left\{ \frac{1}{\binom{d}{2}} \sum_{1 \leq i < j \leq d} \log{\frac{1}{|z_i^{(d)}-z_j^{(d)}|}} \right\} \leq 0,
\end{equation}
but
\numequation \label{uniform non-equidistribution}
\textrm{for all } d \in \NwithoutzeroA, \quad  \quad D(z_1^{(d)},\ldots, z_d^{(d)}) \geq \varepsilon.
\end{equation}
By the Banach--Alaoglu theorem of the compactness of the weak-$*$ unit ball of $C(\T)^*$, we may extract 
a subsequence of the sequence of normalized Dirac masses $\delta_{\{ z_1^{(d)},\ldots, z_d^{(d)} \}}$ that converges weak-$*$ 
to some limit probability measure $\mu$ of the unit circle. By   continuity of the discrepancy functional, (\ref{uniform non-equidistribution}) 
implies that the limit discrepancy
$$
D(\mu) := \sup_{I \subset \T} \big| \mu_{\mathrm{Haar}}(I) - \mu(I)  \big| \geq \varepsilon.
$$
In particular, $\mu$ is not the uniform measure $\mu_{\mathrm{Haar}}$. 

On the other hand, it is a well-known theorem from potential theory that every compact $K \subset \C$ admits a unique probability measure $\mu_K$, called the equilibrium measure, that minimizes the Dirichlet energy integral
$$
I(\nu) :=  \iint_{K \times K} \log{ \frac{1}{|z-w|} } \, \nu(z) \, \nu(w) 
$$
across all probability measures $\nu$ supported by $K$. Since $\T$ is invariant under rotation and $\mu_{\T}$ is unique, we have $\mu_{\T} = \mu_{\mathrm{Haar}}$, and since $I(\mu_{\mathrm{Haar}}) = 0$, but $\mu \neq \mu_{\mathrm{Haar}}$, we have the strict inequality
\numequation \label{positive energy}
I(\mu) =  \iint_{\T \times \T} \log{ \frac{1}{|z-w|} } \, \mu(z) \, \mu(w)  > 0. 
\end{equation}
If the measure $\mu$ is continuous (that is, the measure of a point is $0$, or equivalently the diagonal of $\T \times \T$ has $\mu \times \mu$ measure $0$), then the positive energy (\ref{positive energy}) contradicts (\ref{non-positive energy}) by 
weak-$*$ convergence.  In more detail, take a continuous function $\phi : [0, \infty) \to [0,\infty)$ to have
$\phi|_{[0,1/2]} \equiv 0$ and $\phi|_{[1,\infty)} \equiv  1$, and let $\phi_{\eta}(t) := \phi(t /\eta)$ for $0 < \eta \leq 1$. 
Then, since $\phi_{\eta}(t) < 1$ implies $\log{(1/t)} > 0$ while $\phi_{\eta}(t) \leq 1$ always, assumption
(\ref{non-positive energy}) implies
$$
\lim_{d \to \infty} \frac{1}{\binom{d}{2}} \sum_{1 \leq i < j \leq d} \phi_{\eta}\big( |z_i^{(d)}-z_j^{(d)}| \big) \log{\frac{1}{|z_i^{(d)}-z_j^{(d)}|}} \leq 0
$$
leading by weak-$*$ convergence to the non-positivity
$$
 \iint_{\T \times \T} \phi_{\eta}(|x-y|) \log{ \frac{1}{|z-w|} } \, \mu(z) \, \mu(w) \leq 0, 
$$
for every $\eta \in (0,1]$. Since the diagonal has measure $0$,  this runs in contradiction with (\ref{positive energy}) upon letting $\eta \to 0$.

If instead the measure $\mu$ is not continuous, then there is a point $a \in \T$ and a positive constant $c > 0$ such that, for any $\eta > 0$, and any $d \gg_{\eta} 1$ sufficiently large, there are 
at least $cd$ points among $\{  z_1^{(d)}, \ldots, z_d^{(d)} \}$ in the neighborhood $|z - a| < \eta / 2$. The contribution to (\ref{non-positive energy}) from all these pairs of points 
is alone $\geq c^2 \log(1/\eta)$, and since the total contribution from any subset of the points is in any case $\geq -\log{2}$, we get again in contradiction with (\ref{non-positive energy}) on letting 
$\eta \to 0$.
\end{proof}

\subsubsection{Damping the Cauchy estimate} \label{damping}

We combine Lemmas~\ref{Fekete} and  \ref{non-equidistribution}  for our choice of the damping term $ V( \mathbf{z})^M$. In the following, all asymptotics are taken  under $\alpha \to \infty$ with respect to all other parameters. 

By Lemma~\ref{Siegel} and our defining assumption that all $f_i(\varphi(z))$ are holomorphic on some neighborhood of the closed unit disc $|z| \leq 1$, we have uniformly on the polycircle $\mathbf{z} \in \T^d$ the 
pointwise bound
\numequation \label{pointwise majorization}
\log{|F(\varphi(z_1), \ldots, \varphi(z_d))|} \leq D \sum_{j=1}^d \log^+{|p(\varphi(z_j))|}
+ \kappa C \alpha + o(\alpha). 
\end{equation}

Since the function $\log^+{|p \circ \varphi|} : \T \to \R$ is  of finite variation $V(\log^+{|p \circ \varphi|}) < \infty$, Koksma's estimate~(\ref{Koksma}) yields, on the well-distributed part $\mathbf{z} \in \T^d$, the uniform pointwise upper bound
\begin{equation*}
\begin{split}
%\label{major part}
 D(z_1, \ldots, z_d) < \varepsilon \quad \Longrightarrow \\
\log{  |F(\varphi(z_1), \ldots, \varphi(z_d)) | } \leq    d D  \,  \int_{\T} \log^+{ |p \circ \varphi|}  \, \mu_{\mathrm{Haar}}
 + \kappa C \alpha  + O_{p,\varphi} (\varepsilon \, dD)  + o(\alpha).
\end{split}
\end{equation*}
The implicit constant in $O_{p,\varphi}( \varepsilon \, dD)$ can be taken as the total variation  $V(\log^+{|p \circ \varphi|})$; that this error term is $o_{\varepsilon \to 0}( dD) 
= o_{\varepsilon \to 0} (\alpha)$ is all that matters to us in the asymptotic argument.

On the poorly distributed but exceptional part $D(z_1, \ldots, z_d) \geq \varepsilon$, the  sum in (\ref{pointwise majorization}) can get as large as $d \sup_{|z| = 1} \log{|p \circ \varphi|}$.
This trivial bound gives, for all $ \mathbf{z} \in \T^d$:
\numequation \label{pointwise trivial}
  \log{  |F(\varphi(z_1), \ldots, \varphi(z_d)) | } \leq dD \sup_{|z| = 1} \log^+{|p \circ \varphi|} +  \kappa C \alpha  + o(\alpha). 
\end{equation}

We now impose the condition
\numequation \label{d large}
d \geq d_0(\varepsilon), \quad  \quad \textrm{ for the function $d_0(\varepsilon)$ in Lemma~\ref{non-equidistribution}},
\end{equation}
for the remainder of the proof of Corollary~\ref{holonomy} (at the end we will firstly take $d \to \infty$, and only then $\varepsilon \to 0$), 
and we
 select the Vandermonde exponent
\numequation \label{exponent in Vandermonde}
M := \Bigl\lfloor \frac{\sup_{|z| = 1} \log^+{|p \circ \varphi|}}{c(\varepsilon)} \frac{D}{d} \Bigr\rfloor,
\end{equation}
with $c(\varepsilon)$ the function from Lemma~\ref{non-equidistribution}. We are now in a position to usefully
 estimate the supremum of $|V(\mathbf{z})^M F(\varphi(\mathbf{z}))|$ uniformly across the unit 
polycircle $\mathbf{z} \in \T^d$, by separately examining the well-distributed and the poorly distributed cases 
of $\mathbf{z}$. 

On the poorly distributed part  $D(z_1, \ldots, z_d) \geq \varepsilon$, Lemma~\ref{non-equidistribution} with (\ref{pointwise trivial}),  (\ref{d large}) and (\ref{exponent in Vandermonde}) gives
\numequation \label{biased part}
\sup_{\mathbf{z} \in \T^d: \, D(z_1, \ldots, z_d) \geq \varepsilon} \log{ |V(\mathbf{z})^MF( \varphi(\mathbf{z}))|  } \ll \kappa \alpha. 
\end{equation}

On the well-distributed part $D(z_1, \ldots, z_d) \leq \varepsilon$, we have 
\numequation
\begin{split}
\label{balanced part}
\sup_{\mathbf{z} \in \T^d: \, D(z_1, \ldots, z_d) \leq \varepsilon} \log{ |V(\mathbf{z})^M F(\varphi(\mathbf{z}))|  }  \\ \leq
  d D  \,  \int_{\T} \log^+{ |p \circ \varphi|}  \, \mu_{\mathrm{Haar}}
+ \kappa C \alpha  + O_{p,\varphi} (\varepsilon \alpha) + O_{\varepsilon, p,\varphi}\Big(\frac{\log{d}}{d} \alpha \Big)  + o(\alpha).
\end{split}
\end{equation}
 by (\ref{exponent in Vandermonde}) and Lemma~\ref{Fekete}.

Consider the holomorphic function
\numequation  \label{our damped function}
H(\mathbf{z}) :=     V( \mathbf{z})^M    F(\varphi(z_1), \ldots, \varphi(z_d))  =: \sum_{\mathbf{n} \in \NwithzeroB^d}  c(\mathbf{n}) \, \mathbf{z^n}   \in \C\llbracket \mathbf{z} \rrbracket, 
\end{equation}
convergent on the closed unit disc $\| \mathbf{z} \| \leq 1$.  For each $\mathbf{n} \in \NwithzeroB^d$, the $\mathbf{z^n}$ coefficient  of $H(\mathbf{z})$ is given
by the Cauchy integral formula
\numequation \label{Cauchy}
c(\mathbf{n}) =  \int_{\T^d} \frac{ H(\mathbf{z}) }{ \mathbf{z^n}  } \, \mu_{\mathrm{Haar}}(\mathbf{z}), 
\end{equation}
entailing the Cauchy upper bound % (here $\NwithzeroB := \N \cup \{0\} = \{0,1,2,\ldots\}$)
\numequation \label{CauchyBound}
 |c(\mathbf{n})|    \leq \sup_{\mathbf{z} \in \T^d} |H(\mathbf{z})|, \quad \textrm{for all } \mathbf{n} \in \NwithzeroB^d.
\end{equation}
On combining the bounds (\ref{balanced part}), on the well-distributed part of $\T^d$,   and (\ref{biased part}), on the 
poorly distributed part of $\T^d$, we arrive at our damped Cauchy estimate: 
\numequation
\begin{split}
 \label{the upper bound}
 \log{ |c(\mathbf{n})| }   \leq  dD \,
  \int_{\T} \log^+{ |p \circ \varphi|}  \, \mu_{\mathrm{Haar}} \\ 
  + O(\kappa  \alpha)  + O_{p,\varphi} (\varepsilon \, \alpha)  + O_{\varepsilon, p, \varphi}\Big(\frac{\log{d}}{d} \alpha \Big)  + o(\alpha), 
\end{split}
\end{equation}
asymptotically under $\alpha \to \infty$.

\subsubsection{The extrapolation}  \label{extrapolation part reprised}
Finally we combine the degree estimate (\ref{degrees}) of Lemma~\ref{Siegel} with the Cauchy bound (\ref{the upper bound}) 
and the integrality properties of the functions $F(x(\mathbf{t})) \in \Z\llbracket \mathbf{t} \rrbracket$ of (\ref{theform}) and $V(\mathbf{z}) \in \Z[\mathbf{z}]$ of (\ref{Vandermonde}).

Let $\beta \geq \alpha$ be the exact order of vanishing of $F(\mathbf{x})$ at the origin $\mathbf{x = 0}$.    Among the nonvanishing monomials $c \, \mathbf{x^n}$ of this minimal order $|\mathbf{n}| = \beta$, choose the one whose degree vector $\mathbf{n}$ has the highest lexicographical ordering. By the chain rule and the minimality of $|\mathbf{n}|$, the normalization condition~\eqref{normal} on the formal substitution $x(t)$ entails that $c \, \mathbf{t^n}$ is a minimal order term in the $t$-expansion $F(x(\mathbf{t}))$. Hence the integrality $f(x(\mathbf{t})) \in \Z \llbracket \mathbf{t} \rrbracket$ gives that $c \in \Z \setminus \{0\}$ is a nonzero rational integer. 

Consider now our product function $H(\mathbf{z}) = V(\mathbf{z})^M F(\varphi(\mathbf{z})) \in \C\llbracket \mathbf{z} \rrbracket$. 
 In the factor $V(\mathbf{z})^M$, it is $z_1^{(d-1)M} z_2^{(d-2)M} \cdots z_{d-1}^M$ that has the highest lexicographical ordering. Consequently, by the chain rule again, 
$$
c  \, \varphi'(0)^{\beta} \, z_1^{n_1+(d-1)M} z_2^{n_2+(d-2)M} \cdots z_d^{n_d}
$$
 exhibits a monomial in $V(\mathbf{z})^M F(\varphi(\mathbf{z}))$ of the minimal order $\beta + M \binom{d}{2}$,
 this is because
$$\big(n_1 + (d-1) M, n_2 + (d-2)M, \ldots, n_d\big)$$
 has the strictly highest lexicographical ordering across all monomials of degree $\beta + M \binom{d}{2}$ in $V(\mathbf{z})^M F(\varphi(\mathbf{z}))$.

We have thus found a nonzero coefficient of $H(\mathbf{z}) \in \C\llbracket \mathbf{z} \rrbracket$ that belongs to the $\Z$-module $\varphi'(0)^{\beta} \Z$, where $\beta \geq \alpha$. Thus the Cauchy upper bound~(\ref{CauchyBound}) is supplemented with the Liouville lower bound
\numequation \label{Liouville}
\sup_{\mathbf{n} \in \NwithzeroB^d} \left\{ \log{|c(\mathbf{n})|} \right\} \geq \beta \log{|\varphi'(0)|}  \geq \alpha \log{|\varphi'(0)|}.
\end{equation}
We get the requisite holonomy rank bound (\ref{abstracted bound}) on combining the degree bound
(part~(\ref{degrees}) of Lemma~\ref{Siegel}) with the Cauchy upper bound
(\ref{the upper bound}) and the Liouville lower bound~(\ref{Liouville}), and letting firstly $\alpha \to \infty$, then $d \to \infty$, then $\kappa \to 0$, and finally $\varepsilon \to 0$. 

\medskip

This concludes also our original proof of Theorem~\ref{abstract form}. \qed

\subsubsection{A potential-theoretic generalization} \label{potential sketch} The path with~\SSSSs~\ref{auxiliary construction} and~\ref{initial argument} leads straightforwardly to an extension in potential theory, which we formulate without detailing a proof. Consider a compact subset $K \subset \overline{D(0,1)}$ with \emph{transfinite diameter} $d(K)$ and \emph{equilibrium measure} $\mu_K$. 
This means that the logarithmic energy functional satisfies
$$
\iint_{K \times K} \log{\frac{1}{|x-y|}} \, \mu(x) \, \mu(y) \geq -\log{d(K)}
$$
for all probability measures~$\mu$ supported by~$K$, and the equality is attained if and only if~$\mu = \mu_K$. See, for example, \cite{Kirsch} for these definitions and their basic properties, including the relation to capacitance. 

If 
$$
\log{|\varphi'(0)|} + \log{d(K)} > 0,
$$
then under the hypotheses of Corollary~\ref{holonomy} we have the holonomy rank bound 
  \numequation  \label{potential bound}
\dim_{\Q(p(x))} \mathcal{V}(U,x(t),\Z)  \leq e  \, \frac{  \int_{K} \log^+{|p \circ \varphi|} \, \mu_{K}}{\log{|\varphi'(0)|} + \log{d(K)}}.
  \end{equation}
The cases $K  = \overline{D(0,1)}$ or $K = \T$ both recover Corollary~\ref{holonomy}.

\begin{remark}
The result is still more general than~\ref{potential sketch}, and the restriction here to $\Z\llbracket t \rrbracket$ expansions was chosen as minimal for our application to noncongruence modular forms. 
In a sequel work we will generalize our integrated holonomy rank bound, in particular to the case of $\Q\llbracket t \rrbracket$ formal functions, and study its applications to transcendence theory. 
With regard to the latter, it is of some interest to inquire about the optimal numerical constant that could take the place of the coefficient $e$ in~(\ref{potential bound}).

In these optics,  Bost and Charles~\cite[Corollary~8.3.5]{BostCharles} have very recently refined our Theorem~\ref{abstract form} 
to the cleaner form
$$
m  \leq \frac{\iint_{\T^2} \log{ |p(\varphi(z)) - p(\varphi(w))| }  \, \mu_{\mathrm{Haar}}(z) \mu_{\mathrm{Haar}}(w) }{\log{|\varphi'(0)|}}. 
$$
In particular, on replacing $p$ by $p^k$ with using the elementary inequality $\log{|x-y|} \leq \log^+{|x|} + \log^+{|y|} + \log{2}$ and
taking $k \to +\infty$, their result improves our coefficient $e$ in~\eqref{potential bound} to the value~$2$.   We do not know whether or not this is the best-possible constant. 
\end{remark}

\section{Our approach to the Unbounded Denominators Conjecture}   \label{main plan}

In this section, we lay out our main approach to the  unbounded
denominators conjecture.
This will reduce the proof to a number of independent results
in group theory, complex geometry, and complex analysis which we
take up in \SSSSs~\ref{serresection}, \ref{uniformizations}, and~\ref{nevanlinna}.
Our main idea is to use 
our arithmetic holonomicity theorems to prove the following:

\begin{proposition} \label{strategy}

Let~$F_N: D(0,1) \rightarrow \C \setminus  \mu_N$ be an analytic universal covering map sending~$0$  to~$0$.
Suppose that:
\begin{enumerate}
\item \label{radius} The conformal radius $|F_N'(0)|$ of~$F_N$ is asymptotically at least
$$16^{1/N} \left(1 + \frac{A}{N^3} \right)$$
 for some constant~$A > 0$.
\item \label{mean}  For a fixed $B > 0$, the following mean value bound holds on the circle $|z| = 1 - BN^{-3}$: 
$$
\int_{|z| = 1 - BN^{-3}} \log^+{|F_N|} \, \mu_{\mathrm{Haar}}   \ll_B \frac{\log{N}}{N}.
$$
\end{enumerate}
Then the $\Q(\lambda)$-vector space $R_{2N}$ generated by the modular functions with Fourier coefficients in $\Q$ and bounded denominators at the cusp $\zeta = i \infty$, and having cusp widths dividing~$2N$ at all cusps $\zeta \in \mathbb{P}^1(\Q)$, has dimension at most~$CN^3\log{N}$ over the field $\Q(\lambda)$ of modular functions of
level~$\Gamma(2)$,
for some absolute constant~$C$.
\end{proposition}

\begin{proof}
Let $t := q^{1/N} = e^{\pi i \tau / N}$.
We use Corollary~\ref{holonomy} with $U := \C \setminus 16^{-1/N}\mu_N$, $p(x) := x^N$ and
\numequation \label{coordinate}
x := (\lambda(\tau)/16)^{1/N} \in t + t^2 \Z[1/N]\llbracket t \rrbracket,
\end{equation}
with the Kummer integrality condition $p(x) = x^N \in \Z\llbracket q \rrbracket = \Z\llbracket t^N \rrbracket \subset \Z\llbracket t \rrbracket$ being in place. 

The integrality and cusp widths conditions in the definition of the $\Q(\lambda)$-vector space $R_{2N}$ entail  a basis of $R_{2N}$ made of elements of the ring $\mathcal{H}(U,x(t),\Z)\otimes_{\Z}\Q$. More precisely, for a modular function $f$ with Fourier expansion at $i\infty$ lying in $\Z\llbracket q^{1/N}\rrbracket \otimes_{\Z}\Q$, by our choice of $t = q^{1/N}$ and $x(t) = (\lambda(q)/16)^{1/N}$ we have  $x^* f \in \Z\llbracket t\rrbracket \otimes_{\Z} \Q$, on defining $x^*f$ as the formal $x$-expansion of $f(q) = f(q(x)) \in \Z \llbracket q^{1/N} \rrbracket \otimes_{\Z} \Q \subset  \Q \llbracket x \rrbracket$. As $f$ is a regular function on some affine modular curve $Y$ over $\overline{\Q}$ which admits a Galois finite \'etale map to $Y(2)_{\overline{\Q}}$, a minimal-order nonzero linear differential operator $L$ over $\overline{\Q}(\lambda)$ with~$L(f)= 0$ has trivial local monodromies around any~$\lambda \neq 0, 1, \infty$. Indeed, by the minimality of~$L$, this amounts to analytically continuing the algebraic function~$f(\lambda) \in \C\llbracket \lambda^{1/N}\rrbracket$ along all paths in~$Y(2)_{\C} = \C \setminus \{0,1\}$; this is for instance since, by the lifting property for covering maps, every holomorphic map~$D(0,1) \to Y(2)_{\C} = \spec \C[\lambda,1/\lambda, 1/(1-\lambda)]$ based at~$0 \mapsto y_0 \in Y(2)_{\C}$ lifts to a holomorphic map~$D(0,1) \to Y_{\C}$ based at an arbitrary fiber point of~$y_0$ under the covering~$Y \to Y(2)$.  %and all singularities $\lambda \neq 0,1, \infty$ of $L$ in $\C \setminus \{0,1\}$ have trivial local monodromy. (Since $L$ is a minimal linear differential operator with finite global monodromy group, all elements of $\ker(L)$ are regular functions on some affine modular curve admitting a finite \'etale map to $Y(2)_{\overline{\Q}}$, and thus the local monodromies at $\lambda \neq 0,1, \infty$ are trivial.)
Moreover, our assumption on the cusp widths dividing $2N$ implies that a local coordinate in a small neighborhood of each cusp of $Y$ above $\lambda=0$ can be chosen to be the lift of some (positive integer) power of $x=(\lambda/16)^{1/N}$. This means that the pullback of $L$ to $U\setminus\{0\} = \C^{\times} \setminus 16^{-1/N} \mu_N$  admits a full set of meromorphic solutions in some sufficiently small neighborhood of $x=0$, i.e. has a trivial local monodromy around $x = 0$.  Therefore $f\in \mathcal{H}(U,x(t),\Z)\otimes_{\Z}\Q$, and $R_{2N}\subset \mathcal{V}(U,x(t),\Z)$. 

 It thus suffices to bound $\dim_{\Q(x^N)}\mathcal{V}(U,x(t),\Z)$ by $CN^3\log{N}$.
We take $r := 1 - AN^{-3}/2$ and
$$
\varphi(z) := 16^{-1/N}F_N(rz) \quad  : \quad \overline{D(0,1)} \to U.
$$

By assumption~\ref{strategy}(\ref{radius}) of  Proposition~\ref{strategy} and the choice of radius $r = 1  - AN^{-3}/2$, we have
\numequation \label{deg}
 \log{|\varphi'(0)|}
> \log{(1+A/N^3)}  + \log{r}  
= A N^{-3} / 2 + O_{A}(N^{-6}).
\end{equation}

Thus, with $c := A/3$, we get for $N \gg 1$ sufficiently large that
\numequation \label{redeg}
\log{|\varphi'(0)|} > c N^{-3}.
\end{equation}
Corollary~\ref{holonomy} now  gives the upper bound
\numequation
\label{extrastep}
\dim_{\Q(x^N)} \mathcal{V}(U,x(t),\Z) \leq e \cdot \frac{\int_{|z| = 1 - A/(2N^3)} \log^+{|F_N^N|} \, \mu_\mathrm{Haar}}{cN^{-3}}. 
\end{equation}
From assumption~(\ref{mean})
of Proposition~\ref{strategy} (with the choice $B := A/2$) together with the identity $\log^{+}|F_N^N| = N \log^{+} |F_N|$,
we have
$$
 e \cdot \frac{\int_{|z| = 1 - A/(2N^3)} \log^+{|F_N^N|} \, \mu_\mathrm{Haar}}{cN^{-3}}
   \ll_B e \cdot  \frac{\displaystyle{ N \cdot \frac{\log{N}}{N} }}{ c N^{-3}} = O(N^3 \log N),
   $$
   which, combined with equation~(\ref{extrastep}), is the desired upper bound.
\end{proof}

\begin{remark}
We may also prove this proposition by using Theorem~\ref{abstract form} directly. Using the notation in the proof
of Proposition~\ref{strategy}, let $Y'$ denote the modular curve $Y$ with all the cusps above $0\in Y(2)\cup\{0\}$ filled in. The fiber product $Y'\times_{Y(2)_{\overline{\Q}}\cup\{0\}} U$ with its natural map to $U$ is a covering map (one can check this claim locally; the assumption on cusp widths is used to prove that $0\in U$ is not ramified). Therefore, the universal covering map $D(0,1)\rightarrow U$ factors through $Y'\times_{Y(2)_{\overline{\Q}}\cup \{0\}} U$ and thus we obtain a map $D(0,1)\rightarrow Y'\times_{Y(2)_{\overline{\Q}}\cup \{0\}} U\rightarrow Y'$ (the second map is the natural map) such that its composition with $Y'\rightarrow Y(2)_{\overline{\Q}}\cup \{0\}$ is the map $D(0,1)\rightarrow U \rightarrow Y(2)_{\overline{\Q}}\cup \{0\}$. Thus $f\circ \varphi$ is also given by the natural pullback of $f$ from $Y'$ to $D(0,1)$ and thus it is holomorphic over $D(0,1)$ as far as $f$ is holomorphic at all cusps in $Y'$, which can be achieved by multiplying $f$ with a suitable power of $\lambda$. Thus we verify the analyticity property in Theorem~\ref{abstract form}. The rest of the proof is the same as  above.
\end{remark}

\subsection{A guide to the  proof of the main theorem}

We prove both of the assumptions of Proposition~\ref{strategy} hold in
Theorems~\ref{exactconformal} and
Theorem~\ref{growth term} respectively. This provides a~$CN^3 \log{N}$ dimension bound 
for the vector space $R_{2N}$ of all modular functions against the obvious $\gg N^3$ lower bound for the subring of the \emph{congruence} examples from the
fact that~$[\Gamma(2):\Gamma(2N)] \gg N^3$ (see equation~\ref{zeta2}).
We then need to provide an additional argument to overcome this ``small error''
(a logarithmic gap $O(\log{N})$ in every level~$N$) between the lower
and upper bounds. 

The following is  a guide to what we do in the next few sections of our paper:
 \begin{enumerate}
 \item In~\SSSS~\ref{serresection}, we prove that the logarithmic gap between  the ring of modular forms
 with bounded denominators and the ring of congruence modular forms 
  can be leveraged
 to prove the full unbounded denominators conjecture. The main idea here is that given
 a noncongruence modular form~$f(q) \in \Z  \llbracket q^{1/N} \rrbracket$, one can construct
 many more 
 such forms independent over the ring of congruence forms by considering~$f(q^{p}) \in \Z\llbracket q^{1/N} \rrbracket$ for primes~$p$. 

 \item In~\SSSS~\ref{uniformizations}, we study the properties of the function~$F_N$.
 It turns out more or less to be related to a Schwarzian automorphic function on a
 (generally non-arithmetic) triangle group. This allows us to compute the conformal radius of~$F_N$
 exactly (see Theorem~\ref{exactconformal}), and indeed it has the form~$16^{1/N}\big(1 + (\zeta(3)/2) N^{-3}+ \cdots \big)$.

\item In~\SSSS~\ref{uniformizations}, we also study the maximum value of $|F_N|$ on the
circle $|z| = R$, uniformly in both $N$ and $R <  1$. The main idea here is that a normalized variant function~$G_N(q) = F_N(q^{1/N})^N$
``converges'' to the modular~$\lambda$ function~$\lambda(q) = 16 q - 128 q^2 + \cdots$.
Approximating the region where~$F_N$ is large by the corresponding region for~$\lambda(q)$
one predicts a growth rate of the desired form. However, the problem is that the convergence
of~$G_N(q)$ to~$\lambda(q)$ is not in any way uniform, especially in the neighbourhoods of the cusps
of~$F_N$ which certainly vary with~$N$.

\item In~\SSSS~\ref{nevanlinna}, we  solve this uniformity problem on the abstract grounds of Nevanlinna theory. We
combine the crude growth bound on $|F_N|$ with a version of Nevanlinna's lemma on the logarithmic derivative to prove our requisite uniform upper estimate on the mean proximity function $m(r,  F_N) = \int_{|z| = r} \log^+{|F_N|} \, \mu_{\mathrm{Haar}}$.

\item Putting all the pieces together, the proof of Theorem~\ref{theorem:main} is then completed in \SSSS~\ref{putting together}.

\end{enumerate}

The following leitfaden gives an abbreviated summary of how the argument is laid out:

\begin{center}
\begin{tikzpicture}[node distance = 1.5cm, auto]
\node (A){\TS\ref{putting together}};
\node (R)[below of=A]{\ref{theorem:main}};
\node(B)[above of =A, left of =A]{\ref{growth term}};
\node(C)[above of = A, node distance =3cm,  right of =A]{\ref{exactconformal}};
\node(Z)[above of = A, node distance =7.5cm]{\ref{holonomy}};
\node(S)[left of = Z, node distance = 2cm]{\ref{abstract form}};
\node(Y)[right of = C]{};
\node(W)[above of = Y, right of = Y]{};
\node(X)[right of = W]{\ref{selfcontained}};
\node(V)[left of = X, node distance = 2cm]{\ref{fix}};
\node(U)[left of = V, node distance = 2cm]{\TS\ref{amalgams}};
\node(T)[left of = U, node distance = 2cm]{\ref{enhanced}};
\node(D)[ left of =B]{\TS\ref{proof resumed}};
\node(Q)[above of=D, left of=D]{\ref{trivialsupFN}};
\node(E)[above of=A, node distance =6cm]{\ref{strategy}};
\draw[-] (A) -- (B);
\draw[-] (A) -- (X);
\draw[-]  (B) -- (D);
\draw[-] (A) -- (E);
\draw[-] (A) -- (C);
\draw[-] (Z) -- (E);
\draw[-] (V) -- (X);
\draw[-] (U) -- (V);
\draw[-] (T) --( U);
\draw[-] (S) --(Z);
\draw[-] (R) --(A);
\draw[-] (Q) --(D);
\end{tikzpicture}
\end{center}

\section{Noncongruence forms}
\label{serresection}

\subsection{Wohlfahrt Level}

We begin by recalling a notion of level for noncongruence subgroups
due to Wohlfahrt~\cite{Wohlfahrt}.
Let~$G \subset \SL_2(\Z)$ be a finite index subgroup. (Many of the arguments of
this section do not require this hypotheses but since it is  satisfied for our applications
we assume it to avoid unnecessary distractions.)
The group consisting of the two matrices $\pm  I$,
where~$I =  \left( \begin{matrix} 1 & 0 \\ 0 & 1 \end{matrix} \right)$,
will be denoted by~$E$.
The group~$\SL_2(\Z)$ acts via M\"{o}bius transformations both on the upper half plane~$\H$ and the extended upper half 
plane~$\H^* = \H \cup \mathbf{P}^1(\Q)$. 
The action of~$\SL_2(\Z)$ on~$\mathbf{P}^1(\Q)$ is transitive. It follows that 
if a nontrivial element~$\gamma \in G$  fixes an element~$\zeta \in \mathbf{P}^1(\Q)$,
then~$\gamma$ has the form~$\pm M U^m M^{-1}$, where~$M  \in \SL_2(\Z)$, $M  \infty = \zeta$, and
\numequation
\label{levelN}
U = \left( \begin{matrix} 1 & 1 \\ 0 & 1 \end{matrix} \right).
\end{equation}
We call such a~$\zeta \in \mathbf{P}^1(\Q)$ a \emph{cusp} of~$G$.
If~$\zeta \in \mathbf{P}^1(\Q)$, then~$ M U^m M^{-1}  = (M U M^{-1})^m \in G$ for some~$m$ because~$G$
has finite index in~$\SL_2(\Z)$, and hence every element of~$\mathbf{P}^1(\Q)$ is a cusp of~$G$. 
The stabilizer in~$G$ of a cusp~$\zeta \in \mathbf{P}^1(\Q)$ is either isomorphic to~$E \times \Z = \Z/2 \Z \times \Z$ or~$\Z$,
depending on whether~$E \subset G$ or not. For each~$\zeta$, there is a minimal positive integer~$m$ such
that~$\pm M U^m M^{-1} \in G$, and we say that~$m$ is the \emph{width} of the cusp~$\zeta$. The action of~$G$ on~$\mathbf{P}^1(\Q)$
has finitely many orbits, and the cusp width only depends on the orbit of the cusp under~$G$.
Geometrically, the complex structure on~$\H$ imbues the quotient~$X(G) = \H^*/G$ with the structure of an algebraic curve.
From this point of view,  the equivalence classes
of cusps of~$G$ (up to the action of~$G$) are in bijection with  the pre-images of~$\infty$ under the 
projection~$X(G) \rightarrow X(\SL_2(\Z)) = \mathbf{P}^1_j$, and the cusp widths are exactly the ramification indices
of this map at~$j = \infty$.

\begin{df}[\cite{Wohlfahrt}] 
The level~$L(G)$ of~$G$ is  the lowest common multiple of all the cusp widths of~$G$.
\end{df}

We begin with some elementary properties concerning this definition.
We typically only consider groups containing~$E = \langle -I \rangle$ since we are  generally interested
in stabilizers of functions under M\"obius transformations.

\begin{lemma} \label{intersect}
Let~$G$ and~$H$ be finite index subgroups of~$\SL_2(\Z)$ both containing~$E$. Suppose that~$L(G)$ and~$L(H)$  both divide~$N$.  Then any cusp of~$G \cap H$ also has cusp width dividing~$N$.
\end{lemma}

\begin{proof}  The stabilizer of a cusp inside any subgroup of~$\SL_2(\Z)$ containing~$E$ is~$E \times  \Z$.
In particular, if~$G$ contains the group~$E \times a \Z$ and~$H$ contains~$E \times b \Z$ then~$G \cap H$ contains~$E \times \mathrm{lcm}(a,b) \Z$,
and the result follows.
\end{proof}

\begin{lemma} \label{normallevel}
Let~$G \subset \SL_2(\Z)$ be a finite index subgroup containing~$E$ with Wohlfahrt level~$N$.
Let~$N(G)$ be the largest normal subgroup of~$\SL_2(\Z)$ contained in~$G$.
Then~$N(G)$  has finite index in~$\SL_2(\Z)$ and~$L(N(G)) = N$.
\end{lemma}

\begin{proof}  Since~$G$ has finite index in~$\SL_2(\Z)$, the group~$N(G)$ is the intersection of the finitely many conjugates
of~$G$ by~$\SL_2(\Z)$. Hence~$N(G)$ has finite index and~$L(N(G))=N$ by Lemma~\ref{intersect}.
\end{proof}

\begin{note} Let~$A$ denote the following matrix:
\numequation
\label{defineA}
\displaystyle{A:=  \left( \begin{matrix} p & 0 \\ 0 & 1 \end{matrix} \right)}.
\end{equation}
\end{note}
(We use this notation so as to be consistent with
that of Serre in~\cite{SerreThompson} which we follow below.)
We now prove the following lemma concerning how the level of a subgroup changes under conjugation
by~$A$.

\begin{lemma}  \label{timesp} Let~$H \subset \SL_2(\Z)$  be a finite index subgroup containing~$E$ such that~$L(H) = N$.
Then~$L(A^{-1} H A \cap \SL_2(\Z))$ divides~$Np$.
\end{lemma}

\begin{proof}
Let us write~$\wH:=A^{-1} H A \cap \SL_2(\Z)$. Note that~$\wH$ contains~$A^{-1}EA = E$.
In particular, the stabilizer of any cusp of~$\wH$ has the form~$E \times \Z$, where the~$\Z$
is generated by a unipotent element~$\wh$ of~$\wH$ conjugate in~$\SL_2(\Z)$
to~$U^m$ for some positive integer~$m$,
and we want to show that~$m$ divides~$Np$.

Any unipotent element~$\wh$ in~$\wH$ has the form~$\wh = A^{-1} h A$ for some unipotent
element~$h \in H$. The element~$h \in H$ will stabilize some cusp~$\zeta \in \mathbf{P}^1(\Q)$.
The stabilizer of~$\zeta$ in~$H$ has the form~$E \times \Z$ where~$\Z$ is generated by a
unipotent element~$\gamma \in H$. It follows that~$h$ will be the smallest power of~$\gamma$ which lies
in~$\wH$, or equivalently in~$\SL_2(\Z)$.
Since~$L(H) = N$, we may write
$$\gamma = B \left( \begin{matrix} 1 & n \\ 0 & 1 \end{matrix} \right) B^{-1}$$
with~$n|N$, and~$\displaystyle{B = \left( \begin{matrix} a & b \\ c & d \end{matrix} \right) \in \SL_2(\Z)}$.
We  define
$$\begin{aligned}
\wgamma:= A^{-1} \gamma A = \ &
A^{-1} B   \left( \begin{matrix} 1 & n \\ 0 & 1 \end{matrix} \right)  B^{-1} A \\
= & \  (A^{-1} B A) \left( A^{-1}  \left( \begin{matrix} 1 & n \\ 0 & 1 \end{matrix} \right)  A \right)  (A^{-1} B^{-1} A)  \\
=  & \ (A^{-1} B A)  \left( \begin{matrix} 1 & n/p \\ 0 & 1 \end{matrix} \right)  (A^{-1} B A)^{-1}, \end{aligned}
$$
where~$\displaystyle{ (A^{-1} B A)  =  \left( \begin{matrix} a & b/p \\ cp & d \end{matrix} \right)}$.
Since~$h$ is a power of~$\gamma$, we deduce that~$\wh = A^{-1} h A$ is a power of~$\wgamma$,
although~$\wgamma$ need not be in~$\wH$ since it is not necessarily integral.
We consider two cases:
\begin{enumerate}
\item
Suppose that~$(a,p) = 1$. 
Since~$(a,c) = 1$ we have~$(a,pc) = 1$, and thus there exist~$r,s \in \Z$ 
with~$a s - r p c = 1$, and hence 
$$C = \left( \begin{matrix} a & r \\ pc & s \end{matrix} \right) \in \SL_2(\Z).$$
For such a~$C$, we have, with~$t = b s - d p r$, the identity
$$A^{-1} B A = C  \left( \begin{matrix} 1 & t/p \\ 0 & 1 \end{matrix} \right).$$
But since~$\left( \begin{matrix} 1  & t/p \\ 0 & 1 \end{matrix} \right)$ commutes with~$\left( \begin{matrix} 1  & n/p \\ 0 & 1 \end{matrix} \right)$,
it follows that we may write
$$ \wgamma =
(A^{-1} B A)  \left( \begin{matrix} 1 & n/p \\ 0 & 1 \end{matrix} \right)  (A^{-1} B A)^{-1} 
= C  \left( \begin{matrix} 1 & n/p \\ 0 & 1 \end{matrix} \right)  C^{-1}.$$
We now have
$$(\wgamma)^p = C  \left( \begin{matrix} 1 & n \\ 0 & 1 \end{matrix} \right)   C^{-1} \in \SL_2(\Z)$$
 and thus~$(\wgamma)^p$ is in~$\wH$.
 Hence either~$\wh = \wgamma$ if~$\wgamma$ lies in~$\SL_2(\Z)$ or~$\wh =  (\wgamma)^p$.
 In  particular, the cusp width at this cusp is either~$n$ or~$n/p$ and certainly divides~$N$ and hence also~$Np$.
\item Suppose that~$p | a$, so~$p$ does not divide~$c$, so~$a/p$ and~$c$ are  co-prime integers.
 Now take
$$ C = \left( \begin{matrix} a/p & b \\ c & pd \end{matrix} \right)=(A^{-1}BA)\left( \begin{matrix} p & 0 \\ 0 & 1/p \end{matrix} \right) \in \SL_2(\Z).$$
Then
$$  \wgamma = 
(A^{-1} B A)  \left( \begin{matrix} 1 & n/p \\ 0 & 1 \end{matrix} \right)  (A^{-1} B A)^{-1} 
= C  \left( \begin{matrix} 1 & np \\ 0 & 1 \end{matrix} \right)  C^{-1},$$
and hence~$\wh = \wgamma$ and the cusp width at this cusp is~$np$ which divides~$Np$. \qedhere
\end{enumerate}
\end{proof}

\subsection{Modular Forms}

For an  integer~$N$, we will consider the following spaces of modular functions
with rational coefficients generated by forms with bounded denominators,
 that is, subspaces of~$\Q \llp q^{1/N} \rrp = \Q \llbracket q^{1/N} \rrbracket [1/q]$ 
 (with~$q = e^{\pi i \tau}$) generated by elements
of~$\Z \llbracket q^{1/N} \rrbracket \otimes \Q$ as~$\Q(\lambda)$-vector spaces.
\begin{df} \label{rings}  \leavevmode
\begin{enumerate}
\item Let~$M_{2N}$ denote the~$\Q(\lambda)$-vector space generated by holomorphic modular functions on the modular curve~$Y(2N) = \H/\langle E,\Gamma(2N)\rangle$
with coefficients in~$\Q$ at the cusp~$\zeta = i \infty$.
\item Let~$R_{2N}$ denote the~$\Q(\lambda)$-vector space  generated by holomorphic modular functions
with coefficients in~$\Q$,  bounded denominators at the cusp $\zeta = i \infty$, and cusp widths dividing~$2N$ at all cusps $\zeta \in \mathbf{P}^1(\Q)$.
\end{enumerate}
\end{df}

(The vector space~$R_{2N}$ was also defined in Proposition~\ref{strategy} but we repeat the definition here for convenience.)
For example, the (weight~$0$) holomorphic modular forms on~$Y(2)$ are given by~$\Q[\lambda, 1/\lambda, 1/(1-\lambda)]$, and the~$\Q(\lambda)$-vector
space generated by such elements inside~$\Q \llp q \rrp $ with~$q = e^{\pi i \tau}$ is~$M_2 = \Q(\lambda)$.

\begin{lemma}  \label{unrefined} There is a containment~$M_{2N} \subset R_{2N}$, and~$M_{2N}$ and~$R_{2N}$ have finite dimensions over~$M_2 = \Q(\lambda)$.
\end{lemma}

\begin{proof} 
Let~$f$ be a holomorphic modular function on~$Y(2N)$, that is, a meromorphic function
on the compact modular curve~$X(2N)$ whose poles are all at the cusps. Assume also that~$f$ has
 coefficients in~$\Q$ at the cusp~$\zeta = i \infty$, Then the modular form~$f \Delta(\tau)^m$ is holomorphic at the cusps for
 sufficiently large~$m$. Moreover, $f \Delta(\tau)^m$ has coefficients in~$\Q$. It follows 
 from~\cite[Theorem~3.52]{Shimura} that~$f \Delta(\tau)^m$ has bounded denominators. Since~$\Delta^{-1} \in q^{-2} \Z \llbracket q \rrbracket$ has integral coefficients,  it follows that~$f$ also has bounded denominators,  and thus there is a containment~$M_{2N} \subset R_{2N}$.
 
The second claim follows from Corollary~\ref{holonomy} and the remark (cf. the second paragraph
of~\SSSS~\ref{sec_overconv}) that the conformal radius of $\C \setminus 16^{-1/N}\mu_N$ is strictly larger than~$1$. Indeed, 
$$\sqrt[N]{\lambda(z^N)/16} : D(0,1) \to \C \setminus 16^{-1/N}\mu_N$$
 is a well-defined holomorphic map with unit derivative at the origin, and hence
by Schwarz's lemma the universal covering~$D(0,1) \to \C \setminus 16^{-1/N}\mu_N$ has derivative strictly larger than~$1$ in absolute value. (Later, in Theorem~\ref{exactconformal} below, we will exactly compute this latter derivative.)
\end{proof}

We have the following refinement of Lemma~\ref{unrefined}:

\begin{lemma} \label{finiteness} 
The vector spaces~$M_{2N}$ and~$R_{2N}$ are fields.
The space~$M_{2N}$ may be identified with the field of rational functions on the modular curve~$Y(2N)/\Q$. 
There are injective algebra maps
$$M_2 \rightarrow M_{2N} \rightarrow R_{2N}.$$
The space~$R_{2N}$ is invariant under  a normal finite index subgroup~$G_{2N} \subset \langle E, \Gamma(2N)\rangle \subset  \SL_2(\Z)$ containing~$E$ 
with~$L(G_{2N}) = 2N$.
\end{lemma}

\begin{proof} 
Note that~$M_{2N}$ and~$R_{2N}$ are subspaces of~$\Q \llp q^{1/N} \rrp$, which is a domain.
Hence if~$M_{2N}$ and~$R_{2N}$ are rings then they are also integral domains, and any integral  domain which has finite dimension over a field is itself a field.

The curve~$Y(2N)$ has a standard model over~$\Q$ (as a moduli space of elliptic curves~$C$ with a given symplectic isomorphism~$C[2N] \simeq \Z/2N\Z \oplus \mu_{2N}$)
such that the cusp~$\zeta = i  \infty$ is  defined over~$\Q$, and the action of~$\Gal(\Qbar/\Q)$ on the global sections of~$Y(2N)$
is compatible with the~$q$-expansion map.  It follows that the set of generators of~$M_{2N}$ is closed under addition and multiplication and hence that~$M_{2N}$
is a ring, and thus a field. Moreover, $M_{2N}$ contains the global sections of the (affine) curve~$Y(2N)/\Q$, and hence~$M_{2N}$
must be the function field of~$Y(2N)/\Q$.

The vector space~$R_{2N}$ is generated by holomorphic modular forms with bounded denominators at~$\zeta = i \infty$. To show~$R_{2N}$ is a ring, it suffices to show that the
product of any two such generators~$g$ and~$h$ is also a generator. Certainly~$gh$ is a holomorphic modular form  with rational coefficients and  bounded denominators, so it suffices to show that the cusp width still divides~$2N$.
But we may assume that~$g$ and~$h$ are invariant under finite index subgroup~$G,H \subset \SL_2(\Z)$ containing~$E$,
and thus~$gh$ is invariant under~$G \cap H$. 
It follows from Lemma~\ref{intersect} that~$L(G \cap H)$ also has  Wohlfahrt level dividing~$2N$.

Since~$R_{2N}$ is finite over~$M_2$, it is generated by a finite number of  basis elements each of which is invariant under some finite index subgroup~$\Phi \subset \SL_2(\Z)$ containing~$E$ with~$L(\Phi)$ dividing~$2N$.
The intersection of all these groups still has finite index and level~$2N$ by Lemma~\ref{intersect},
and then we take~$G_{2N}$ to be the largest normal subgroup of~$\SL_2(\Z)$ contained in this intersection, which also has~$L(G_{2N})=2N$
by Lemma~\ref{normallevel}.
\end{proof}

\subsection{A leveraging argument}
\label{leveraging}

Let us assume that there exists an~$N$ such that~$R_{2N}$ is strictly larger than~$M_{2N}$.
Let~$f(\tau)  \in \Z  \llbracket q^{1/N} \rrbracket \in R_{2N}$ be an element which does not lie in~$M_{2N}$.
Recall that all  forms in~$R_{2N}$ and thus in particular~$f$ is invariant by a subgroup~$G=G_{2N} \subset \langle E, \Gamma(2N) \rangle$
which is normal with finite index in~$\langle E,\Gamma(2N)\rangle$ and has~$L(G_{2N}) = 2N$ by Lemma~\ref{finiteness}.
The main idea of this section is to exploit the fact that~$f(p \tau) \in \Z \llbracket q^{1/N} \rrbracket$ is also a modular form with integer coefficients for any prime~$p$.
Since the form~$f(\tau)$ is invariant under~$G$, the form~$f(p \tau)$
is invariant under~$A^{-1}GA$  and thus also the group~$A^{-1}GA \cap \SL_2(\Z)$.
Now, by Lemma~\ref{timesp}, we know that this group has 
(Wohlfahrt)
level dividing~$2Np$. In particular~$f(p \tau)$ has cusp width dividing~$2Np$ at each cusp,
and hence~$f(p \tau) \in R_{2Np}$.

Our main result is as follows:
\begin{theorem} \label{fix}
Suppose that~$(p,2N)=1$ is  prime.
Suppose that~$f(\tau) \in R_{2N}$ is \emph{not} invariant under a congruence subgroup.
Then the form~$f(p \tau)$ is not in the~$M_{2Np}$-algebra generated by~$R_{2N}$.
\end{theorem}

That is, we can leverage one exception to the unbounded denominators
to produce many examples. 
Before proving Theorem~\ref{fix} (whose proof is deferred to~\SSSS~\ref{amalgams}), we first draw the following consequence:

\begin{theorem} \label{theoremgoal}
Let~$p$ be a prime not dividing~$2N$.
Suppose %that there exists an integer~$N$ such 
that~$[R_{2N}:M_{2N}] > 1$.
Then one has
$$[R_{2Np}:M_{2Np}] \ge 2 [R_{2N}:M_{2N}].$$
\end{theorem}

\begin{proof}
Let~$f(\tau)$ be a form in~$R_{2Np}$ which is not in~$M_{2N}$.
By Theorem~\ref{fix},
we deduce
that~$f(p \tau) \in R_{2Np}$ is not in the~$M_{2Np}$-algebra
$R_{2N} M_{2Np}$ 
generated by~$R_{2N}$ (which is a subfield of~$R_{2Np}$). We have
\begin{equation*}
\begin{aligned}
[R_{2Np}:M_{2N}] = & \  [R_{2Np}:R_{2N} M_{2Np}] [R_{2N} M_{2Np}:M_{2N}] \\
= & \ [R_{2Np}:R_{2N} M_{2Np}]  [R_{2N}:M_{2N}] [M_{2Np}:M_{2N}], \end{aligned}
\end{equation*}
because the intersection of~$M_{2Np}$ and~$R_{2N}$ is~$M_{2N}$. Thus
$$\frac{[R_{2Np}:M_{2Np}] }{[R_{2N}:M_{2N}]} =   [R_{2Np}:R_{2N} M_{2Np}]$$
is an integer which is~$\ge 2$, 
which implies
 Theorem~\ref{theoremgoal}.
\end{proof}

Our goal is to prove that~$R_{2N}= M_{2N}$. 
%We can and do assume that~$N$ is even. 
As noted in Lemma~\ref{unrefined}, $[R_{2N}:M_2] < \infty$. 
The degree of~$M_{2N}$ over~$M_2$ is equal to the degree of the modular curve~$Y(2N)$ over~$Y(2)$, and this is given,
for~$N > 1$, 
by the explicit formula
\numequation
\label{zeta2}
\begin{aligned}
[M_{2N}:M_2] =
\frac{1}{2}[\Gamma(2):\Gamma(2N)] = & \  \frac{(2N)^3}{2[\SL_2(\Z):\Gamma(2)]} \prod_{p|2N} \left(1 - \frac{1}{p^2} \right)  \\
> & \  \frac{(2N)^3}{2[\SL_2(\Z):\Gamma(2)]} \prod_{p} \left(1 - \frac{1}{p^2} \right) = \frac{2 N^3}{3 \zeta(2)}
= \frac{4 N^3}{\pi^2}. \end{aligned}
\end{equation}
The factor of~$1/2$ comes from the fact that~$E = -I \in \Gamma(2)$ and the degree of~$Y(2N)$ over~$Y(2)$
 is the index of the images of these groups inside~$\PSL_2(\Z)$.
Since~$[R_{2N}:M_2] = [R_{2N}:M_{2N}] \cdot [M_{2N}:M_2]$, 
it follows that we have a bound:
\numequation
\label{goal}
[R_{2N}:M_{2N}] \le  \frac{\pi^2[R_{2N}:M_2]}{4 N^3}
\end{equation}
for all~$N$.  We can now compare this bound against the one
coming from Theorem~\ref{theoremgoal}.

\begin{proposition} \label{selfcontained}
Suppose that there exists a constant~$C$ and a bound
$$[R_{2N}:M_2] \le C N^3 \log N$$
for all  integers~$N$. Then~$R_{2N} = M_{2N}$ for every~$N$, that is,
the unbounded denominators conjecture holds.
\end{proposition}

\begin{proof}
Assume there exists an~$N$ such that~$R_{2N} \ne M_{2N}$.
Let~$S$ denote the set of primes $< X$ which are co-prime to~$2N$.
By induction, Theorem~\ref{theoremgoal}  implies for such an~$N$ that,
for any~$\varepsilon > 0$,
\numequation
\label{lowerbound}
[R_{2N\prod_{p\in S} p}:M_{2N\prod_{p\in S}p}] \ge 2^{\#S}
> 2^{(1 - \varepsilon) X/\log X},
\end{equation}
for sufficiently large~$X$ (depending on~$N$ and~$\varepsilon$) by the prime number theorem.
 The right-hand side
 certainly increases faster than any power of~$X$.
On the other hand, from the assumed bound on~$[R_{2N}:M_2]$
together with the bound~(\ref{goal}), we obtain
\numequation
\label{upperbound}
\begin{aligned}
 \relax [R_{2N\prod_{p\in S}p}:M_{2N\prod_{p\in S}p}]   \le  & \ 
\frac{C \pi^2}{4} \cdot  \log \left(2N\prod_{p\in S}p\right) \\
= & \ \frac{C \pi^2}{4} \cdot  \log 2N + \frac{C \pi^2}{4}  \sum_{p\in S} \log p <  \frac{C \pi^2}{4}   \cdot  X(1 + \varepsilon),
\end{aligned}
\end{equation}
where the last inequality follows  (with the same~$\varepsilon > 0$) once more from the prime number theorem
for sufficiently large~$X$.
Combining the bounds~(\ref{lowerbound})  and~(\ref{upperbound})  gives, for
all sufficiently large~$X$,
$$ 2^{(1 - \varepsilon) X/\log X} < \frac{C \pi^2}{4} \cdot   X(1 + \varepsilon),$$
which (by some margin!)  is a contradiction for any fixed~$\varepsilon < 1$.
\end{proof}

\begin{remark} \label{expansion rate} The argument still works with a bound weaker 
than~$[R_{2N}:M_2] \ll N^3 \log N$, although~$[R_{2N}:M_2] \ll N^{3 + \varepsilon}$ would not be
strong enough.
\end{remark}

\subsection{Amalgams and a non-abelian version of Ihara's Lemma}
\label{amalgams}
In~\SSS~\ref{leveraging}, we introduced a group~$G = G_{2N}  \subset \langle E,\Gamma(2N) \rangle$
which was normal with finite index and had~$L(G_{2N}) = 2N$.
In this section, we consider more generally (up to a notational shift)
a group~$G = G_N  \subset \langle E,\Gamma(N) \rangle$
which is normal of finite index and with~$L(G_{N}) = N$, and then apply our results to the particular
group~$G$ of~\SSS~\ref{leveraging}  when we prove of Theorem~\ref{fix}.
(See equations~(\ref{serre}) and~(\ref{serreagain}) and the surrounding discussion.)

Since~$G = G_N$ is normal and is contained in~$\langle E,\Gamma(N)\rangle$, 
we may define a group~$S$ by taking~$S = \langle E,\Gamma(N) \rangle/G$.
By construction, the group~$S$
is finite.  There is a natural projection:
$$f:\langle E,\Gamma(N) \rangle \rightarrow \langle E,\Gamma(N) \rangle/G = S.$$
We define
 two homomorphisms~$f_1$ and~$f_2$ from~$\langle E,\Gamma(N) \rangle  \cap \Gamma_0(p)$ to~$S$ as follows:
\begin{enumerate}
\item The map~$f_1$ is the restriction of~$f$ to~$\langle E,\Gamma(N) \rangle  \cap \Gamma_0(p)$ under the natural
inclusion
$$\langle E,\Gamma(N) \rangle  \cap \Gamma_0(p) \rightarrow \langle E,\Gamma(N) \rangle,$$
so~$f_1(x) = f(x)$.
\item Conjugation by~$A$ induces an isomorphism
$$ \langle E,\Gamma(N) \rangle  \cap \Gamma_0(p) \rightarrow  \langle E,\Gamma(N) \rangle  \cap \Gamma^0(p), \quad
\gamma \longrightarrow A \gamma A^{-1}. $$
The map~$f_2$ is the composition of this map composed with~$f$, so~$f_2(x) = f(AxA^{-1})$. 
\end{enumerate}

\begin{lemma}[Serre, Berger] \label{serreberger}
The map~$(f_1,f_2): \langle E, \Gamma(N)\rangle \cap \Gamma_0(p) \rightarrow S \times S$
is surjective.
\end{lemma}

This is more or less precisely~\cite[Theorem~3]{SerreThompson}
with the addition of
level structure  as in~\cite{Berger}. Ihara's Lemma~\cite{RibetICM} is (informally) the statement
that the two maps~$H^1(\Gamma,\F_q) \rightarrow H^1(\Gamma_0(p),\F_q)$ coming from the restriction map
and (respectively) the restriction map conjugated by~$A$ have images which are as disjoint as possible.
One may think of Lemma~\ref{serreberger}
as a non-abelian version of Ihara's Lemma, because (as explained below
in the proof of Lemma~\ref{ihararibet}) the
case when~$S$ is a vector space over~$\F_q$ reduces precisely to the statement
of Ihara's Lemma as proved by Ribet~\cite{RibetICM}. (The proofs of both
claims are very similar.)

\begin{proof} 
The intersection of~$\langle E,\Gamma(N) \rangle$ with~$A \langle E,\Gamma(N) \rangle A^{-1}$ is
the group~$\langle E,\Gamma(N) \rangle \cap \Gamma_0(p)$. 
We proceed by contradiction.
Assume that the map~$(f_1,f_2)$ is not surjective. By Goursat's lemma, there exists a nontrivial 
quotient~$\Delta$ of~$S$  and projections~$\pi_i: S \rightarrow \Delta$ such that the composites~$\pi_1 \circ f_1$ 
and~$\pi_2 \circ f_2$ agree.  
We define a map~$g_1$ by the composite
\numequation
 \xymatrix{g_1:   \langle  E,\Gamma(N)\rangle \ar[r]^-{f} & S \ar[r]^{\pi_1} & \Delta} 
 \end{equation}
and a map~$g_2$ by the composite
\numequation
 \xymatrix{g_2:   A^{-1} \langle  E,\Gamma(N)\rangle A
 \ar[r] & \langle  E,\Gamma(N)\rangle \ar[r]^-{f} & S \ar[r]^{\pi_2} & \Delta} 
 \end{equation}
 where the first map sends~$x \rightarrow A x A^{-1}$. On the intersection
 $$\langle  E,\Gamma(N) \rangle \cap  A^{-1} \langle  E,\Gamma(N)\rangle A
 = \langle  E,\Gamma(N) \rangle  \cap \Gamma_0(p),$$
 the restriction of~$g_1$ is given by~$\pi_1 \circ f_1$ and the restriction of~$g_2$
 is given by~$\pi_2 \circ f_2$. By construction these maps coincide, and hence
 they induce a surjective map on the amalgam
 $$\Phi:=\langle E,\Gamma(N) \rangle  \star_{\langle  E,\Gamma(N) \rangle  \cap \Gamma_0(p)} 
  A^{-1} \langle  E,\Gamma(N)\rangle A \rightarrow \Delta.$$
There are natural inclusions from~$\langle E,\Gamma(N) \rangle $ and~$ A^{-1} \langle  E,\Gamma(N)\rangle A$ 
to the  congruence subgroup of~$\SL_2(\Z[1/p])$ consisting of matrices congruent to~$\pm I \bmod N$, and these inclusions
induce a map from~$\Phi$ to this congruence subgroup.
This map is an isomorphism (\cite[p.919]{Berger}, using ideas of~\cite{Trees} and following the proof of~\cite[Theorem~3]{SerreThompson}).
But the group~$\SL_2(\Z[1/p])$ (and thus the congruence subgroup~$\Phi$) satisfies
the congruence subgroup property~\cite{Mennicke,SerreSL2}. Hence the map~$\Phi \rightarrow \Delta$
is a congruence map, and thus the same is true for the restriction to~$\langle E,\Gamma(N) \rangle \subset \Phi$.
This implies that the kernel~$K \supseteq G$ of the map 
$$\langle E, \Gamma(N)\rangle  \rightarrow  \langle E, \Gamma(N)\rangle/G = S \rightarrow \Delta$$
is a congruence subgroup of~$\SL_2(\Z)$ containing~$E$ and strictly contained in~$\langle E,\Gamma(N)\rangle$.
But this contradicts the assumption that the Wohlfahrt level  of~$G$ is~$N$, because the smallest congruence
 subgroup of Wohlfahrt level~$N$ containing~$E$ is precisely~$\langle E,\Gamma(N)\rangle$
 by~\cite[Theorem~2]{Wohlfahrt}.
 \end{proof}

Let~$B \subset \SL_2(\F_p)$ denote the Borel subgroup of upper triangular matrices.
There is a natural surjection~$\pi: \langle E, \Gamma(N) \rangle \cap \Gamma_0(p) \rightarrow B$
whose kernel is~$\Gamma(Np)$. 
We have the following extension of Lemma~\ref{serreberger}.

\begin{lemma}\label{effort} The map~$(f_1,f_2,\pi): \langle E, \Gamma(N)\rangle \cap \Gamma_0(p) \rightarrow S \times S \times B$
is surjective.
\end{lemma}

\begin{proof} Let~$\gamma = \left( \begin{matrix} 1 & N \\ 0 & 1 \end{matrix}  \right)$ and~$\eta = \left( \begin{matrix} 1 & 0 \\ N& 1 \end{matrix}  \right)$. The assumption that~$L(G) = N$
and~$G$ has finite index in~$\SL_2(\Z)$ impies that~$\gamma, \eta \in G$. 
Since~$A^{-1} \gamma^p A = \gamma \in A^{-1} G A \cap \Gamma_0(p)$, we see that~$\gamma \in \ker(f_1)$ and~$\gamma \in \ker(f_2)$, and yet
$$\pi(\gamma) = \left( \begin{matrix} 1 & N \\ 0 & 1 \end{matrix}  \right) \in B$$
generates the normal unipotent subgroup~$U \subset B$. 
By Goursat's Lemma, we can detect the failure of surjectivity coming from a map of~$S \times S$ and~$B$ to some common quotient.
Because the image contains~$0 \times 0 \times \langle U \rangle$, this common quotient is  a quotient of the abelian group~$B/\langle U \rangle$.
 Thus, by Nakayama's Lemma, the failure of surjectivity can be detected by maps to~$\F_q$ for 
 primes~$q$.
 Maps to~$\F_q$ are determined by cohomology classes with coefficients in~$\F_q$. 
 Let~$\Gbox:=G \cap \Gamma(N)$. Since~$E \in G$, we have
 $$S \simeq \langle E,\Gamma(N) \rangle/G
 \simeq \Gamma(N)/\Gbox,$$
 and so the map~$(f_1,f_2)$ remains surjective after restriction
  to~$\Gamma(N) \cap \Gamma_0(p)$.
 The surjectivity of~$(f_1,f_2)$ implies the injectivity of the map
 \numequation \label{beforeupgrade}
 H^1(\Gamma(N)/\Gbox,\F_q)^{2} = H^1(S,\F_q)^2  \rightarrow 
H^1(\Gamma(N) \cap \Gamma_0(p),\F_q).
\end{equation}
The assumption that~$L(G)=N$ 
implies that
\numequation \label{nogo}
H^1(\Gamma(N)/\Gbox,\F_q) \cap \Hcong(\Gamma(N),\F_q) = 0 \in H^1(\Gamma(N),\F_q),
\end{equation}
where~$\Hcong(\Gamma(N),\F_q) \subset H^1(\Gamma(N),\F_q)$ denotes
the classes which vanish after restriction to a congruence subgroup (Definition~\ref{conghom}).
This is because the kernel of any nontrivial map in~$ \Hcong(\Gamma(N),\F_q)$
has level strictly divisible by~$N$.
The claim~(\ref{beforeupgrade})  follows from~(\ref{nogo}) as a consequence
of Ihara's Lemma, as proved by Ribet~\cite{RibetICM} (see Lemma~\ref{ihararibet}).
The maps
$$ \langle E, \Gamma(N)\rangle \cap \Gamma_0(p)  \rightarrow B/\langle U \rangle \rightarrow \F_q$$
on the other hand come from the classes in~$H^1(\Gamma(N) \cap \Gamma_0(p),\F_q)$
which restricts to zero on~$H^1(\Gamma(N) \cap \Gamma_1(p),\F_q)$,
and thus what is required is to upgrade  the injection of~(\ref{beforeupgrade}) to an injection
 \numequation \label{afterupgrade}
 H^1(\Gamma(N)/\Gbox,\F_q)^{2} = H^1(S,\F_q)^2  \rightarrow 
H^1(\Gamma(N) \cap \Gamma_1(p),\F_q),
\end{equation}
which is dual to the desired claim that the map
$$\Gamma(N) \cap \Gamma_0(p) \rightarrow S^{\ab}/q S^{\ab} \times S^{\ab}/q S^{\ab} \times
B/U$$
is surjective.
But now we may invoke an enhanced version of Ihara's Lemma (Lemma~\ref{enhanced})
which we prove in~\SSSS~\ref{enhancedihara}, and  the injectivity of~(\ref{afterupgrade})
follows
directly from~(\ref{nogo}).
\end{proof}

\begin{remark} 
Because~$\gamma$ and~$E$ map to zero in~$S \times S$ and~$B/\langle E,U \rangle$ has order~$(p-1)/2$,
the proof of Lemma~\ref{effort} is almost immediate if one imposes the additional hypothesis
 that~$\left( \frac{p-1}{2}, |S|\right) = 1$. In particular, one would not have to appeal to the
 results in~\SSSS~\ref{invariant} and~\SSSS~\ref{enhancedihara} (which are not used elsewhere in this paper).
It turns out that proving Lemma~\ref{effort} under this weaker hypothesis would suffice for the proof of the unbounded denominators conjecture.
The key point is that if~$\langle E,\Gamma(N) \rangle/G_N \simeq S$ and~$G_{Np}$ is the group 
given by the intersection of the three groups~$\langle E,\Gamma(Np) \rangle$, $G$, and~$A G A^{-1}$, then~$\langle E,\Gamma(Np) \rangle/G_{Np} \simeq S \times S$. In particular, one can control the primes dividing~$S$ as one varies~$N$. Then, in the argument of Proposition~\ref{selfcontained}, instead of adding all primes~$< X$ prime to~$N$, one
 only includes  primes in some arithmetic progression
satisfying the congruence~$\left( \frac{p-1}{2}, |S|\right) = 1$ for some fixed~$S$.
However, it seems more natural to prove Lemma~\ref{effort} without such an ugly hypothesis.  Additionally, \SSSS~\ref{invariant} and~\SSSS~\ref{enhancedihara} may be of independent interest.
\end{remark}

Returning to the assumptions of Lemma~\ref{effort},
let~$K = \langle \ker((f_1,\pi)), \ker(f_2) \rangle$ be the group generated by~$\ker((f_1,\pi))$ and~$\ker(f_2)$. We deduce from Lemma~\ref{effort} that
the image of~$(f_1,f_2,\pi)$ contains the elements~$(x,0,z)$ and~$(0,y,0)$ for any
triple~$(x,y,z) \in S \times S \times B$. But the pre-images of these elements
clearly lie in~$\ker(f_2)$ and~$\ker((f_1,\pi))$ respectively, and thus lie in~$K$. But then the pre-image of
any element lies in~$K$, and we deduce that~$K = \langle E,\Gamma(N) \rangle \cap \Gamma_0(p)$, or equivalently that
\numequation 
\label{serre}
\langle E,G \cap \Gamma(Np), A^{-1} G A  \cap \Gamma_0(p) \rangle = \langle E,\Gamma(N)\rangle \cap \Gamma_0(p).
\end{equation}
Now specializing to the group~$G = G_{2N}$ of~\S~\ref{leveraging}, we obtain the corresponding identity
\numequation 
\label{serreagain}
\langle E,G \cap \Gamma(2Np), A^{-1} G A  \cap \Gamma_0(p) \rangle = \langle E,\Gamma(2N)\rangle \cap \Gamma_0(p).
\end{equation}

We now complete the proof of Theorem~\ref{fix} and hence the proof of Theorem~\ref{theoremgoal}
(as explained at the beginning of \SSSS~\ref{leveraging}).

\begin{proof}[Proof of Theorem~\ref{fix} ] 
Consider the function~$f(p \tau)$.
Assume that this lies in the algebra generated by~$f(\tau)$ and~$M_{2Np}$.
Then~$f(p \tau)$
is invariant under both~$A^{-1}GA \cap \SL_2(\Z)$  and~$G \cap \Gamma(2Np)$.
But from~(\ref{serreagain}) we see that these groups together generate a congruence subgroup,
and thus~$f(p \tau)$ and~$f(\tau)$ are congruence, a contradiction.
\end{proof}

\subsection{Invariant vectors}
\label{invariant}

The \emph{congruence completion} $\Gammah$ of a congruence subgroup~$\Gamma \subset \SL_2(\Z)$ is the inverse limit of all quotients of~$\Gamma$
by normal congruence subgroups.
We recall the following definition (cf.~\cite[\SSS~3.7]{CalegariVenkatesh}).
\begin{df} \label{conghom} Let~$\Gamma \subset \SL_2(\Z)$ be a congruence subgroup.
A \emph{congruence class}~$\eta \in H^1(\Gamma,\F_{\ell})$ is a class
that restricts to zero on some congruence subgroup~$\Gamma' \subset \Gamma$.
Denote the subgroup of congruence classes by
$$\Hcong(\Gamma,\F_{\ell}) \subset H^1(\Gamma,\F_{\ell}).$$
 \end{df}
If~$\Gammah$ denotes the congruence completion of the group~$\Gamma$,
then~$\Hcong(\Gamma,\F_{\ell}) \simeq H^1(\Gammah,\F_{\ell})$.
In practice, we shall usually talk about~$H^1(\Gammah,\F_{\ell})$ rather than~$\Hcong(\Gamma,\F_{\ell})$
but we have recalled the definition here to allow for an easier comparison with the arguments of~\cite{CalegariVenkatesh}.
For a prime~$\ell$, one may define (\cite[\SSS~2]{CEgrowth}, \cite[\SSS~1]{MR3466858}, see also~\cite{survey}) the groups
$$\Htw^1(\F_{\ell}) := \lim_{N} H^1(\Gamma(N),\F_{\ell}),
\quad \Htw^1(\Q/\Z) := \lim_{N} H^1(\Gamma(N),\Q/\Z)$$
over all levels~$N$.  The limit has an action of the group~$\SL_2(\Zhat) = \prod_{p} \SL_2(\Z_p)$.
The goal of this section is to prove:

\begin{theorem} \label{noinvariants} The~$\SL_2(\Zhat)$-invariant subspace of~$\Htw^1(\F_{\ell})$ is trivial.
\end{theorem}

It follows that the~$\SL_2(\Zhat)$-invariant subspace of~$\Htw^1(\Q/\Z)$ is also trivial.
We shall use Theorem~\ref{noinvariants} in the following equivalent form.

\begin{corr} \label{new}
Let~$N$ be an integer, and~$\eta \in H^1(\Gamma(N),\F_{\ell})$.
If, for all~$g \in \SL_2(\Z/N\Z)$, the class~$g \eta - \eta
\in H^1(\Gamma(N),\F_{\ell})$ is a congruence class, then~$\eta$ is   a congruence class.
\end{corr}

\begin{proof} The assumptions imply that the image of~$\eta$ in~$\Htw^1(\F_{\ell})$
is~$\SL_2(\Zhat)$-invariant, and thus zero. But the kernel of the 
map~$H^1(\Gamma(N),\F_{\ell}) \rightarrow \Htw^1(\F_{\ell})$ consists
precisely of congruence classes.
\end{proof}

Our first goal is to control the
group~$H^2(\Gammah(N),\F_{\ell})$ for various~$N$, in particular for~$N=1$, which we do in a sequence
of steps.

\begin{lemma} \label{lemmaone}
We have~$H_2(\SL_2(\F_p),\Z) = 0$ for all primes~$p$. 
\end{lemma}

\begin{proof}
It suffices to prove the vanishing of~$H_2(\Delta,\Z)$
for any Sylow subgroup~$\Delta$ of~$\SL_2(\F_p)$. For odd primes, the Sylow
subgroup is cyclic and the cohomology of a cyclic group is only non-zero
in even degree. 
For a finite group~$G$, we have~$H^{n+1}(G,\Z) \simeq \Ext^1(H_{n}(G,\Z),\Z)$ for~$n \ge 0$ by
the universal coefficient theorem.
Hence  the homology of a cyclic group is zero in even degree~$n >0$.
The~$2$-Sylow subgroup is a generalized quaternion group,
whose cohomology also vanishes in odd degree
(as follows from~\cite[Satz~25.3(a), p.643]{Huppert}
and~\cite[Theorem~2]{Swan}), and once more we are done by the universal coefficient theorem.
\end{proof}

\begin{lemma}  \label{firstlemma}
For~$n=1$ and~$n=2$, we have:
$$H^n(\SL_2(\F_p),\F_{\ell})^{\vee} \simeq H_n(\SL_2(\F_p),\F_{\ell}) = 
\begin{cases} \F_{\ell}, & p=\ell \in \{2,3\} \\
0, & \text{otherwise}. \end{cases}
$$
\end{lemma}

\begin{proof} 
There is a short exact sequence:
$$0 \rightarrow H_2(G,\Z)/\ell \rightarrow H_2(G,\F_{\ell}) \rightarrow H_1(G,\Z)[\ell] \rightarrow 0,$$
and~$H_1(G,\F_{\ell}) \simeq H_1(G,\Z)/\ell$. Hence the result follows from combining
Lemma~\ref{lemmaone} with the fact that~$\SL_2(\F_p)^{\mathrm{ab}}$ is trivial
for~$p \ge 5$ and~$\Z/p\Z$ for~$p=2$ and~$p=3$.
\end{proof}

\begin{lemma} \label{forkunneth}
 For~$n=1$, we have:
$$H^1(\SL_2(\Z_p),\F_{\ell})  = 
\begin{cases} \F_{\ell}, & p=\ell \in \{2,3\} \\
0, & \text{otherwise}. \end{cases}
$$
For~$n =2$, we have~$H^2(\SL_2(\Z_p),\F_{\ell}) = 0$ unless~$\ell = p$ and~$p \le 5$. 
\end{lemma}

\begin{remark} We shall compute the exceptional cases when~$\ell = p \le 5$ in
Lemma~\ref{postkunneth} below
 as a consequence
of Theorem~\ref{noinvariants}.
\end{remark}

\begin{proof}
Assume that~$\ell \ne p$. By Hochschild--Serre, we have an isomorphism
$$H^*(\SL_2(\Z_p),\F_{\ell})  \simeq H^*(\SL_2(\F_p),\F_{\ell})$$
and thus the result follows from Lemma~\ref{firstlemma}.
Thus we may assume that~$\ell = p$.
Assume that~$p > 2$.
Let~$G(p)$ be the~$p$-congruence subgroup of~$\SL_2(\Z_p)$.
Recall that a group~$G$ is~$p$-powerful  if~$[G,G]$ is contained in the subgroup
generated by~$p$th powers (for~$p$ odd) or~$4$th powers for~$p=2$.
%We may assume that~$p > 2$. 
The group~$G(p)$ is~$p$-torsion free
and~$p$-powerful (for~$p > 2$), so, with~$M = M_0(\F_p)^{\vee}$ where~$G(p)/G(p^2) \simeq M_0(\F_p)$, we deduce by
Lazard's Theorem~(\cite[Chapter~V, 2.2.6.3 and~2.2.7.2, page~551]{Lazard})
 that there are isomorphisms
$$H^1(G(p),\F_p) \simeq M, \quad H^2(G(p),\F_p) \simeq \wedge^2 M,$$
where the cup product map~$\wedge^2 M: H^1 \wedge H^1 \rightarrow H^2$
is an isomorphism.
Assuming~$p \ge 3$, we find that~$M \simeq M^{\vee}$ is self-dual as a~$\SL_2(\F_p)$-module and so~$\wedge^2 M \simeq M$.
Moreover, we have an equality~$M^{\SL_2(\F_p)} = 0$. Consider the Hochschild--Serre spectral
sequence:
$$E^{2}_{i,j} = H^i(\SL_2(\F_p),H^j(G(p),\F_p)) \Rightarrow H^{i+j}(\SL_2(\Z_p),\F_p).$$
Since~$M^{\SL_2(\F_p)} = 0$ and~$M \simeq \wedge^2 M$ we have~$E^{2}_{0,1} = E^{2}_{0,2} = 0$.
It follows that~$E^{\infty}_{0,1} = E^{\infty}_{0,2} = 0$, but also that~$E^{\infty}_{2,0} = E^{2}_{2,0}
= H^2(\SL_2(\F_p),\F_p)$. 
The vanishing of~$E^{\infty}_{0,2}$ implies that~$H^2(\SL_2(\Z_p),\F_p)$
is an extension of~$E^{\infty}_{2,0}$ by~$E^{\infty}_{1,1} \subseteq E^{2}_{1,1} = H^1(\SL_2(\F_p),M)$,
and hence
there is an exact sequence:
 \numequation
 \label{p3}
 0 \rightarrow H^2(\SL_2(\F_p),\F_p)
\rightarrow H^2(\SL_2(\Z_p),\F_p)
\rightarrow H^1(\SL_2(\F_p),M).
\end{equation}
If~$p \ne 5$, then~$H^1(\SL_2(\F_p),M)=0$ (see~\cite[Lemma~2.48]{DDT}) and the result follows from Lemma~\ref{firstlemma}.
\end{proof}

We deduce:

\begin{lemma} \label{mapiszero}  For every prime~$\ell$, there is an isomorphism~$H^2(\SL_2(\Zhat),\F_{\ell})
\simeq H^2(\SL_2(\Z_{\ell}),\F_{\ell})$. If~$N$ is a power of~$\ell$
and~$G(N) \subset \SL_2(\Z_{\ell})$ the corresponding principal congruence subgroup,
then
$$H^2(\Gammah(N),\F_{\ell}) \simeq H^2(G(N),\F_{\ell}).$$
 If~$\ell$ is odd and~$N$ is a  nontrivial power of~$\ell$ or~$N \ge 8$
is a power of~$\ell = 2$, then the map
$$H^2(\SL_2(\Zhat),\F_{\ell}) \rightarrow H^2(\Gammah(N),\F_{\ell})$$
is trivial.
\end{lemma}

\begin{proof} Since~$H^n(\SL_2(\Z_p),\F_{\ell}) = 0$ for~$n=1$ and~$n=2$
unless~$\ell = p$, the first two claims follow from the K\"{u}nneth formula
and Lemma~\ref{forkunneth}.
It remains to show that the map
$$H^2(\SL_2(\Z_p),\F_p) \rightarrow H^2(G(N),\F_p)$$
is the zero map for~$N = p$ if~$p$ is odd and~$N = 8$ if~$p = 2$.
For~$p > 2$,
 we have~$H^2(G(N),\F_p) \simeq M$  and~$M^{\SL_2(\F_p)} = 0$.
 Since  the source is~$\SL_2(\F_p)$-invariant, the image must be trivial. For~$p = 2$, we have~$H^2(G(N),\F_2) \simeq \wedge^2(M) \simeq M$
whenever~$N \ge 4$ (to ensure that~$G(N)$ is~$2$-powerful). 
Unlike what happens for~$p$ odd, we have~$M^{\SL_2(\F_2)} = \F_2$.
However, the map
$$M = H^1(G(4),\F_2) \rightarrow H^1(G(8),\F_2) = M$$ 
is zero, and thus the induced map
$$\wedge^2 M =  H^2(G(4),\F_2) \rightarrow H^2(G(8),\F_2) = \wedge^2 M$$ 
is also zero.
\end{proof}

Now let us consider the following commutative diagram for~$N \in \{3,4,5,8,16\}$
and~$\ell$ dividing~$N$ coming from compatible
Hochschild--Serre spectral sequences:

{\small
\numequation \label{HScommutative}
\begin{tikzcd}[column sep=small, row sep=large]
0 \arrow[r] & H^1(\SL_2(\Zhat),\F_{\ell}) \arrow[r] \arrow[d] & H^1(\SL_2(\Z),\F_{\ell}) \arrow[r] \arrow[d] & (\Htw^1(\F_{\ell}))^{\SL_2(\Zhat)} \arrow[r] \arrow[d, hook] & H^2(\SL_2(\Zhat),\F_{\ell}) \arrow[r] \arrow[d] & H^2(\SL_2(\Z),\F_{\ell})  \\
0 \arrow[r] & H^1(\Gammah(N),\F_{\ell}) \arrow[r] & H^1(\Gamma(N),\F_{\ell}) \arrow[r] & (\Htw^1(\F_{\ell}))^{\Gammah(N)} \arrow[r] & H^2(\Gammah(N),\F_{\ell}) \arrow[r] & 0
\end{tikzcd}
\end{equation}
\stepcounter{equation}
}
Here~$H^2(\Gamma(N),\F_{\ell}) = 0$ because~$\Gamma(N)$ is a free group.
(For~$N \ge 3$, $\Gamma(N)$ is torsion free, and so~$\Gamma(N)$ may
be identified with the fundamental group of a surface~$\H/\Gamma(N)$ with cusps.) 
The last vertical map is zero by the previous lemma if~$N \in \{3,5,8,16\}$, and thus the image
of~$(\Htw^1)^{\SL_2(\Zhat)}$ in~$(\Htw^1)^{\Gammah(N)}$ lands in the image
of~$H^1(\Gammah(N),\F_{\ell})$ in these cases. But these are finite groups we can compute explicitly.

\begin{lemma} \label{computeinvariants}
We have
$$
\begin{aligned}
 & \dim H^1(\Gamma(3),\F_3)^{\SL_2(\F_3)} =  0, \\
&  \dim H^1(\Gamma(5),\F_3)^{\SL_2(\F_5)} =  0, \\
\dim H^1(\Gamma(4),\F_2)^{\SL_2(\Z/4\Z)} = \
&  \dim H^1(\Gamma(8),\F_2)^{\SL_2(\Z/8\Z)}  = 
\dim H^1(\Gamma(16),\F_2)^{\SL_2(\Z/16\Z)}  = 1.
\end{aligned}
$$
For~$N = 3,4,5$, the same result holds even after considering
the semi-simplifications 
of these modules.
\end{lemma}

\begin{proof}
Recall that for~$N=3$, $4$, and~$5$ that~$X(N)$ has genus zero. Hence the cohomology of the module
$V =H^1(\Gamma(N),\Z)$
 is coming entirely from the the cusps, which correspond
to the cosets of~$\langle \left( \begin{matrix} 1 & 1 \\ 0 & 1 \end{matrix} \right) \rangle$
in~$\PSL_2(\Z/N \Z)$. In particular,  in the Grothendieck 
group~$K_0(\Q[\SL_2(\Z/N \Z)])$ of~$\SL_2(\Z/N\Z)$-representations over~$\Q$,
$$V_{\Q} :=  \left[ H^1(\Gamma(N),\Q)\right] \simeq \Q \left[\PSL_2(\Z/N\Z)/ \left( \begin{matrix} 1 & 1 \\ 0 & 1 \end{matrix} \right) \right] - [\Q].$$
Since~$\Gamma(N)$ is free, this is enough to determine the semi-simplification 
of~$V_{\F_{\ell}}:= H^1(\Gamma(N),\F_{\ell})$.
We consider each case in turn.
\begin{enumerate}
\item For~$N=3$, we have~$\PSL_2(\F_3) = A_4$, and~$V_{\Q}$
is absolutely irreducible of dimension~$3$.
The associated Brauer character is also irreducible
and so~$[V_{\F_{\ell}}]$ is also
 irreducible and has no invariants.
\item For~$N=5$, we have~$\PSL_2(\Z/5\Z) \simeq A_5$, and~$V_{\Q}$ decomposes
as a sum of irreducibles of dimensions~$3$, $3$, and~$5$.
The corresponding Brauer characters are all still irreducible, so~$[V_{\F_{\ell}}]$ 
does not contain the trivial representation.
 \item
 For~$N =4$, we have~$\PSL_2(\Z/4\Z) = S_4$,
 and~$V_{\Q}$ is a sum of absolutely irreducible representations of dimensions~$2$
 and~$3$. The group~$S_4$ has two Brauer characters of dimension~$1$ and~$2$ respectively.
 The~$2$-dimensional representation remains irreducible and the semi-simplification
 of both the~$3$-dimensional representations has constituents of dimensions~$1$ and~$2$.
 Hence the invariant space  of~$V^{\sss}_{\F_{\ell}}$ is~$1$-dimensional.
 But~$H^1(\Gamma(4)/\Gamma(8),\F_2)$ is a direct sum of the~$1$ and~$2$-dimensional
 representations, so this $1$-dimensional constituent occurs as a sub-representation.
\item 
For~$N=8$ and~$N=16$, we resort to a less elegant calculation; the groups~$\Gamma(N)$
are free (of ranks~$33$ and~$257$ respectively). The~$\SL_2(\Z/N\Z)$-invariant part of cohomology over~$\F_2$ can be determined as (the dual of) the quotient of this group by the 
relations~$x^2 = (x y)^2 = e$ for each generator~$x \in \Gamma(N)$ and the relations~$g x g^{-1} = x$ for the generators~$g$ of~$\SL_2(\Z)$. In both cases, \texttt{magma} determines that
the corresponding quotients have order~$2$.
\end{enumerate}
This completes the proof of the Lemma.
\end{proof}

\begin{proof}[Proof of Theorem~\ref{noinvariants}]
We now complete the proof of Theorem~\ref{noinvariants}.

We need to show that any~$v \in \Htw^1(\F_{\ell})^{\SL_2(\Zhat)}$ is zero.
Let us consider the images of~$v$ under various 
maps in equation~(\ref{HScommutative}).
We first note that~$H^1(\SL_2(\Z),\F_{\ell}) \simeq \Z/12\Z \otimes \F_{\ell}$ and that the map
 $$H^1(\SL_2(\Zhat),\F_{\ell}) \rightarrow H^1(\SL_2(\Z),\F_{\ell})$$
 is an 
isomorphism for any~$\ell$. Hence we may assume the image of~$v$ in~$H^2(\SL_2(\Zhat),\F_{\ell})$
is non-zero.
  From the K\"{u}nneth formula
and Lemma~\ref{forkunneth},  there is an isomorphism~$H^2(\SL_2(\Zhat),\F_{\ell}) = 0$
for any prime~$\ell > 5$, and hence we may assume that~$\ell \le 5$.

Suppose that~$\ell = 3$ or~$\ell = 5$, and take~$N = \ell$ in equation~(\ref{HScommutative}). We
proved that the map
$H^2(\SL_2(\Zhat),\F_{\ell}) \rightarrow H^2(\Gammah(\ell),\F_{\ell})$
is zero by Lemma~\ref{mapiszero}. It follows that the image of~$v$
in~$(\Htw^1(\F_{\ell}))^{\Gammah(\ell)}$  is~$\SL_2(\Z/\ell \Z)$-invariant
and lands in the image of~$H^1(\Gamma(\ell),\F_{\ell})$.
Thus the~$\SL_2(\Z/\ell\Z)$-invariants of the semi-simplification
of~$H^1(\Gamma(\ell),\F_{\ell})$ as a~$\SL_2(\Z/\ell\Z)$-module is nontrivial.
But this space has dimension~$0$ by Lemma~\ref{computeinvariants}.

Finally, let~$\ell = 2$. By Lemma~\ref{mapiszero}, the map
$$H^2(\SL_2(\Zhat),\F_{2}) \rightarrow H^2(\Gammah(8),\F_{2})$$
is zero, and thus, arguing as in the case~$\ell =3$ or~$5$ above,
  the image of~$v$ in~$(\Htw^1(\F_{2}))^{\Gammah(8)}$ 
coincides with the image of an element~$w \in H^1(\Gamma(8),\F_2)$.
Furthermore, the~$\SL_2(\Z/8\Z)$-module generated by~$w$
is~$\SL_2(\Z/8\Z)$-invariant  after passing to the quotient by the congruence homology
$$H^1(\Gammah(8),\F_2) \simeq H^1(\Gammah(8)/\Gammah(16),\F_2) \simeq H^1(\Gamma(8)/\Gamma(16),\F_2).$$
But that means that the image of~$w$ in~$H^1(\Gamma(16),\F_2)$
 is invariant under~$\SL_2(\Z/16 \Z)$. By
Lemma~\ref{computeinvariants}, the space of such invariants is 
$1$-dimensional. But this $1$-dimensional space lands in the image
of~$H^1(\Gammah(16),\F_2)$,  and thus the image of~$w$ and hence also of~$v$ must be trivial 
in~$\Htw^1(\F_2)^{\Gammah(16)} \subset \Htw^1(\F_2)$, and in particular~$v = 0$.
\end{proof}

We note in passing that this result implies 
the following strengthening of
Lemma~\ref{forkunneth}:

\begin{lemma} \label{postkunneth}
 For~$n=1$ and~$n=2$ we have:
$$H^n(\SL_2(\Z_p),\F_{\ell})  = 
\begin{cases} \F_{\ell}, & p=\ell \in \{2,3\} \\
0, & \text{otherwise}. \end{cases}
$$
We also have~$H_2(\SL_2(\Z_p),\Z) = 0$ for all~$p$.
\end{lemma}

\begin{proof} 
From  Lemma~\ref{forkunneth},
it suffices to consider the case of~$n=2$ and~$\ell = p$.
% \in \{2,3,5\}$.
For any~$\ell$ and~$p$,
there is an exact sequence:
\numequation
\label{forkunnethsequence}
0 \rightarrow H_2(\SL_2(\Z_p),\Z)/\ell
\rightarrow H_2(\SL_2(\Z_p),\F_{\ell}) 
\rightarrow H_1(\SL_2(\Z_p),\Z)[\ell] \rightarrow 0.
\end{equation}
Since~$H_1(\SL_2(\Z_p),\Z) \simeq \Z/12\Z \otimes \Z_p$,
this proves that~$H^2(\SL_2(\Z_p),\F_p)$ has dimension
at least one when~$p \in \{2,3\}$.
Suppose we prove that~$H^2(\SL_2(\Z_p),\F_{\ell})$
has dimension at most one when~$\ell = p \in \{2,3\}$ and dimension zero otherwise.
First, this would complete the computation of~$H^2(\SL_2(\Z_p),\F_{\ell})$.
Second, it would follow that the second map in equation~(\ref{forkunnethsequence})
is always an isomorphism, and so~$H_2(\SL_2(\Z_p),\Z)/\ell = 0$
for all primes~$\ell$.
Since the group~$H_2(\SL_2(\Z_p),\Z)$
is  a finitely generated abelian group, this will also show that
 that~$H_2(\SL_2(\Z_p),\Z) = 0$, completing the proof of the lemma.
 
 Let us now bound from above the dimension of~$H^2(\SL_2(\Z_p),\F_{\ell})$.
 We may assume that~$\ell = p$ by Lemma~\ref{forkunneth}.
 There are maps:
 \numequation
 \label{usedforbound}
 H^2(\SL_2(\Z_{\ell}),\F_{\ell}) \simeq
H^2(\SL_2(\Zhat),\F_{{\ell}}) \hookrightarrow H^2(\SL_2(\Z),\F_{{\ell}}),
\end{equation}
where the first map is an isomorphism
by Lemma~(\ref{mapiszero}), and the second map is an inclusion
from the exact sequence~(\ref{HScommutative})
and the vanishing of~$(\Htw^1(\F_{\ell}))^{\SL_2(\Zhat)}$
by Theorem~\ref{noinvariants}.
But for~$n > 0$ we have
$$H^n(\SL_2(\Z),\Z) = \begin{cases} \Z/12\Z, & n \equiv 0 \bmod 2, \\
0, & n \equiv 1 \bmod 2 \end{cases},$$
from which it follows that~$H^2(\SL_2(\Z),\F_{\ell}) = \F_{\ell}$
if~$\ell \in \{2,3\}$ and is zero otherwise.
This gives the desired upper bound on the dimension
of~$H^2(\SL_2(\Z_{\ell}),\F_{\ell})$ via 
the inclusion~(\ref{usedforbound}), completing the proof.
\end{proof}

\subsection{An enhancement of Ihara's Lemma}
\label{enhancedihara} 
We shall prove an enhanced version of Ihara's Lemma. We begin
by recalling Ihara's Lemma.
 Let~$q$ be prime, let~$N \ge 3$,  and let~$(N,p) = 1$. There is a homomorphism
\numequation \label{ihara}
H^1(\Gamma(N),\F_q)^{2} \rightarrow H^1(\Gamma(N) \cap \Gamma_0(p),\F_q),
\end{equation}
given by the difference of the following two maps:
\begin{enumerate}
\item The map sending~$\psi: \Gamma(N) \rightarrow \F_q$ to its restriction to~$\Gamma(N) \cap \Gamma_0(p)$.
\item The twisted restriction map coming from viewing~$\psi \in H^1(\Gamma(N),\F_q)$ as a map~$\Gamma(N) \rightarrow \F_q$
and then considering the map
$$A\psi: \Gamma(N) \cap \Gamma_0(p) \rightarrow \F_q, \quad g \mapsto \psi(A g A^{-1}).$$
\end{enumerate}
By abuse of notation we denote the restriction of~$\psi$ by~$\psi$, so the map sends~$(\psi,\phi)$ to~$\psi - A \phi$.

\begin{lemma}[Ihara's Lemma]  \label{ihararibet} The kernel of the map~(\ref{ihara}) lies 
inside~$\Hcong(\Gamma(N),\F_q)^2$. 
\end{lemma}

\begin{proof} This version of Ihara's Lemma was essentially proved
by Ribet in~\cite{RibetICM}. The proof is just an abelian version
of Lemma~\ref{serreberger}. 
We recall some of the details.
Let~$\Phi \subset \Gamma(N)$ be the maximal normal subgroup whose quotient is an
elementary~$q$-abelian group~$T$. Canonically, we have~$\Gamma(N)/\Phi \simeq T$
and~$H^1(\Gamma(N),\F_q) \simeq \Hom(T,\F_q)$. The kernel of the map
$$\Hom(T,\F_q) \times \Hom(T,\F_q) \rightarrow H^1(\Gamma(N) \cap \Gamma_0(p),\F_q)$$
is governed by the cokernel of the dual map
$$\Gamma(N) \cap \Gamma_0(p) \rightarrow T \times T.$$
Exactly as in the proof of Lemma~\ref{serreberger}, we deduce from
Goursat's Lemma that the cokernel~$\Delta$ arises
from two maps from~$\Gamma(N)$ to~$\Delta$ which agree along~$\Gamma(N)
\cap \Gamma_0(p)$, and thus on their amalgam~$\SL_2(\Z[1/p])(N)$.
Since~$\SL_2(\Z[1/p])(N)$ has the congruence subgroup property, it thus
arises from a congruence quotient of this group at primes away from~$p$.
 But that precisely means that the
classes in~$\Hom(T,\F_q) = H^1(\Gamma(N),\F_q)$ become trivial after passing to
a congruence subgroup~$\Gamma' \subset \Gamma(N)$, hence the claim.
\end{proof}

Using  Corollary~\ref{new}, we prove a slight enhancement of this
claim.

\begin{lemma}[Ihara's Lemma, enhanced] \label{enhanced}
The kernel of the composite of the Ihara map~(\ref{ihara})
with the map
\numequation
\label{newmap}
H^1(\Gamma(N) \cap \Gamma_0(p),\F_q) \rightarrow H^1(\Gamma(N) \cap \Gamma_1(p),\F_q)
\end{equation}
also lies inside~$\Hcong(\Gamma(N),\F_q)^2$. 
\end{lemma}

\begin{proof}
The map~(\ref{ihara}) is~$\SL_2(\Z/N\Z)$-equivariant. But the kernel
of the map~(\ref{newmap}) is also easily seen to be~$\SL_2(\Z/N\Z)$-invariant.
Hence, if~$(\psi,\phi)$ lies in the kernel of the composite of~(\ref{ihara})
and~(\ref{newmap}), then~$(\psi^g - \psi,\phi^g - \phi)$ lies in the kernel of~(\ref{ihara}),
and thus lies in~$\Hcong(\Gamma(N),\F_q)^2$ by Lemma~\ref{ihararibet}. But then~$\psi$ and~$\phi$ are themselves
congruence classes by Corollary~\ref{new}.
\end{proof}

\section{The uniformization of~\texorpdfstring{$\C \setminus \mu_N$}{CminusMu}}
\label{uniformizations}

In this section we develop all the particular analytic properties that  we need of the universal 
covering map $F_N : D(0,1) \to \C \setminus \mu_N$ for $N\geq 2$.  Andr\'e has pointed out to us that
our two main results here, namely Theorem~\ref{exactconformal} and Lemma~\ref{trivialsupFN}, 
appear in work of Kraus and Roth~\cite[Remark 5.1 and Theorems 1.2 and 1.10]{KrausRoth}. Nevertheless, as our proofs are simplified to cover our current needs, and since the results of Kraus and Roth rely on some previous work of themselves and others, 
we keep our self-contained exposition as a convenience to the reader, and  refer to~\cite{ASVV,KrausRothSugawa,KrausRoth} and the references there for various further results and a more thorough study of the uniformization of 
$\C \setminus \mu_N$. The reader will also benefit from the material in \SSSS~III.1 in Goluzin's book~\cite{Goluzin}, which recovers~$F_N$ via  an explicit computation of the Riemann map of a $\Z / N\Z$-rotationally symmetric circular $N$-gon, taking the case of zero angles and doing Schwarz reflections in the sides of the circular polygon. 

\begin{remark}[A word on notation]
We denote by~$\H$  the upper half plane and by $\mathbb{P}^1 = \C \cup \{\infty\}$ the complex projective line or Riemann sphere. There is a conformal isomorphism from the disc~$D(0,1)$ to~$\H$
by the Cayley transform
$$z \mapsto 
 i \cdot \frac{1+z}{1-z}.$$
 This allows one to pass freely between uniformizations by~$D(0,1)$ and~$\H$.
 In this section, we choose notation so that the corresponding passage from~$D(0,1)$
 to~$\H$ is marked by the addition of a tilde. Thus, for example, $\wF_N$ constructed
 below denotes a map on~$\H$ and~$F_N$ (Definition~\ref{tildeexample}) is simply
 the pull-back of~$\wF_N$ to~$D(0,1)$ via the map above. Similarly, $\Gamma_N$
 will denote a lattice in~$\PSU(1,1)$ whereas~$\wGamma_N$ denotes the corresponding lattice in~$\PSL_2(\R)$.
 
Unless  we expressly state otherwise, we reserve  $z, \tau, x$ to denote respectively the coordinates on $D(0,1)$, $\H$, and $\C\setminus \mu_N$.
 \end{remark}
 
  \subsection{Schwarzians and the conformal radius}
\label{section:normalizations}
Let~$N \ge 2$ be an integer. Then~$ \mathbf{C} \setminus \mu_N = \mathbf{P}^1 \setminus \{\infty,\mu_N\}$
is the complement of at least~$3$ points, and thus admits a complex uniformization map:
$$\wF_N: \H \rightarrow \H/\wGamma_N = \mathbf{C} \setminus \mu_N,$$
where $\wGamma_N\subset \PSL_2(\R)$ denotes the Fuchsian group of $\mathbf{C} \setminus \mu_N$.
The map~$\wF_N$ is unique up to the action of~$\PSL_2(\R)$ by M\"{o}bius transformations
on the source,
which also changes~$\wGamma_N$ by conjugation.
% --- we pin down some precise choices below, but the analysis of this section will not depend on any such choices.
The cusps of~$\wGamma_N$ are the elements~$x \in  \partial \H = \mathbf{P}^1(\R) = \R \cup \{i \infty\}$
such that the stabilizer of~$x$
under~$\wGamma_N$ contains a parabolic element. If~$\H^*$ denotes the union of~$\H$ with the cusps,
then~$\H^*/\wGamma_N$ may be identified with the compactification~$\mathbf{P}^1$
of~$\C \setminus \mu_N$. Since~$\PSL_2(\R)$ acts transitively on~$\partial \H$, we may
assume, after
translation by an element of~$\PSL_2(\R)$,  that~$i \infty$ is a cusp of~$\wGamma_N$
and  that~$\wF_N(i \infty) = 1$.
The stabilizer of~$i \infty$ in~$\PSL_2(\R)$ consists of M\"{o}bius transformations of the 
form~$\tau \rightarrow a \tau + b$ for some~$a,b \in \R$. Thus we may pin down~$\wF_N$ and~$\wGamma_N$
exactly by further specifying that~$\wF_N(i) = 0$.
%The stabilizer of~$i\in \H$ consists of M\"{o}bius transformations given by elements in $\SO_2(\R)$, so by specifying~$\wF_N(i \infty) = 1$ (here by abuse of notation, we also use $\wF_N$ to denote the natural extension of $\wF_N$ on $\H$ to cusps), we determine~$\wF_N$ uniquely.

%Using the action of~$\PSL_2(\R)$ we may assume that~$\wF_N(i \infty) = 1$.
%The stabilizer of~$\infty$ consists of M\"{o}bius transformations~$z \mapsto a z + b$,
%so by specifying~$\wF_N(0) = 0$ we determine~$\wF_N$ uniquely.

\begin{df} \label{tildeexample}
Define~$F_N: D(0,1) \rightarrow \C \setminus \mu_N$
by the formula
$$
F_N(z) = \wF_N \left( i \cdot \frac{1+z}{1-z} \right).
$$
\end{df}

Note that $F_N$ is just the map~$\wF_N$ composed with a conformal isomorphism~$D(0,1) \rightarrow \H$
 sending~$0$ to~$i$,
and hence
$$
F_N:D(0,1) \rightarrow \C \setminus \mu_N
$$
is the universal covering map with $F_N(0)=0$ and $F_N(1)=1$.

Note that the statements of the main results of this section, Theorem~\ref{exactconformal} and Lemma~\ref{trivialsupFN}, only depend on the normalization $F_N(0)=0$ and do not depend on the choice $F_N(1)=1$.

The following lemma gives the basic symmetric property of $\wF_N$ and $F_N$.
\begin{lemma}\label{FNbasic}
Let~$\zeta_N = \exp(2 \pi i/N)$ and $\zeta$ be any $N$th root of unity. Then $\zeta_N \wF_N(\tau) = \wF_N(\wr_N \cdot \tau)$ and $F_N(\zeta x) = \zeta F_N(x)$, where \numequation
\label{elementr}
\wr_N = \left( \begin{matrix}  \cos(\pi/N) & -\sin(\pi/N) \\ \sin(\pi/N) & \cos(\pi/N) \end{matrix} \right) \in \PSO_2(\R).
\end{equation}
\end{lemma}
\begin{proof}
Note that~$\zeta_N \wF_N$ is another covering map such that $\zeta_N \wF_N(i)=0$. Therefore $\zeta_N \wF_N$ must
differ from~$\wF_N$ by a M\"{o}bius transformation in the stabilizer of $i$; that is~$\zeta_N \wF_N(\tau) = \wF_N(\wr_N \cdot \tau)$ for some $\wr_N\in \PSO_2(\R)$. 
We deduce that~$\wF_N(\wr^N_N \cdot \tau) = \zeta_N^N \wF_N(\tau) = \wF_N(\tau)$,
 and thus~$\wr^N_N \in \PSO_2(\R)$ must also lie in~$\wGamma_N$. But~$\wGamma_N$ is a free group (due to the fact that $\wF_N$ is a covering map with no ramification points),
 and hence~$\wr^N_N$ is trivial in $\PSO_2(\R)$, and~$\wr_N$ is a hyperbolic rotation around~$i$ of order~$N$.
 
The action of~$\PSO_2(\R)$ on~$D(0,1)$ under the pullback map is just given by rotation, and hence~$\wr_N$
acts on~$D(0,1)$ by a rotation of order~$N$. We deduce that~$F_N(\zeta^m z) = \zeta F_N(z)$ for some~$(m,N) = 1$. By taking
the derivatives with respect to $q$ of both sides at $q=0$, we have $\zeta^m F_N'(0)=\zeta F_N'(0)$.
Since~$F_N$
is a covering  map, we must also have~$F'_N(0) \ne 0$. We deduce that $\zeta^m = \zeta$ and hence also that $F_N(\zeta z) = \zeta F_N(z)$. We thus also deduce~\eqref{elementr} since it follows
that~$\wr_N$ is a (hyperbolic) rotation  by~$2\pi/N$ degrees around~$\tau = i$ in~$\H$.
\end{proof}

Our first main goal of this section is an explicit computation of the uniformization radius of $\C \setminus \mu_N$. This formula has been previously proved by Kraus and Roth in~\cite[Remark~5.1]{KrausRoth}. 

\begin{theorem} \label{exactconformal} The conformal size $|F_N'(0)|$ (Riemann uniformization radius of $\C \setminus \mu_N$) %of~$F_N: D(0,1) \rightarrow \C \setminus \mu_N$
is equal to
\numequation   \label{radius explicit}
|F_N'(0)| =  \gamma_N := 16^{1/N}\frac{
\displaystyle{
 \Gamma \left( 1+\frac{1}{2N} \right)^2  \Gamma \left(1 - \frac{1}{N} \right)
}
}{
\displaystyle{
 \Gamma \left( 1-\frac{1}{2N} \right)^2  \Gamma \left(1 + \frac{1}{N} \right)
}}.
\end{equation}
We have an expansion for~$\gamma_N$ as follows:
\numequation\label{explicit-gammaN}
\gamma_N = 16^{1/N} \left(1 + \frac{\zeta(3)}{2 N^3} + \frac{3 \zeta(5)}{8 N^5} + O(N^{-6}) \right),
\end{equation}
where the remaining term $O(N^{-6})$ is a positive real number.
\end{theorem}

To prove this formula, we follow Hempel~\cite{Hempel} to get a second order linear ODE whose ratio of two linearly independent solutions gives the (local analytic) inverse of~$\wF_N$  (Lemma~\ref{diff-eq-FN}). 
The uniformization maps of Riemann surfaces---and their inverses---do not typically admit explicit solutions in terms of standard functions, but
our particular case of interest turns out to be an exception due to the extra symmetries of~$\C \setminus \mu_N$. 
We use Lemma~\ref{FNbasic} to define a function $G_N$ closely related to $F_N$ (see Definition~\ref{GN}) and explicitly find two solutions of the associated linear ODE in terms of hypergeometric functions (Lemma~\ref{sol-diffG}). These solutions allow us to compute the explicit conformal radius for $G_N$ and then derive
the corresponding conformal radius for~$F_N$ given in equation~\eqref{radius explicit}.

Our computation here is very similar to the treatment by Goluzin in \cite[\SSS~III.1]{Goluzin}, who also gives the explicit formula for the inverse of $G_N$ in terms of hypergeometric functions. See the $q = 0$ case of equation~(17) and the last paragraph on page 86 of \emph{loc.~cit.}  Goluzin more generally computes the Riemann map for the $\Z/N\Z$-rotationally symmetric circular $N$-gon with angles $\pi q$, and explains~\cite[\SSS~II.6]{Goluzin} how the $q = 0$ case (formula~(21) on page~86 of \emph{loc.~cit.}) by Schwarz reflections entails a description of $G_N$ and $F_N$.   %the formulae are the same by Pfaff transformation.

\begin{df} 
Let~$\psi_N$ be the local analytic inverse of~$F_N$ such that $\psi_N(0) = 0$. 
\end{df}

This inverse exists and is unique in a small neighborhood of $z=0$. As all we need is to compute $F'_N(0)=\psi_N'(0)^{-1}$, having $\psi_N$ well-defined in a small neighborhood of $z=0$ is enough for our purpose.

%The map~$\psi_N$ is well-defined up to the action of~$\Gamma_N \subset \PSU(1,1)$ acting on~$D(0,1)$.

\begin{lemma}\label{diff-eq-FN}
The local analytic inverse map~$\psi_N$ of~$F_N$ has the form $\psi_N = \eta_1 / \eta_2$, where $\eta_1$ and $\eta_2$ satisfy
the second order linear differential equation 
\numequation
\label{3.3}
4  (x^N- 1)^2 y'' + ((N^2-1) x^{N-2} + x^{2N-2}) y = 0.
\end{equation}
\end{lemma}

\begin{remark} \label{remark Hermann Schwarz}
The equation~\eqref{3.3} is more transparent in terms of the \emph{Schwarzian derivative}: 
$$
y'' + \frac{1}{2} \{\tau,\wF_N\} y = 0.
$$ 
We recall here the role~\cite[\S~IV.1.2]{SaintGervais} of Schwarz's \emph{departure from infinitesimal projectivity}: 
$$
\{w, x\} := \left( \frac{w''}{w'} \right)' - \frac{1}{2} \left( \frac{w''}{w'} \right)^2 = \frac{d^2}{dx^2} \log{\frac{dw}{dx}} - \frac{1}{2} \left(
\frac{d}{dx} \log{\frac{dw}{dx}} \right)^2,
$$
 the simplest differential operator invariant under all M\"obius transformations. It is featured in the ODE~\cite[Proposition~VIII.3.5]{SaintGervais}
\begin{equation} \label{Schwarzian ODE}
\frac{d^2y}{dx^2} + \frac{1}{2}\{ w, x \} y = 0
\end{equation}
that can be used to formally represent an unknown function~$w$ as the quotient~$w = v_1/v_2$ of the two linearly independent solutions~$y = v_1 := w/\sqrt{w'}$ and~$y = v_2 := 1/\sqrt{w'}$ of the second-order linear ODE~\eqref{Schwarzian ODE}. Following Poincar\'e in his ODE approach to the uniformization of Riemann surfaces, we are interested to describe in this way the multivalued holomorphic inverse~$w : U \to \H$ to an analytic universal covering map~$\H \to U$, in the case that the Riemann surface~$U = \C \setminus \{a_1, \ldots, a_N\}$ is the complement of finitely many punctures in the complex plane, and with~$x$ taken as some local coordinate of  the complex projective line. In this case, by a local analysis near the punctures~$\{a_i\} \cup \{\infty\} \subset \mathbb{P}^1$, the Schwarzian~$\{w,x\} \in \C(x)$ is simply a rational function, which is far easier to compute in practice than the map~$w$ {\it a priori}. The ODE~\eqref{Schwarzian ODE} then furnishes a local analytic description of the requisite inverse map~$w = v_1/v_2$ which can then be analyzed both locally (as in the rest of the current~\S~\ref{section:normalizations}) and globally (as in the next~\S~\ref{GeoFNGN}).  
\end{remark}

\begin{example}[See also Example~\ref{ex23}] \label{example H}
\emph{For~$U = \C \setminus \{0, 1\}$, a universal covering map is~$\lambda : \H \to U$, but the more basic element is the multivalued inverse (``upper half plane'')
$$
\tau : \C \setminus \{0,1\} = U \to \H, \qquad
\tau(\lambda) := \omega_2 / \omega_1 = i K'(\lambda) / K(\lambda) = i \frac{ \pFq{2}{1}{1/2,1/2}{1}{1-\lambda}}{ \pFq{2}{1}{1/2,1/2}{1}{\lambda}} 
$$
which is locally the quotient of two periods of the Legendre elliptic curve~$y^2 = x(x-1)(x-\lambda)$, {\it alias} two linearly independent
  solutions~$K(\lambda) =  \pFq{2}{1}{1/2,1/2}{1}{\lambda}$ and~$i K'(\lambda) = i \cdot  \pFq{2}{1}{1/2,1/2}{1}{1-\lambda}$ of Gauss's hypergeometric ODE
$(\lambda^2-\lambda) y'' + (2\lambda-1)y' + y/4 = 0$. The latter is tantamount to~\eqref{Schwarzian ODE} for the case~$x = \lambda$ and~$w = \tau$, and this is the 
picture that we want to generalize.}
\end{example}

\begin{proof}[Proof of Lemma~\ref{diff-eq-FN}]
First, since $F_N$ is the composition of $\wF_N$ and a M\"{o}bius transformation, we only need to prove the similar assertion for $\wF_N$.  By taking reciprocals,
we have a companion uniformization map~$1/\wF_N: \H \rightarrow \mathbf{P}^1 \setminus \{0,\mu_N\}$. (The reason for first considering the reciprocal of~$\wF_N$ is that the standard form considered in~\cite{Hempel}
is for maps from $\H$ to~$\mathbf{P}^1 \setminus S$ where~$S$ is a finite set of points which does not contain~$\infty$.) This is similar to Example~\ref{example H}, except now taking~$x = \wF_N$, rather than~$x=\lambda$, as the coordinate of the projective line, once again parametrized by~$\H$ via our universal covering map of~$\C \setminus \mu_N$. 

By~\cite[Lemma~3.3]{Hempel}, the analytic local
inverse map of~$1/\wF_N$ (resp.~$\wF_N$) is, up to a M\"{o}bius transformation, the ratio of two linearly independent solutions of the differential equation equation~$y'' + \frac{1}{2} \{\tau,1/\wF_N\} y=  0$ (resp.~$y'' + \frac{1}{2} \{\tau,\wF_N\} y=  0$), where $\{ \tau, 1/\wF_N\}$ and $\{ \tau, \wF_N\}$ denote the Schwarzian derivatives.
 We now compute $\{ \tau, 1/\wF_N\}$ and then $\{ \tau, \wF_N\}$ following~\cite[\SSS~3, \SSS~6]{Hempel}.
Let~$p_k =  \zeta_N^k = e^{2 \pi i k/N}$ for~$k = 1, \ldots, N$ and let~$p_0 = 0$.
We deduce from~\cite[Theorem~3.1]{Hempel} that
the Schwarzian~$\{\tau,1/\wF_N\}$ is given by
\numequation 
\label{firstS}
\{\tau,1/\wF_N\} = \frac{1}{2} \sum_{k=0}^{N} \frac{1}{(X - p_k)^2}
 + \sum_{k=0}^{N} \frac{m_k}{X - p_k},
 \end{equation}
 where the~$m_k$ for~$k = 0 ,\ldots, N$ denote the so-called \emph{accessory parameters} at~$z = p_k$.
The accessory parameters are notoriously hard to compute in general, but  in our particular example we may find them using the~$\Z/N \Z$ symmetry.
Expressing the fact that~\eqref{firstS} vanishes to order four at $X = \infty$, the accessory parameters are subject to the following three constraints~\cite[Theorem~3.1]{Hempel} obtained by equating the $1/X, 1/X^2$ and $1/X^3$ coefficients to zero:
\numequation
\label{constraints}
\sum_{k=0}^{N} m_k = 0, \quad \sum_{k=0}^{N} 2 m_k p_k + 1 = 0, \quad \sum_{k=0}^{N} m_k p^2_k + p_k = 0.
\end{equation}
Since~$\mathbf{C} \setminus \mu_N$ is invariant under the action of~$\mu_N$, we deduce
exactly as
 in~\cite[\SSS~6, Example~1]{Hempel} that
 the accessory parameters~$m_k$ for~$k \ne 0$ satisfy the symmetry~$m_k = c \cdot \zeta_N^{-k}$
for some constant~$c$. 
The constraint~$\sum_{k=1}^{N} m_k = 0$ in~\eqref{constraints} then gives~$m_0 = 0$.
Then the second constraint in~(\ref{constraints}) gives 
$$\sum_{k=0}^{N} (2 m_k \zeta_N^k + 1) = 1 + \sum_{k=1}^{N} (2c + 1) = 0,$$
and hence~$c = -\frac{1}{2} - \frac{1}{2N}$. This determines all the~$m_k$, and turns~(\ref{firstS}) (still with $X=1/\wF_N$) into
$$\{\tau,1/\wF_N\} 
=
 \frac{1}{2 X^2} + \frac{1}{2} \sum_{k=1}^N \frac{1}{(X  - \zeta_N^k)^2}
- \frac{(1+N)}{2N} \sum_{k=1}^{N} \frac{\zeta_N^{-k}}{X - \zeta_N^k} = \frac{ (1 + (N^2  - 1) X^N)}{2 X^2 (X^N - 1)^2}.$$
From the chain rule, we deduce that with~$x = \wF_N = 1/X$
the equality:
$$\{\tau,\wF_N\} = \frac{1}{x^4} \frac{ (1 + (N^2  - 1) (1/x)^N)}{2 (1/x)^2 ((1/x)^N - 1)^2}
=
\frac{(N^2 - 1) x^{N-2} + x^{2N-2}}{2 (x^N- 1)^2},
$$
and from this we find that the equation~$y'' + \frac{1}{2} \{\tau,\wF_N\} y=  0$ 
is given by~(\ref{3.3}). We then conclude the proof by~\cite[Lemma~3.3]{Hempel}. 
\end{proof}

\begin{df}  \label{GN} Let~$G_N$ denote the map~$D(0,1) \rightarrow \C \setminus \{1\}$ such
that~$G_N(z^N) = (F_N(z))^N$, or equivalently~$G_N(z) = (F_N(z^{1/N}))^N$.
\end{df}

The fact that~$G_N$ is well-defined is a formal consequence of the relation~$F_N(\zeta z) = \zeta F_N(z)$ in Lemma~\ref{FNbasic}.

The inverse map of~$G_N$ is closely related to the inverse map of~$F_N$, and turns out to have a nicer form. We will give some geometric description of $G_N$ in \SSSS~\ref{GeoFNGN} in terms of triangle groups, which suggests an explicit description of the inverse of~$G_N$ in terms of hypergeometric functions.

\begin{lemma} \label{Gdiff} \label{sol-diffG}
Let~$\varphi_N$ denote the  local  inverse map of~$G_N$ around $x = 0$, normalized so that~$\varphi_N(0) = 0$.
The function~$\varphi_N$ has the form~$\delta^{-1}_N (\phi_1/\phi_2)^N$, where~$\phi_1$ and~$\phi_2$ are the solutions
to the differential equation:
\numequation\label{diff-eq-G}
x (x-1)^2 y'' + \left(1-\frac{1}{N} \right)(x-1)^2 y' + \left(\frac{1}{4} + \frac{x-1}{4 N^2} \right) y = 0
\end{equation}
given explicitly by
\numequation\label{sol-diff-G}
\phi_1 = \sqrt{1-x}  \cdot x^{1/N}  \cdot  \displaystyle{
\pFq{2}{1}{\frac{N+1}{2N}, \frac{N+1}{2N}}{1 + \frac{1}{N}}{x}
}, \quad
\phi_2 = \sqrt{1-x}   \cdot  \displaystyle{
\pFq{2}{1}{\frac{N-1}{2N}, \frac{N-1}{2N}}{1 - \frac{1}{N}}{x}
},
\end{equation}
and $\delta_N=|G_N'(0)|$ denotes conformal radius  of the map~$G_N$. 

Further, let~$s_N(x)$ denote the function
\numequation\label{sN}
s_N(x):=x^{1/N} \frac{
\displaystyle{
\pFq{2}{1}{\frac{N+1}{2N}, \frac{N+1}{2N}}{1 + \frac{1}{N}}{x}
}
}{
\displaystyle{
\pFq{2}{1}{\frac{N-1}{2N}, \frac{N-1}{2N}}{1 - \frac{1}{N}}{x}
}}.
\end{equation}
Then $\varphi_N(x)=\delta_N^{-1}s_N(x)^N$, $\psi_N(x)=|F'_N(0)|^{-1}s_N(x^N)$, and $\delta_N=|F'_N(0)|^{N}$.
\end{lemma}

%\numequation
%\phi_1 = 1 -    \frac{(N+1) z}{4N} + \cdots, \qquad \phi_2 = z^{1/N} \left(1  -  \frac{(N-1) z}{4N} + \cdots.  \right). \end{equation}

\begin{proof}
By Definition~\ref{GN}, $G_N(z)=(F_N(z^{1/N}))^N$; and by the assumptions $\varphi_N(0)=0$ and $\psi_N(0)=0$, we obtain the formal identity~$\varphi_N(x) = \psi_N(x^{1/N})^{N}$ (formally: $x=G_N(z)$).

Let $\eta_1, \eta_2$ denote the solutions of the differential equation in Lemma~\ref{diff-eq-FN} such that $\eta_1(0)=0, \eta'_1(0)=1, \eta_2(0)=1, \eta'_2(0)=0$; then $\eta_1, \eta_2$ are linearly independent and $\eta_1(x)/\eta_2(x)=x+O(x^2)$. Since $\psi_N(x) = |F_N'(0)|^{-1} x + O(x^2)$, then by Lemma~\ref{diff-eq-FN}, we have $\psi_N= |F_N'(0)|^{-1} \, \eta_1/\eta_2$.
We deduce
$$
\varphi_N(x) = |F_N'(0)|^{-N} (\eta_1(x^{1/N})/\eta_2(x^{1/N}))^N.
$$

Let~$\phi_i (x)= \eta_i(x^{1/N})$. Then 
$$
\phi'_i(x)=N^{-1}x^{\frac{1}{N}-1}\eta'_i(x^{1/N})
$$ 
and 
$$
\phi''_i (x)= \frac{1-N}{N^2}x^{\frac{1}{N}-2}\eta'_i +N^{-2}x^{\frac{2}{N}-2}\eta''_i(x^{1/N}).
$$
From equation~\eqref{3.3}, we have 
$$
4(x-1)^2\eta''_i(x^{1/N})+((N^2-1)x^{1-\frac{2}{N}}+x^{2-\frac{2}{N}})\eta_i(x^{1/N})=0.
$$
We rewrite this differential equation in terms of derivatives of $\phi_i$ using the above equations and then conclude that $\phi_1, \phi_2$ are solutions to~\eqref{diff-eq-G}. 

In order to prove that $\phi_i$ are given by the explicit formula in~\eqref{sol-diff-G}, we first deduce from the second order differential equations satisfied by hypergeometric functions that both $\sqrt{1-x}  \cdot x^{1/N}  \cdot  \displaystyle{
\pFq{2}{1}{\frac{N+1}{2N}, \frac{N+1}{2N}}{1 + \frac{1}{N}}{x}
}$ and $\sqrt{1-x}   \cdot  \displaystyle{
\pFq{2}{1}{\frac{N-1}{2N}, \frac{N-1}{2N}}{1 - \frac{1}{N}}{x}
}$ satisfy~\eqref{diff-eq-G}. (See, for instance, \cite[pp.~84--85]{Goluzin} on how to adjust by some rational power of $1-x$ to obtain a hypergeometric differential equation and then obtain the two solutions.) Moreover, we conclude that these explicit solutions are exactly $\phi_1$ and $\phi_2$ by noticing that they have the same leading terms as $\eta_1, \eta_2$ (once we replace $x$ by $x^N$).

The last assertion is just a summary of the above results.
\end{proof}

%If~$\wpsi_N$ is the local inverse of~$F_N$ around~$0$ so that~$\wpsi_N(0) = i$, then clearly $\psi_N(z)$ is a  M\"{o}bius translate of~$\psi_N$ and hence also of~$\eta_2/\eta_1$.
%But now from the equality~$F_N(\zeta q) = \zeta F_N(q)$ we deduce (from the form of~$\eta_1$ and~$\eta_2$ in equations~(\ref{eta1}) and~(\ref{eta2})) that
%\numequation \label{gammaN}
%\psi_N(z) = \gamma^{-1}_N \cdot \frac{\eta_2}{\eta_1} = \gamma^{-1}_N \left(z +  \frac{z^{N+1}}{2N} + \cdots \right)
%\end{equation}
%for some constant~$\gamma^{-1}_N$. Moreover, we certainly also have
%$$F_N(q) = \gamma_N q + O(q^2), $$
%and hence the conformal radius of~$F_N$ is given by~$|\gamma_N|$. 

%We find that equation~(\ref{3.3}) admits  solutions
%\numequation
%\label{eta1}
%\eta_1 = 1 -    \frac{(N+1) z^{N}}{4N} + \cdots  \in \Q \llbracket z^{N} \rrbracket, \end{equation}
%\numequation
%\label{eta2}
%\eta_2 = z -  \frac{(N-1) z^{N+1}}{4N} + \cdots  \in  z \cdot \Q \llbracket z^{N} \rrbracket,\end{equation}
%the inclusions following in an elementary way from the Frobenius method applied to~(\ref{3.3}).

In order to prove Theorem~\ref{exactconformal}, we need the following formula of the behavior of hypergeometric functions in Lemma~\ref{Gdiff} near $x=1$.

\begin{lemma}[See, for instance, {\cite[15.3.10]{AS}}]\label{lem-asymp2F1} Given $a\notin \Z_{\leq 0}$, for $|x|<1$, we have
 \numequation\label{asymp2F1}
   \displaystyle{
\pFq{2}{1}{a,a}{2a}{1-x}
=  \frac{\Gamma(2a)}{\Gamma(a)^2}
\sum_{k=0}^\infty \frac{(a)_k (a)_k}{k!^2} x^k
\left( - \log{x} + 2 (\psi(k+1) - \psi(k+a)) \right),
}
 \end{equation}
 where $\psi$ denotes the digamma function~$\psi(x) = d/dx \log \Gamma(x) = \Gamma'(x)/\Gamma(x)$.
 Note that the above hypergeometric function is multivalued around~$x = 0$, but all different branches are
accounted for by the branches of the logarithm.  
%\qed
\end{lemma}

\begin{proof}[Proof of Theorem~\ref{exactconformal}] %From equation~(\ref{gammaN}), it suffices to compute the limit of~$\phi_2/\phi_1$ as $z\rightarrow 1$, which we can do directly given the explicit form in terms of hypergeometric functions (see also \eqref{asymp2F1}).
Since $F_N$ is a covering map of $\C \setminus \mu_N$, the local inverse $\psi_N$ is naturally defined on $D(0,1)\subset \C\setminus \mu_N$. Moreover, since $F_N(1)=1$, we have that $\lim_{x\rightarrow 1} \psi_N(x)=1$, where $x\in D(0,1)$ approaches $1$. (\emph{A priori}, we only conclude that $\lim_{x\rightarrow 1} \psi_N(x)$ approaches a cusp of $D(0,1)$, and thus that $|\lim_{x\rightarrow 1} \psi_N(x)|=1$; this suffices for the rest of the proof. The more precise statement $\lim_{x\rightarrow 1} \psi_N(x)=1$ follows from our assumption $\psi_N(0)=0$, either by the description of the fundamental domain further down in Lemma~\ref{prop-GammaN}, or more directly by the computation of the rest of the proof that follows, which shows  that $\lim_{x\rightarrow 1} \psi_N(x)$ is a positive real number.)  

By Lemma~\ref{Gdiff}, we have $\psi_N(x)=|F'_N(0)|^{-1}s_N(x^N)$. In particular, 
$$
|F'_N(0)|^{-1}\lim_{x\in D(0,1), x\rightarrow 1} s_N(x^N) = 1.
$$
 Thus, by~\eqref{sN} and~\eqref{asymp2F1}, we have 
 \[|F'_N(0)|=\lim_{x\in D(0,1), x\rightarrow 1} s_N(x^N)\]
 \[=\lim_{x\in D(0,1), x\rightarrow 1} \frac{\displaystyle{\pFq{2}{1}{\frac{N+1}{2N}, \frac{N+1}{2N}}{1 + \frac{1}{N}}{x}}}{\displaystyle{ \pFq{2}{1}{\frac{N-1}{2N}, \frac{N-1}{2N}}{1 - \frac{1}{N}}{x}}}=  \frac{
\displaystyle{
 \Gamma \left( \frac{N - 1}{2N} \right)^2  \Gamma \left(1 + \frac{1}{N} \right)
}
}{
\displaystyle{
 \Gamma \left( \frac{N + 1}{2N} \right)^2  \Gamma \left(1 - \frac{1}{N} \right)
}}.\]

Basic properties of the Gamma function~\cite[6.1.18]{AS} transform the latter expression into
\[\gamma_N=2^{4/N}\frac{
\displaystyle{
 \Gamma \left( 1+\frac{1}{2N} \right)^2  \Gamma \left(1 - \frac{1}{N} \right)
}
}{
\displaystyle{
 \Gamma \left( 1-\frac{1}{2N} \right)^2  \Gamma \left(1 + \frac{1}{N} \right)
}}.\]
Then by~\cite[6.1.33]{AS}, we also have
\[
\log \gamma_N = \frac{\log 16}{N} +
\sum_{k=1}^{\infty} \frac{  (2^{2k} - 1)}{2^{2k-1} (2k+1)}
\cdot \frac{ \zeta(2k+1) }{N^{2k+1}}.
\]
We obtain~\eqref{explicit-gammaN} by taking the exponential of the above formula.
\end{proof}

\begin{example} \label{ex23}   \emph{If~$N=2$, then~$\C \setminus \{\pm 1\}$ is biholomorphic to~$Y(2) = \mathbf{P}^1 \setminus \{0,1,\infty\}$, 
and a direct description of the uniformization~$\H \rightarrow \C \setminus \{\pm 1\}$  sending~$i$ to~$0$ is given by~$2 \lambda(\tau) - 1$.
In this case, the formulas above specialize to the standard identity~$q = e^{-\pi K'/K}$ where the elliptic periods~$K'$ and~$K$ are
directly related to hypergeometric functions. 
The only other such case of an incidental isomorphism $\C \setminus \mu_N \cong Y(N)$ is $N = 3$: this is~\cite[\SSS~6 Example~5]{Hempel}. In our notation, 
these two respective uniformization maps~$F_N : D(0,1) \to \C \setminus \mu_N$ are explicitly
$$
F_2 : D(0,1) \to \C \setminus \{\pm 1\}, \qquad F_2(z) = 2 \lambda \left( i \frac{1-z}{1+z} \right) - 1
$$
and
\begin{equation*}
\begin{aligned}
& F_3 : D(0,1) \to \C \setminus \mu_3,  \\  F_3(z)  & =   9\eta\left( 9i \frac{( 2 + \sqrt{3} + 3i)z + 2 + \sqrt{3} - 3i }{ (-6-\sqrt{3}+3i)z + 6 + \sqrt{3} + 3i } \right)^3  \eta\left( i \frac{( 2 + \sqrt{3} + 3i)z + 2 + \sqrt{3} - 3i }{ (-6-\sqrt{3}+3i)z + 6 + \sqrt{3} + 3i }  \right)^{-3} + 1,
\end{aligned}
\end{equation*}
where~$\eta(z) := q^{1/12} \prod_{n=1}^{\infty} (1-q^{2n})$ with~$q := e^{\pi i z}$ is the Dedekind eta function. These explicit examples confirm our general formula for the derivative: 
\begin{equation*}
\begin{aligned}
 |F_2'(0)| & = \frac{\Gamma(1/4)^4}{4\pi^2} = 4.37687923\ldots &  > \sqrt{16}  \\
|F_3'(0)| & = \frac{\Gamma(1/6)^3}{12\pi^{3/2}} = 2.5810565\ldots & > 2.519842\ldots = \sqrt[3]{16}. 
\end{aligned}
\end{equation*}}
\end{example}

\subsection{Geometry of~\texorpdfstring{$\Gamma_N$}{GammaN} and a uniform growth estimate of $F_N$}\label{GeoFNGN}
Our second aim in the present~\SSSS~\ref{uniformizations} is the uniform supremum growth estimate Lemma~\ref{trivialsupFN} of the universal covering map $F_N : D(0,1) \to \C \setminus \mu_N$ near the boundary,  as both the circle $|z| = r$ and the level $N$ vary. This result is subsumed by the more precise bound of Kraus and Roth~\cite[Theorems~1.2 and~1.10]{KrausRoth}; our treatment is self-contained. In \SSSS~\ref{nevanlinna} we will refine this  supremum growth bound  (uniformly exponential in $\frac{N}{1-r}$) to an integrated growth bound (uniformly linear in $\frac{N}{1-r}$). 

The idea of the proof is to use the symmetry of $\C\setminus \mu_N$ and the action of the Fuchsian group $\wGamma_N$ to reduce the question to the study of the asymptotic of $\wF_N$ near the cusp $\tau= i \infty$. We study this asymptotic using the explicit description of the inverse of $\wF_N$ given in Lemma~\ref{Gdiff}. 

We begin in this subsection by explicitly describing the Fuchsian group $\wGamma_N$. 

%Recall that~$\wF_N(\tau | r_N) = \zeta \wF_N(\tau)$, where~$r_N \in \SO_2(\R)$ corresponds in~$D(0,1)$ to multiplication by~$\zeta$.
%Explicitly, we have
%\numequation %\label{elementr}
%r_N = \left( \begin{matrix}  \cos(\pi/N) & -\sin(\pi/N) \\ \sin(\pi/N) & \cos(\pi/N) \end{matrix} \right).
%\end{equation}
%The element~$r_N$ acts transitively on the cusps which include~$i \infty$, and hence~$N$ of the cusps are given
%by~$\cot(\pi k/N)$ for any integer~$k$. For example, if~$N$ is even, then we can take~$k = (N/2)$ and see that~$0$
%is a cusp.

\begin{proposition}\label{prop-GammaN}
The stabilizer of~$i \infty$ in~$\wGamma_N$ is generated by
$$\wt_N:=\left( \begin{matrix} 1& 2 \cot(\pi/2N) \\ 0 & 1 \end{matrix} \right).$$
The group~$\wGamma_N$  is the free group on~$N$ generators given by~$\wt_N$ and its conjugates by powers of~$\wr_N$ in~\eqref{elementr}.
A fundamental domain $\Omega\subset D(0,1)$ for $\Gamma_N$ is  the region with~$2N$ cusps given by half-integer powers of~$\zeta_N= \exp(2 \pi i/N)$ and bounded by geodesics connecting adjacent cusps.
\end{proposition}

\begin{proof} Since $\wF_N$ is the universal covering map of $\C\setminus \mu_N$,  the Fuchsian group~$\wGamma_N$ is  generated by the stabilizers of the cusps~$c$ with~$\wF_N(c) \in \mu_N$. If we denote
the generator of the stabilizer of~$i \infty$ by
\numequation
\wt:=\left( \begin{matrix} 1& c_N \\ 0 & 1 \end{matrix} \right),
\end{equation}
then the stabilizers of the other cusps associated to~$\mu_N$ are generated by the conjugates of~$\wt_N$ by~$\wr_N$ since $\zeta_N^k\wF_N(\tau)=\wF_N(\wr_N^k\cdot \tau)$ by Lemma~\ref{FNbasic}.

Consider the Dirichlet domain~$\Omega_N\subset D(0,1)$ associated to~$\Gamma_N$
around~$z = 0$. More precisely, we can describe~$\Omega_N$ as the region
$$
\{z \in D(0,1) \, : \,   d(g z,0) \ge  d(z,0)  \ \text{for all} \ g, g^{-1} \in \{r_N^k \, t \, r_N^{-k}\}, k = 0,1,\ldots, N-1 \}.
$$
Here, $d$ is the hyperbolic distance in $D(0,1)$.
The region in $D(0,1)$ such that~$d(g z,0) \ge d(z,0)$ and~$d(g^{-1}z,0) \ge d(z,0)$ for $g=r_N^k \cdot t \cdot r_N^{-k}$
is the region bounded by two geodesics starting at~$\zeta_N^k$ going in opposite directions and intersecting the boundary
at~$\zeta_N^k e^{\pm i \theta}$ where~$c_N = 2 \cot(\theta/2)$.  There are exactly~$2N$ such arcs corresponding to the~$N$
generators and their inverses.
In particular,
if~$\theta < \pi/N$ is too small, the fundamental region will have infinite volume, whereas if~$\theta > \pi/N$ is too big, then
the Dirichlet domain will only contain at most~$N$ cusps. Since $\H/\wGamma_N$ has~$N+1$ cusps and~$\wGamma_N$ has finite covolume, we must have~$c_N = 2 \cot(\pi/2N)$ and these geodesics intersecting at $\zeta_N^{k+1/2}$
for $k=0,\ldots, N-1$.
\end{proof}

%It follows from this that the cusp width at this remaining cusp will be~$N c_N = 2N \cot(\pi/2N)$,
%and that all the remaining~$N$ cusps are in the same orbit of~$\wGamma_N$. As an example, note that
%$e^{k \pi i/N}$ in the Poincar\'{e} disc model corresponds to
%\numequation
%\label{othercusp}
% i \cdot \frac{1 + e^{- k \pi i/N}}{1 - e^{-k \pi i/N}} = \cot(\pi k/2N)
%\end{equation}
%in~$\partial \H$. But now we have (for example)
%\numequation
%r^{m}_N t_N r^{-m}_N \cdot  \cot \left( \frac{\pi (2m-1)}{2N} \right)  =  \cot \left( \frac{\pi (2m+1)}{2N} \right).
% \end{equation}
 
 Now we consider the group associated to $\wF_N^N$.
 \begin{df}
 Let~$\wPhi_N$ denote the group~$\langle \wGamma_N, \wr_N \rangle=\langle \wr_N, \wt_N \rangle$ and let $\Phi_N$ denote the corresponding lattice in $\PSU(1,1)$.
 \end{df}
 
\begin{cor}\label{corPhiN}
The function~$\wF^N_N$ is invariant under~$\wPhi_N$ and $\wPhi_N$ is the largest subgroup of $\PSL_2(\R)$ with this property. A fundamental domain $\Omega'_N$ for~$\Phi_N$ in $D(0,1)$ is given by the hyperbolic quadrilateral with vertices~$0, \zeta_N^{-1/2}, 1, \zeta_N^{1/2}$.
Translated to~$\H$ this is bounded by
geodesics from~$i$ to~$\cot(\pi/2N)$ to~$i \infty$ to~$-\cot(\pi/2N)$ and back to~$i$.
\end{cor}
\begin{proof}
The statements follow directly from Lemma~\ref{FNbasic} and Proposition~\ref{prop-GammaN}.
\end{proof}

An example of the fundamental domain in Corollary~\ref{corPhiN}
for~$N=3$ is given in Figure~\ref{domain}, which also includes translates
of the domain by elements in~$\Omega'_3$
by words in~$\{r_3,r^2_3,t_3,t^{-1}_3\}$ of length at most~$6$.
The shading reflects where the absolute value of~$|F^3_3 - 1|$ is small --- it
vanishes precisely at the cusps corresponding to~$z=1$.

\addtocounter{subsubsection}{1}
 \begin{figure}[!h]  
\begin{center}
  \includegraphics[width=50mm]{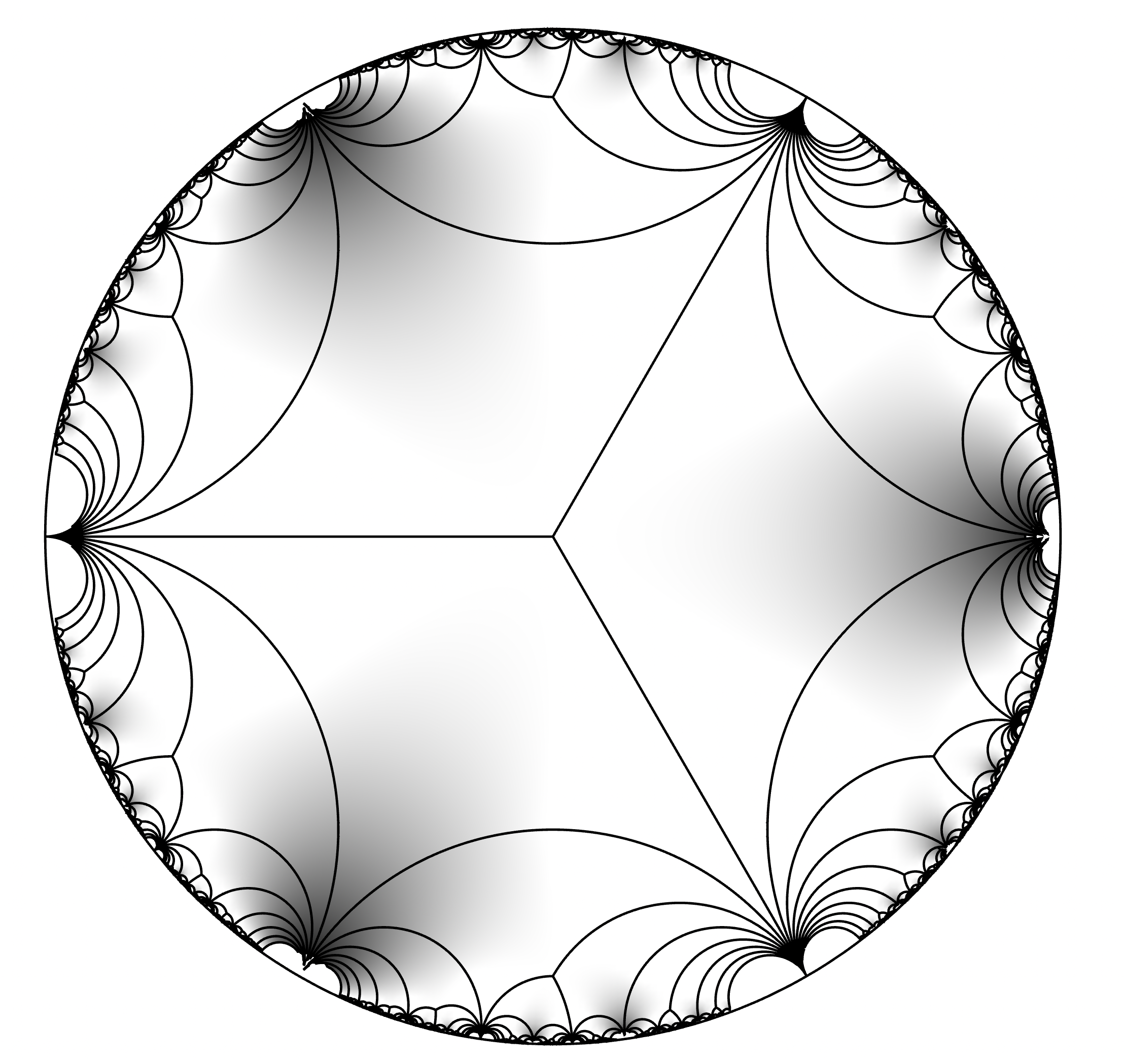}
\end{center}
\caption{A (partial) tiling of~$D(0,1)$ by a fundamental
domain for~$\Phi_3$, together with a density plot of~$|F^3_3 - 1|$ which vanishes at~$z=1$.}
   \label{domain}
\end{figure}

%Since~$r_N$ normalizes~$\wGamma_N$, this contains~$\wGamma_N$ as a normal subgroup with~$\Phi_N/\wGamma_N \simeq \Z/N \Z$.
%We have~$\wPhi_N \simeq  \Z/N\Z \ast \Z$ where~$r_N$ is as in equation~(\ref{elementr}),

\begin{lemma}   \label{inversion symmetry}
Let $\ws\in \PSL_2(\R)$ be a rotation of order~$2N$ such that $\ws^2=\wr_N$.
Then
\numequation
\label{twist}
1 - \wF^N_N(\ws \cdot \tau) =  \frac{1}{1 - \wF^N_N(\tau)}.
\end{equation}
\end{lemma}

\begin{proof}
Recall our normalization of $\wF_N$ that $\wF_N(i)=0$ and $\wF_N(i \infty)=1$. By Corollary~\ref{corPhiN}, $\tau=\pm \cot(\pi/2N)$ (in the same $\wPhi_N$-orbit) is the other cusp of $\wPhi_N$ and thus $\wF_N(\pm \cot(\pi/2N))=\infty$. Moreover,
both sides of~\eqref{twist} are uniformizers of~$\H/\wPhi_N$ which take the value~$0$ at the cusp $\cot(\pi/2N)$ and~$\infty$ at the cusp $i\infty$.
This specifies them uniquely up to~$x \mapsto \lambda x$ scalings. However, this last ambiguity is
removed by noting that both sides are $1$ at~$\tau = i$.
\end{proof}

%We also deduce from this the equation
%$$
%1 - G_N(-q) =   \frac{1}{1 - G_N(q)}.
%$$

%From the fact that~$F_N$ is the universal covering map, we also immediately deduce the following geometric description of~$G_N$.

%\begin{lemma} The map~$G_N: D(0,1) \rightarrow \C \setminus \{1\}$ is a covering map away from the point~$0 \in \C \setminus \{1\}$.
%The map~$G_N$ is locally an isomorphism in a neighbourhood of~$0$ in~$D(0,1)$, but it is totally ramified of degree~$N$ at all other
%preimages of~$0$. Moreover, it is universal with respect to any such map.
%\end{lemma}

\begin{remark}
The group~$\wPhi_N$ is contained with index two in the larger group~$\wwPhi_N = \langle \ws,\wt_N \rangle$.
The group~$\wwPhi_N$ has a fundamental domain consisting of the points~$0,1,\zeta_N^{1/2}$.
But this is none other than
a  hyperbolic triangle with angles~$\{\alpha,\beta,\gamma\} = \{\pi/N,0,0\}$, whose conformal mapping from $\H$ is given by Schwarz triangle functions (see, for instance, \cite[\SSS~404 on page~185]{Caratheodory}).
This suggests that~$\wF_N$ should directly be related to Schwarz triangle functions,
which leads to a direct description of the inverse functions~$\psi_N$ and~$\varphi_N$
in terms of hypergeometric functions in Lemma~\ref{Gdiff}.
\end{remark}

In order to study the behavior of $\wF_N$ near $x=1$, we use the explicit formula for its inverse $\psi_N$ given in Lemma~\ref{Gdiff}. The following lemma gives the asymptotic of the function $s_N$ used in formula of $\psi_N$.

% \subsection{Uniform asymptotics for hypergeometric functions}

 \begin{lemma}  \label{uniform} Fix a real constant~$M_0 > 0$. 
 For~$M\geq M_0$ and $|x| < e^{-M N}$, we have the uniform estimate:
 \numequation
\left| \frac{s_N(1 - x)}{\gamma_N} -(1 - x)^{1/N}  \frac{-\log x + 2 \gamma -2  \psi(1/2 + 1/2N)}{-\log x + 2 \gamma - 2 \psi(1/2 - 1/2N)}\right|
\ll \frac{|x|}{N},
\end{equation}
where $\psi$ is the digamma function as in Lemma~\ref{lem-asymp2F1},
and the implicit constant depends on~$M_0$ but not on~$N,M$.
\end{lemma}

 \begin{proof}
By Lemma~\ref{lem-asymp2F1}, we have for~$a =1/2 \pm 1/2N$ and~$|x| < 1$ the following equality:
 \numequation\label{2F1temp}
   \displaystyle{
\pFq{2}{1}{a,a}{2a}{1-x}
=  \frac{\Gamma(2a)}{\Gamma(a)^2}
\sum_{k=0}^\infty \frac{(a)_k (a)_k}{k!^2} x^k
\left( - \log{x} + 2 (\psi(k+1) - \psi(k+a)) \right).
}
 \end{equation}
We first prove that the coefficients in this power series are uniformly bounded.
Since~$|a| < 1$, we have~$|(a)_k|/k! < 1$. Basic properties of the digamma function (cf.~\cite[6.3.5, 6.3.14]{AS}) show that $\psi(x)$ is negative and strictly increasing for $0<x<1$, and $|\psi(k)-\psi(k+a)|> |\psi(k+1) -\psi(k+1+a)|$. Hence, $|\psi(k) - \psi(k + a)|$ is maximized when~$k=1$ and~$a = 1/2 - 1/4$.
This immediately leads to  the uniform estimates
$$ \left| \sum_{k=1}^{\infty} \frac{(a)_k (a)_k}{k!^2} x^k \right|,
\quad \left| \sum_{k=1}^{\infty} \frac{(a)_k (a)_k}{k!^2} x^k  \left(2 (\psi(k+1) - \psi(k+a)) \right) \right|
\ll |x| < e^{-MN}$$
For~$|x| < e^{-MN}$, we also have~$\Re(\log{x}) \le - M N$, and so in particular~$|\log{x}| \ge M N$ regardless
of the branch of logarithm. Combined with Lemma~\ref{Gdiff} and \eqref{2F1temp}, this leads to the estimate
$$\frac{s_N(1-x)}{\gamma_N (1-x)^{1/N}} =  \frac{-\log{x} + 2 \psi(1) - 2\psi(1/2 + 1/2N) + O(x)}{-\log{x} + 2 \psi(1) - 2\psi(1/2 - 1/2N) + O(x)}$$
where the implicit constants are uniform in~$N$, from which the result follows (using that~$|\log{x}| \gg N$).
 \end{proof}

% \subsection{The region~\texorpdfstring{$F^N_N \sim 1$}{FNnear1} and~\texorpdfstring{$F_N \sim \infty$}{FNnearoo}}

 \begin{lemma} \label{fnearone} Fix a pair of real positive numbers~$M_0 > 0$ and $\epsilon > 0$. Consider any $M\geq M_0$, and let $\wOm'_N\subset \H$ denote the fundamental domain for $\wPhi_N$ corresponding to~$\Omega'_N$ in Corollary~\ref{corPhiN}.
 If~$\tau \in \wOm'_N$ has
 $$\|\wF_N(\tau)^N - 1\| < e^{-MN},$$
  then 
 \numequation  \label{horo}
\Im(\tau) >  \frac{2 N^2 M}{\pi^2} (1 - \epsilon)
\end{equation}
once $N\gg_{\epsilon, M_0} 1$, where the implicit constant depends only on~$M_0$ and $\epsilon$.
\end{lemma}

\begin{proof}
By Corollary~\ref{corPhiN}, $\wOm'_N$ is a fundamental domain of~$\wF^N_N$ and the only cusp where~$\wF_N^N = 1$
is at~$\tau=i \infty$.
So it suffices to consider~$\wF^N_N$ in a neighbourhood of the cusp~$i \infty$.

For~$N$ sufficiently large, the inequality~$ \|\wF_N(\tau)^N - 1\| < e^{-MN}$
implies that~$|\wF_N(\tau) - 1| < e^{-M(1 - \varepsilon_0) N}$
for some~$\varepsilon_0$ that tends to zero as~$N$ increases. 
%Hence, replacing~$M$ by~$M(1 - \varepsilon)$, we may equivalently assume that, for the branch of~$\psi_N$ such that~$\psi_N(1) = 1$ and~$\wpsi_N(1) = i \infty$, that
Recall from the proof of Theorem~\ref{exactconformal} that $\psi_N(1)=1$ and hence $\wpsi_N(1)=i\infty$; then it suffices to bound the imaginary part of
$$
\tau = \wpsi_N(1 - x), \qquad \text{for } |x| < e^{-M(1-\varepsilon_0)N}.
$$
We may write this as
\numequation
\tau = i \cdot \frac{1 + \psi_N(1-x)}{1 - \psi_N(1-x)}
= i  \cdot \frac{\gamma_N + s_N((1-x)^N)} {\gamma_N - s_N((1-x)^N)}.
\end{equation}
Writing $1-X := (1- x)^N$, then the same estimate as above implies~$|X| < e^{-M(1 - 2\varepsilon_0)N}$
for sufficiently large~$N$. Thus we reduce the lemma to the estimate of
\numequation
\tau = i  \cdot \frac{\gamma_N + s_N(1-X)} {\gamma_N - s_N(1-X)}, \quad |X| < e^{-M(1- 2\varepsilon_0)N}.
\end{equation}
Now by Lemma~\ref{uniform} and \cite[6.3.7]{AS}, we have
$$\tau =  \frac{i \cot(\pi/2N)}{\pi} \left(2 \gamma - \log{X} - \psi(1/2 - 1/2N) - \psi(1/2 + 1/2N) \right) + O(1),$$
where the implicit constant only depends on $M_0$ and $\varepsilon_0$.
The imaginary part of this does not depend on the choice of branch of~$\log{X}$ and indeed only
depends on~$|X|$, and we deduce with this approximation that
$$
\Im(\tau) \ge
 \frac{ \cot(\pi/2N)}{\pi} \left(2 \gamma + NM (1-2 \varepsilon_0) - \psi(1/2 - 1/2N) - \psi(1/2 + 1/2N) \right) +O(1).
$$
If we choose $\varepsilon_0 := \epsilon/2$, then for $N \gg_{\epsilon, M_0} 1$ this lower bound clearly exceeds 
$$
\frac{2 N^2 M (1-\epsilon)}{\pi^2}
$$
as desired.
\end{proof}

The region~\eqref{horo} is a \emph{horoball} for the cusp $i\infty$ in the upper half plane model $\H$ of the hyperbolic plane. Recall that in the Poincar\'{e} disc model~$D(0,1)$, the horoballs are the euclidean discs inside~$D(0,1)$ which are tangent to the boundary circle. The following easy lemma describes how these horoballs transform under the hyperbolic isometry group. 

\begin{lemma} \label{translates} Let~$\HH_D$ denote the image in the Poincar\'e disc model~$D(0,1)$ of the horoball
$$
\{  \tau \in \H \, : \, \Im(\tau) \ge D \}
$$
in the upper half plane model $\H$. 
  For $\gamma\in \PSU(1,1)$ with image $\widetilde{\gamma}=\begin{pmatrix} a & b \\ c & d \end{pmatrix} \in \PSL_2(\R)$, 
the image~$\gamma \HH_D$ of $\HH_D$ under $\gamma$ in~$D(0,1)$ is the disc with diameter
\numequation
E(\gamma,D):=\frac{2}{1 + D(a^2 + c^2)}
\end{equation}
tangent to the boundary circle $\T$ at the point~$(a - i c)/(a + i c)$.   \qed
\end{lemma}

We  close this section by using Lemma~\ref{fnearone} to derive
 a coarse yet fairly uniform upper bound on $\sup_{|z|=r} \log |F_N(z)|$. Although  not best-possible, it  is enough as an input for the logarithmic error term in the Nevanlinna theory estimate in \SSSS~\ref{nevanlinna}.

\begin{lemma}\label{trivialsupFN}
For $N \gg 1$ and $r\in (0,1)$, we have
\[\sup_{|z|=r} \log |F_N(z)|\ll \frac{N}{1-r},\]
where the implicit constants are both absolute.
\end{lemma}
\begin{proof}
Set $S(M,N):=\{z\in D(0,1) \, : \, |F_N^N(z)-1|< e^{-MN}\}$ and $M := \frac{N}{1-r}$. Since $M\geq 2$, we may take $M_0=2$ and $\epsilon=1/2$ in Lemma \ref{fnearone} and conclude that for $N\gg 1$ (with absolute implicit constant here), we have
$$
S(M,N)   \subset  \bigcup_{\gamma \in \Phi_N}\gamma \HH_D,  \quad \textrm{where }  D=\frac{N^2M}{\pi^2}.
$$

By Shimizu's Lemma (see, for example, \cite[Theorem~3.1]{ElstrodtGrunewaldMennicke}) and Proposition~\ref{prop-GammaN}, we have ~$2(|a|+|c|) \cot(\pi/2N) \ge 1$ for all $\widetilde{\gamma} =\left( \begin{matrix} a & b \\ c & d \end{matrix} \right) \in \widetilde{\Gamma}_N$. Thus by Lemma \ref{translates}, \[E(\gamma,D)\leq 2D^{-1}(a^2+c^2) \ll  N^{-2}N^{-1}(1-r) N^2=\frac{1-r}{N},\]
where the implicit constant is absolute. Thus we have $E(\gamma,D)\leq 1-r$ for all $\gamma \in \Phi_N$ once $N^{-1}$ times the implicit constant is less than $1$.

By Lemma \ref{inversion symmetry}, the set $\{z \in D(0,1) \, : \, |F_N(z)|> e^{M}+1\}$ is contained in $S(M,N)$, which is contained in $D(0,1) \setminus \overline{D(0,r)}$ by the above argument for $N\gg 1$. Thus we conclude that
\[\sup_{|z|=r} \log |F_N(z)| \leq \log(e^M +1) \ll M =\frac{N}{1-r}.\qedhere\]
\end{proof}

\begin{remark}  \label{Renggli note}
A more refined bound is proved in Kraus--Roth~\cite[Theorems~1.2 and 1.10]{KrausRoth}. On the other hand,
one can push our method further and prove, with rather more work but uniformly in $N \in \NwithoutzeroA$ and $M \in [1,\infty)$, that the supremum region $|F_N| < e^M$ is simply connected of conformal radius $1 - O(M^{-2}N^{-3})$
from the origin; this is a sharp estimate.  But taking for $\varphi$ in Corollary~\ref{holonomy} the pullback of $F_N$ by the Riemann map of some such region $|F_N| < e^M$, and
ignoring thus the fine savings from the integrated bound~(\ref{dimensionbound}) as opposed to the supremum,  would only lead to an $O(N^4)$ holonomy rank bound in place of our requisite logarithmically inflated bound $O(N^3\log{N})$. In the next section we will see how to make the full use of the integrated holonomy bound, and use Nevanlinna's value distribution theory to supply our final piece of the proof of the unbounded denominators conjecture. 
\end{remark}

\section{Nevanlinna theory and uniform mean growth near the boundary}  \label{nevanlinna}

For our application of Corollary~\ref{holonomy}, 
we prove in this section the following uniform growth bound. Throughout this section, we assume as
we may that $N \geq 2$. Then the analytic map $F_N :D(0,1) \to \mathbb{P}^1$ omits the $N+1 \geq 3$ values $\mu_N \cup \{\infty\}$. In such a situation, we seek to exploit  whatever growth constraints are imposed on the map by Nevanlinna's value distribution theory. A theorem of Tsuji~\cite[Theorem~11]{Tsuji} gives the general asymptotic
$$
\int_{|z| = r} \log^+{|F|} \, \mu_{\mathrm{Haar}}  =  \frac{1}{N-1} \log{\frac{1}{1-r}} + O_{a_1, \ldots, a_N}(1),
$$ 
 for any universal covering map $F : D(0,1) \to \C \setminus \{a_1, \ldots, a_N\}$ based at $F(0) = 0$ (see also the discussion in Nevanlinna~\cite[page~272]{Nevanlinna}), however this is only asymptotically
in $r \to 1^-$ for given punctures $\{a_i\}$ whereas we need a uniformity in both $r$ and $N$.
It is at the point~\eqref{choice of p} exploiting the small\footnote{Precisely, the relevant point of the specific puncture set~$\{a_1,\ldots, a_N\} \cup \{\infty\} = \mu_N \cup \{\infty\}$ in~$\mathbb{P}^1 \setminus D(0,1)$ is that the degree-$N$ polynomial~$\prod_{i=1}^{N} (x-a_i) \in \C[x]$ has~$N^{O(1)}$ coefficients.} coefficients of~$\prod_{\zeta \in \mu_N} (x-\zeta) = x^N -1$
% of the explicit partial fraction decomposition of the logarithmic derivative $\displaystyle \frac{(x^N - 1)'}{x^N-1} = \frac{N}{x} \frac{x^N}{x^N-1}$ of~$\prod_{\zeta \in \mu_N} (x-\zeta)$ 
that 
our argument below makes a critical use of the special feature of the target set $\mu_N \cup \{\infty\}$ of omitted values.

\begin{thm}  \label{growth term}
For each of the choices
$$
p(x) \in \big\{ x^N,  x^N / (x^N - 1), \, 1/(x^N-1) \big\}, 
$$
we have uniformly in $N \in \NwithoutzeroA$ and $r \in (0,1)$ the mean growth bound
\numequation \label{mean proximity bound}
\int_{|z| = r} \log^+|p \circ F_N| \, \mu_{\mathrm{Haar}} \ll \log{\frac{N}{1-r}}, 
\end{equation}
with some (effectively computable) absolute  implicit constant. 
\end{thm}

\subsection{Preliminaries in Nevanlinna theory} This section collects some standard material from Nevanlinna's value distribution theory. The reader should feel encouraged  to skip this part on a first reading, and refer back as necessary.

\subsubsection{The Nevanlinna characteristic} \label{integratedgrowthbound} 
The left-hand side of \eqref{mean proximity bound} is known as the \emph{mean proximity function at $\infty$} 
$$
m(r, f) = m(r,f ; \infty) := \int_{|z| = r} \log^+{|f|} \, \mu_{\mathrm{Haar}} \in [0,\infty).
$$
 It is complemented by the \emph{counting function}
$$
N(r,f)= N(r,f; \infty) := \sum_{\rho \, : \, 0 < |\rho| < r} \mathrm{ord}_{\rho}^-(f) \log{\frac{r}{|\rho|}} 
+ \mathrm{ord}_0^-(f) \, \log{r},
$$
where, in general for a meromorphic mapping $f : D(0,1) \to \mathbf{P}^1$, we denote by $\mathrm{ord}_{\rho}^-(f) :=  \mathrm{ord}^+(1/f) = \max(0, \mathrm{ord}(1/f))$  the pole order (if $\rho$ is a pole, and $0$ if $f$ is holomorphic at $\rho$).  

The \emph{Nevanlinna characteristic function}
$$
T(r,f)  := m(r,f) + N(r,f)
$$
is the well-behaved quantity functorially. 

\begin{lemma} \label{counting is positive}
For every meromorphic function $f : D(0,1) \to \mathbf{P}^1$ regular at $0$ (that is: with $f(0) \neq \infty$), and for every $r \in (0,1)$, we have
$$
N(r,f)  \geq 0, 
$$
with equality if and only if $f$ is holomorphic (has no poles) throughout the disc $D(0,r)$. 

The Nevanlinna characteristic function $T(r,f)$
satisfies for every $a \in \C$ the relation
\numequation \label{first main theorem}
|T(r,f)  - T(r, 1 / (f-a)) - \log{|c(f,a)|} |  \leq \log^+{|a|} + \log{2}, 
\end{equation}
where
$$
c(f,a) := \lim_{z \to 0} (f(z)-a) z^{-\mathrm{ord}_0(f-a)}.
$$
\end{lemma}

\begin{proof} 
This is Rolf Nevanlinna's \emph{first main theorem}, and is proved formally and straightforwardly from the Poisson--Jensen formula (see, for instance, \cite[Proposition~13.2.6]{BombieriGubler}), which we may rewrite as
\numequation \label{Jensen}
T(r,f)-T(r,1/f)=\log |c(f,0)|,
\end{equation}
 and the triangle inequality relation
\numequation \label{trivial}
\big| \log^+{|f-a|} - \log^+{|f|} \big| \leq \log^+{|a|} + \log{2}.
\end{equation}
See Hayman~\cite[Theorem~1.2]{Hayman} or Bombieri--Gubler~\cite[Theorem~13.2.10]{BombieriGubler} for the details. 
We note that $c(f,a) = f(0) - a$ when $a \neq f(0)$.
\end{proof}

\subsubsection{The lemma on the logarithmic derivative}

The  {\it lemma on the logarithmic derivative}---a strong explicit form of which is cited in~\eqref{log derivative lemma} below---is the centerpiece of Rolf Nevanlinna's original analytic proof of his 
 \emph{second main theorem} of value distribution theory. The logarithmic error feature of this sharp upper bound on the proximity function of a logarithmic derivative enables us to derive Theorem~\ref{growth term} from the relatively crude supremum growth bound in Lemma~\ref{trivialsupFN}. 
 
 The  reader willing to take~\eqref{log derivative lemma} for granted may at this point proceed directly to \SSSS~\ref{proof resumed}. Nevertheless, since the proof simplifies considerably in the case that we need of a \emph{functional unit} (a nowhere vanishing holomorphic function), we include our own self-contained treatment of a basic explicit case of the lemma on the logarithmic derivative.

\begin{lemma}  \label{log derivative self-contained}
Let $g : \overline{D(0,R)} \to \C^{\times}$ be a nowhere vanishing holomorphic function on some open neighborhood of the closed disc $|z| \leq R$. 
Assume that $g(0) = 1$. Then, for all $0 < r < R$,
\numequation \label{log derivatives are small}
m\Big(r , \frac{g'}{g}\Big) <  \log^+{\Big\{ \frac{m(R,g)}{r} \frac{R}{R-r} \Big\}}
+ \log{2} + 1/e.
\end{equation}
\end{lemma}

\begin{proof}
Our functional unit assumption means that the function  $\log{g(z)}$ has a single valued holomorphic branch on a neighborhood of the closed disc $|z| \leq R$ with
$\log{g(0)}  = 0$. Its real part is the harmonic function $\log{|g(z)|}$. Poisson's formula 
on the harmonic extension of a continuous function from the boundary to the interior of a disc reads
\numequation \label{Poisson}
\log{|g(z)|} = \int_{|w| = R} \log{|g(w)|} \cdot \mathfrak{R} \Big( \frac{w+z}{w-z} \Big) \,  \mu_{\mathrm{Haar}}(w),
\end{equation}
where $k(z,w) := \mathfrak{R} \Big( \frac{w+z}{w-z} \Big)$ is the {\it Poisson kernel}.  This formula in fact upgrades to
\numequation \label{Poisson extended}
\log{g(z)} = \int_{|w| = R} \log{|g(w)|} \cdot \frac{w+z}{w-z} \,  \mu_{\mathrm{Haar}}(w),
\end{equation}
because both sides are holomorphic in $z$, have identical real parts, and evaluate to zero at $z = 0$.

Differentiation in the integrand of (\ref{Poisson extended})  gives
a reproducing kernel for our logarithmic derivative as well: 
\numequation \label{reproducing log derivative}
\frac{g'(z)}{g(z)} = \int_{|w| = R} \frac{2w}{ ( w - z )^2 } \log{|g(w)|} \,\mu_{\mathrm{Haar}}(w), \qquad  z \in D(0,R).
\end{equation}
for the logarithmic derivative in the interior of the disc $|z| \leq R$ in terms of boundary values on the circle $|z| = R$.
We have the elementary calculation
\numequation \label{kernel calc}
\int_{|z| = r} |w-z|^{-2} \, \mu_{\mathrm{Haar}}(z) = \frac{1}{R^2-r^2} \quad \textrm{ for } |w| = R > r,
\end{equation}
and thus the $|z| = r$ integral of (\ref{reproducing log derivative}) with the triangle inequality and interchanging the orders
of the integrations and using $|\log{|g|}| = \log^+{|g|} + \log^-{|g|} = \log^+{|g|} + \log^+{|1/g|}$ yields
\begin{equation*}
\begin{split}
\int_{|z| = r} \Big| \frac{g'(z)}{g(z)} \Big| \, \mu_{\mathrm{Haar}}  
\leq 2R  \int_{|z| = r} \int_{|w| = R} |w-z|^{-2} \, \big| \log{|g(w)|} \big| \,\mu_{\mathrm{Haar}}(w) \,  \mu_{\mathrm{Haar}}(z)
\\ 
= 2R \int_{|w| = R} \Big( \int_{|z| = r} |w-z|^{-2} \, \mu_{\mathrm{Haar}}(z) \Big) \, \big| \log{|g(w)|} \big| \,\mu_{\mathrm{Haar}}(w)
\\
=\frac{2R}{R^2-r^2}   \int_{|w| = R}  \big| \log{|g(w)|} \big| \,\mu_{\mathrm{Haar}}(w)
\\
= \frac{2R}{R^2-r^2} \Big( m(R, g) + m(R,1/g) \Big) = \frac{4R \, m(R,g)}{R^2  -r^2}, 
\end{split}
\end{equation*}
on using on the final line the harmonicity property again which implies
$$
\int_{|w| = R} \log{|g|} \, \mu_{\mathrm{Haar}}(w) = \log{|g(0)|} = 0.
$$

The final piece of the proof borrows from~\cite[section~4]{BenbourenaneKorhonen}.
Let
$$
E := \Big\{ z \, : \, |z| = r, \, |g'(z)/g(z)| > 1 \Big\},
$$
a measurable subset of the circle $|z| = r$.
Since the function $\log^+{|x|}$ is concave on $x \in [1,\infty)$ where it coincides with $\log{|x|}$, Jensen's inequality gives
\begin{equation*}
\begin{split}
\int_{|z| = r} \log^+{\Big| \frac{g'}{g} \Big|} \, \mu_{\mathrm{Haar}}  \leq \mu_{\mathrm{Haar}}(E) \log^+ \Big(
\frac{1}{\mu_{\mathrm{Haar}}(E)} \int_{E} \Big| \frac{g'(z)}{g(z)} \Big| \, \mu_{\mathrm{Haar}}(z) \Big) \\
\leq  \mu_{\mathrm{Haar}}(E) \log^+ \Big(
\frac{1}{\mu_{\mathrm{Haar}}(E)} \int_{|z|=r} \Big| \frac{g'(z)}{g(z)} \Big| \, \mu_{\mathrm{Haar}}(z) \Big) \\
\leq  \mu_{\mathrm{Haar}}(E) \log^+ \Big(
 \int_{|z|=r} \Big| \frac{g'(z)}{g(z)} \Big| \, \mu_{\mathrm{Haar}}(z) \Big) +  \mu_{\mathrm{Haar}}(E)\log(1/\mu_{\mathrm{Haar}}(E))\\
\leq \log^+ \int_{|z| = r}  \Big| \frac{g'(z)}{g(z)} \Big| \, \mu_{\mathrm{Haar}}(z)   
+ \sup_{t \in (0,1]} \big\{ t\log{(1/t)} \big\} \\
\leq  \log^+{ \Big\{ \frac{4R \, m(R,g)}{R^2  -r^2} \Big\} } + \frac{1}{e} 
\leq     \log^+{\Big\{ \frac{m(R,g)}{r} \frac{R}{R-r} \Big\}} +  \log{2} + \frac{1}{e}, 
\end{split}
\end{equation*}
using $R^2 - r^2 = (R+r)(R-r) > 2r(R-r)$ on the final line.
\end{proof}

\begin{remark}\label{log-deri-history}
The case of arbitrary meromorphic functions $g : \overline{D(0,R)} \to \mathbf{P}^1$ is handled 
similarly by a differentiation in the  general Poisson--Jensen formula, but with rather more work to estimate
the finite sum over the zeros and poles of $g$. See for instance~\cite[\SSS~IX.3.1, page~244, (3.2)]{Nevanlinna} or~\cite[Lemma~2.3 on page~36]{Hayman} for similar bounds.
By using a technique due to Kolokolnikov for
handling the sum over the zeros and poles, Goldberg and Grinshtein~\cite{GoldbergGrinshtein} 
obtained the general bound
\numequation \label{log derivative lemma}
m\Big(r , \frac{g'}{g}\Big) <  \log^+{\Big\{ \frac{T(R,g)}{r} \frac{R}{R-r} \Big\}}
+ 5.8501,  \quad \textrm{for } g(0) = 1, 
\end{equation}
and proved that it is essentially best-possible in form apart for the value of the free numerical constant $5.8501$ (that has since been somewhat
further reduced in the literature, see Benbourenane--Korhonen~\cite{BenbourenaneKorhonen}).  The paper of  Hinkkanen~\cite{Hinkkanen} and the books of Cherry--Ye~\cite{CherryYe} and Ru~\cite{MinRu}
 discuss the implications to the structure of the error term in Nevanlinna second main theorem,
mirroring Osgood and Vojta's dictionary to Diophantine approximation and comparing to Lang's conjecture modeled on Khinchin's theorem.
\end{remark}

\subsection{Proof of Theorem~\ref{growth term}}   \label{proof resumed}
For $f : D(0,1) \to \C$ holomorphic, the polar divisor is empty, and so $N(r,f) = 0$ and $m(r,f) = T(r,f)$. Since by definition
$F_N^N - 1$ is a unit in the ring of holomorphic functions on $D(0,1)$,  
 our requisite bound (\ref{mean proximity bound}) rewrites in Nevanlinna  notation into
\numequation  \label{Nevanlinna rewriting}
T(r, p \circ F_N) \ll \log{\frac{N}{1-r}}, \quad \textrm{for each of } p(x) \in \big\{ x^N / (x^N - 1), \, 1/(x^N-1), \, x^N \big\}.
\end{equation}
\subsubsection{Equivalence of bounds for different $p(x)$}
By Lemma~\ref{counting is positive}, the fact $x^N / (x^N - 1)=1+1/(x^N-1)$, and~\eqref{trivial}, the three cases for $p(x)$ are equivalent to one another. Here we give the explicit estimate in one direction, which will be used later: 
\numequation
\begin{split}
T(r, p \circ F_N)  =  m\Big(r, 1 + \frac{1}{F_N^N-1}\Big)  
\geq m\Big(r, \frac{1}{F_N^N-1}\Big) -\log{2} \\ 
 =  T\Big(r, \frac{1}{F_N^N-1}\Big) -\log{2} = T(r, F_N^N-1) -\log{2} \\
\label{bootstrap}
\geq T(r, F_N^N) - 2\log{2} = N \, T(r, F_N) - \log{4},
\end{split}
\end{equation}
where we use $F_N(0) = 0$ and $F_N^N-1$ is a unit in the ring of holomorphic functions on $D(0,1)$.
In the rest of this subsection, we will prove Theorem~\ref{growth term} in the form $T(r, F_N^N) \ll \log{\frac{N}{1-r}}$ 
but pivoting around the choice
\numequation \label{choice of p}
p(x) := \frac{x^N}{x^N-1} = \frac{x}{N} \sum_{\zeta \in \mu_N} \frac{1}{x-\zeta}.
\end{equation}

\subsubsection{Reduction to a logarithmic derivative}

By either the chain rule or the partial fractions decomposition, we see that the logarithmic derivative $f'/f$ of the nowhere vanishing holomorphic function 
\numequation \label{intermediary} 
f := 1-F_N^N \quad  : \quad D(0,1) \to \C^{\times}
\end{equation}
 is related to $p \circ F_N =  F_N^N / (F_N^N-1)$ by
\numequation \label{log derivative happy accident}
 p \circ F_N =   \frac{F_N}{NF_N'} \frac{f'}{f}.
\end{equation}

The idea then is that the piece $f'/f$ in the decomposition~\eqref{log derivative happy accident} is small on average over circles by the lemma on the logarithmic derivative (Corollary~\ref{cor_f} below), while---again by the lemma on the logarithmic derivative, in Corollary~\ref{cor_F} below---the characteristic
functions of $p \circ F_N$ and $F_N/F_N'$ are equal respectively to $N T(r,F_N)$ and $T(r,F_N)$ up to a small error. 

\subsubsection{Two corollaries of the lemma on the logarithmic derivative}
\begin{cor}\label{cor_f}
For $f= 1-F_N^N$, we have 
\numequation\label{log derivative is the key}
m\Big( r, \frac{f'}{f} \Big) \ll  \sup_{|z| = (1+r)/2} \log^+{\log{|F_N|}} + \log{\frac{N}{1-r} }.
\end{equation}
\end{cor}
\begin{proof}
By applying Lemma~\ref{log derivative self-contained} to $f$ and the outer radius choice
$R := 1 - (1-r)/2 = (1+r)/2$,
and using (cf.~\cite[Corollary 13.2.14]{BombieriGubler}) that $m(r,f'/f) = T(r,f'/f)$ is a monotone increasing function of $r$, we find the  mean growth bound
\begin{equation*}
\begin{split}
m\Big( r, \frac{f'}{f} \Big) \ll \log^+{T\Big(\frac{1+r}{2}, f\Big)} + \log{\frac{e}{1-r} } \\
=  \log^+{m\Big(\frac{1+r}{2}, 1 - F_N^N \Big)} + \log{\frac{e}{1-r} } \\
\ll \log^+{m\Big(\frac{1+r}{2}, F_N^N \Big)} +  \log{\frac{e}{1-r} } \\
\ll   \log^+{m\Big(\frac{1+r}{2}, F_N\Big)} + \log{\frac{N}{1-r} }  \\
\ll  \sup_{|z| = (1+r)/2} \log^+{\log{|F_N|}} + \log{\frac{N}{1-r} },
\end{split}
\end{equation*}
where in the last step we have estimated a mean proximity function trivially 
by a supremum function. 
\end{proof}

\begin{cor}   \label{cor_F}
We have 
\numequation \label{second piece}
m\Big( r, \frac{F_N}{F_N'} \Big) \leq  T(r,F_N) + O \Big(   \log^+{\frac{N}{1-r} }  + \sup_{|z| = (1+r)/2}
\log^+{\log{|F_N|}} \Big).
\end{equation}
\end{cor}
The idea of the proof is to combine Lemma~\ref{log derivative self-contained} applied to the functional unit $1-F_N$ and the standard chain of implications based on Jensen's formula in the reduction of the second main theorem to the lemma on the logarithmic derivative (see, for example, \cite[pages 33--34]{Hayman}).
\begin{proof}
By~\eqref{Jensen} for the function $F_N'/F_N$, and the fact that $F_N$
is holomorphic on the disc $D(0,1)$ with $F_N(0)=0$ and $F_N'(0) \neq 0$, we have:
\numequation
\begin{split}
m\Big( r, \frac{F_N}{F_N'} \Big) = m\Big( r, \frac{F_N'}{F_N} \Big) + N\Big( r, \frac{F_N'}{F_N} \Big)
- N\Big( r, \frac{F_N}{F_N'} \Big) - \log{c(F_N'/F_N,0)} \\
= m\Big( r, \frac{F_N'}{F_N} \Big) + N\Big( r, 1/F_N \Big)
- N\Big( r, F_N \Big) -  N\Big( r, 1/F_N' \Big) + N\Big( r, F_N' \Big) \\
\label{Nram}
= m\Big( r, \frac{F_N'}{F_N} \Big) + N\Big( r, 1/F_N \Big) 
 -  N\Big( r, 1/F_N' \Big) 
= m\Big( r, \frac{F_N'}{F_N} \Big) + N\Big( r, 1/F_N \Big).
\end{split}
\end{equation}
Here for the last equality we recall that $F_N : D(0,1) \to \C \setminus \mu_N$ is an \'etale analytic mapping, 
hence the derivative $F_N'$ is nowhere vanishing. 

 We continue to estimate with the triangle inequality (for the second and third lines) and then \eqref{Jensen}, noting that $|F'_N(0)|>1$ (for the inequality in the fourth line):
\begin{equation*}
\begin{split}
 m\Big( r, \frac{F_N}{F_N'} \Big)  = m\Big( r, \frac{F_N'}{F_N} \Big) + N\Big( r, 1/F_N \Big) \\ \nonumber \leq 
m\Big( r, \frac{F_N'}{1-F_N} \Big) + m\Big( r, \frac{1-F_N}{F_N} \Big) + N\Big( r, 1/F_N \Big)\\
\leq m\Big( r, \frac{F_N'}{1-F_N} \Big) + \log 2 + m\Big( r, \frac{1}{F_N} \Big) + N\Big( r, 1/F_N \Big)\\
=m\Big( r, \frac{(1-F_N)'}{1-F_N} \Big) + T(r, 1/F_N) + \log 2
\leq m\Big( r, \frac{(1-F_N)'}{1-F_N} \Big) + T(r, F_N) + \log 2 \\
\leq  T(r,F_N) + O \Big(   \log^+{\frac{N}{1-r} }  + \sup_{|z| = (1+r)/2}
\log^+{\log{|F_N|}} \Big),
\end{split}
\end{equation*}
upon again using Lemma~\ref{log derivative self-contained} with $R := (1-r)/2$ but now for the  functional unit
$g = 1-F_N$, and a similar argument as in the proof of Corollary~\ref{cor_f}.
\end{proof}

\subsubsection{Completing the proof from the crude supremum bound  in Lemma~\ref{trivialsupFN}}
At this point the key identity~(\ref{log derivative happy accident}) allows us to combine the estimates~(\ref{log derivative is the key})
and~(\ref{second piece}), arriving  at the uniform bound
\numequation
\begin{split}
T(r, p \circ F_N) = m(r, p \circ F_N) \leq m\Big(r, \frac{f'}{f} \Big) 
+ m\Big(r, \frac{F_N}{F_N'}\Big) \\
\label{second main bound}
\leq T(r,F_N) + O \Big(   \log^+{\frac{N}{1-r} }  + \sup_{|z| = (1+r)/2}
\log^+{\log{|F_N|}}  \Big).
\end{split}
\end{equation}

We leverage the upper bound (\ref{second main bound}) on $T(r, p \circ F_N) = N \, T(r, F_N) + O(1)$
against the lower bound~(\ref{bootstrap}) and get a uniform upper bound on $T(r,F_N)$: 
\numequation \label{our second main}
(N-1) T(r, F_N)  \ll  \log^+{\frac{N}{1-r} }  + \sup_{|z| = (1+r)/2}
\log^+{\log{|F_N|}}.
\end{equation} 
Upon doubling the absolute implicit constant, plainly for $N \geq 2$ this is equivalent to
\begin{equation*}
T(r, F_N^N) =  N T(r, F_N)  \ll \log^+{\frac{N}{1-r} }  + \sup_{|z| = (1+r)/2}  
\log^+{\log{|F_N|}}, 
\end{equation*}
uniformly in all $N \geq 2$ and $r \in (0,1)$.

\medskip

Hence Theorem~\ref{growth term} follows from Lemma~\ref{trivialsupFN} upon replacing $r$ there with $(1+r)/2$.

\subsubsection{A historical note}   \label{some Nevanlinna background}
The bound~(\ref{our second main}) can be compared to the well-known
particular case for entire holomorphic functions of the classical Nevanlinna second main theorem (whose method of proof
we emulate here), stating that for any entire function $g : \C \to \C$, and any $N$-tuple of pairwise distinct points 
$a_1, \ldots, a_N \in \C$, the Nevanlinna characteristic $T(r,g) = m(r,g) = \int_{|z| = r} \log^+{|g|} \, \mu_{\mathrm{Haar}}$  satisfies the upper bound
\numequation \label{Nevanlinna's Picard}
(N-1) T(r, g)  + N_{\mathrm{ram}}(r, g)   \leq \sum_{i=1}^N N(r, a_i) + O(\log{T(r,g)}) + O( \log{r} )
\end{equation}
outside of an exceptional set of radii $r \in E \subset [0,\infty)$ of finite Lebesgue measure: $m(E) < \infty$. 
Here $N_{\mathrm{ram}}(r, g) = N(r, 1/g')$ is a ramification term, which is always nonnegative and vanishes if the map $g$ is \'etale. This is Nevanlinna's quantitative strengthening of Picard's theorem on at most one omitted value for a nonconstant entire function, for if each of $a_1, \ldots, a_N$ is omitted then all counting terms $N(r, a_i) = 0$ vanish on the right-hand side of~(\ref{Nevanlinna's Picard}),
leading if $N \geq 2$ to an $O(\log{r})$ upper bound on the growth $T(r,g)$ of $g$. 
The idea is that we similarly have a holomorphic map $F_N$ omitting the $N$ values $a_h = \exp(2\pi i h/N)$, except $F_N$ is on a disc rather than the entire plane, and that (\ref{Nevanlinna's Picard}) largely extends as a growth bound for holomorphic maps on a disc. For such completely quantitative results we refer the reader to Hinkkanen~\cite[Theorem 3]{Hinkkanen} or Cherry--Ye~\cite[Theorem~4.2.1 or Theorem~2.8.6]{CherryYe}. We cannot directly apply these general theorems in their verbatim forms as they only lead to a bound of the form $m(r, F_N) \ll  \frac{1}{N}\log{\frac{1}{1-r}} + \log{N}$ in place of the required $m(r,F_N) \ll \frac{1}{N} \log{\frac{N}{1-r}}$; cf. the term $(q+1)\log(q/\delta)$ in~\cite[line~(1.24)]{Hinkkanen}, where $q = N$ signifies the number of targets $a_i$. But fortuitously we were able to modify their proofs by making an additional use of the key pivot relation~(\ref{choice of p}) particular to our situation of $\{a_1, \ldots, a_q\} = \mu_N$. 

For our case of functions on the disc, we compare to ~\cite[Theorem~2.1]{Hayman}. For holomorphic $f$, we again have the \emph{ramification term} $N_{\mathrm{ram}}(r, f) = 
N(r, 1/f')$ (this term is denoted by $N_1(r)$ in~\cite{Hayman}), which is always nonnegative. In~\eqref{Nram}, even without using the \'etaleness of $F_N$, one would drop the ramification term by positivity and still obtain the requisite bound $\displaystyle m\Big( r, \frac{\varphi}{\varphi'} \Big) \leq m\Big( r, \frac{\varphi'}{\varphi} \Big) + N\Big( r, 1/\varphi \Big)$. In this way, our treatment also recovers the bound $\int_{|z| = r} \log^+{|\varphi|} \, \mu_{\mathrm{Haar}} \leq  \frac{1}{N-1}\log{\frac{1}{1-r}} + O_{\varphi}(1)$ for every holomorphic map $\varphi : D(0,1) \to \C \setminus \mu_N$  avoiding the $N$-th roots of unity (which is not necessarily the universal covering map).

\subsection{Proof of Theorem~\ref{theorem:main}}  \label{putting together}
At this point we have established all the pieces for the proof of our main result. By Theorem~\ref{exactconformal}, 
assumption~(\ref{radius}) in Proposition~\ref{strategy} is indeed satisfied, with the sharp constant $A := \zeta(3)/2 > 0$. By Theorem~\ref{growth term} with the choices $p(x) := x^N$ and $r := 1 - BN^{-3}$, assumption~(\ref{mean}) in Proposition~\ref{strategy} is also satisfied. In terms of the algebras of modular forms $M_{2N}$ and $R_{2N}$ at an even Wohlfahrt level~$2N$ introduced in~\ref{rings}, the conclusion of Proposition~\ref{strategy} is thus an inequality $[R_{2N}:M_2] \leq C N^3 \log{N}$, for some absolute implicit constant $C \in \R$ independent of $N$. At this point Proposition~\ref{selfcontained} proves the equality $R_{2N} = M_{2N}$ for all $N \in \NwithoutzeroA$, which is the unbounded denominators conjecture. 

\medskip

The proof of Theorem~\ref{theorem:main} is thus completed.     \qed

\begin{remark}\label{remark:numberfield}
Our proof for Theorem~\ref{theorem:main}  generalizes in the obvious way to establish that a modular form $f(\tau)$ having a Fourier expansion in $\overline{\Z} \llbracket q^{1/N}\rrbracket$ (algebraic integer Fourier coefficients) at one cusp, and meromorphic at all cusps, is a modular form for a congruence subgroup of $\SL_2(\Z)$. We include an indication of the details.

Since $f(\tau)$ is a modular form, we are reduced to the situation of a number field $K$ such that $f(\tau) \in O_K\llbracket q^{1/N}\rrbracket$. We use $R_{2N}$ to denote the~$K(\lambda)$-algebra generated by modular functions with coefficients in $K$, bounded denominators at~$\zeta = i \infty$, and cusp widths dividing $2N$ at all cusps $\zeta \in \mathbf{P}^1(\Q)$ (similar to Definition~\ref{rings}). 
We follow the proof of Proposition~\ref{strategy} now on the case of the $K(\lambda)$-vector space
 $\mathcal{V}(U,x(t),O_K)$ from Definition~\ref{holonomy ring}. %the ring of formal power series $f(x)\in K\llbracket x \rrbracket$ such that $f(x(t))\in O_K\llbracket t \rrbracket$ and $L(f)=0$ for some nonzero linear differential operator $L$ over $\C(x)$ without any singularities on $U = \C \setminus 16^{-1/N} \mu_N$. 
 Then $R_{2N}  \subset \mathcal{V}(U,x(t),O_K)$.
Note that $U$ is stable under the action of $\Gal(\overline{\Q}/\Q)$, and thus $\mathcal{V}(U,x(t),O_K)= \mathcal{V}(U,x(t),\Z)\otimes_{\Q} K$ and $\dim_{K(\lambda)}\mathcal{V}(U,x(t),O_K)=\dim_{\Q(\lambda)}\mathcal{V}(U,x(t),\Z)$. Thus by Corollary~\ref{holonomy}, Theorem~\ref{exactconformal}, and Theorem~\ref{growth term}, we still have that $R_{2N}$ has dimension at most $CN^3\log N$ over $K(\lambda)$.
%[\lambda^\pm, (1-\lambda)^\pm]$. 
The claimed extension to $\overline{\Z}\llbracket q^{1/N} \rrbracket$ Fourier expansions now follows upon remarking that the proof of Proposition~\ref{selfcontained} still persists when $\Q$  is replaced by $K$.
\end{remark}

\begin{remark} \label{Zbar reduction}
It is also possible to derive the $\bar{\Z}\llbracket q^{1/N} \rrbracket$ generalization directly from Theorem~\ref{theorem:main}, by the following argument pointed out to us
by John Voight. The absolute Galois group $\mathrm{Gal}(\bar{\Q}/\Q)$ acts on the $q$-expansions of modular forms. If $f(\tau) \in O_K\llbracket q^{1/N} \rrbracket$ is a modular form on a finite index subgroup of $\SL_2(\Z)$, and $\alpha_1, \ldots, \alpha_d$ is a $\Z$-basis of $O_K$, then $f_i(\tau) := \mathrm{Tr}_{K/\Q}( \alpha_i  f(\tau) ) \in \Z\llbracket q^{1/N} \rrbracket$ for each $i = 1,\ldots, d$. Theorem~\ref{theorem:main} gives that each $f_i(\tau)$ is modular for some congruence subgroup, say $\Gamma_i$. At this point $f(\tau)$, being a $K$-linear combination of $f_1,\ldots,f_d$, is modular for the congruence subgroup $\Gamma_1 \cap \cdots \cap \Gamma_d$.

\end{remark}

\section{Generalization to vector-valued modular forms} \label{vvmf}

\subsection{Generalized McKay--Thompson series with roots from monstrous moonshine}  \label{cft}
Our argument also proves a vector generalization of the unbounded denominators conjecture, which has been conjectured by Mason~\cite{MasonConjectured} (see also the earlier work of Kohnen and Mason~\cite{KohnenMason} for a special case) to the setting of vector-valued modular forms of $\SL_2(\Z)$, with motivation from the theory of vertex operator algebras and the monstrous moonshine conjectures. The weaker statement  of algebraicity over the ring of modular forms was conjectured earlier by Anderson and Moore~\cite{AndersonMoore}, within the context of the partition functions or McKay--Thompson series attached to rational conformal field theories. We refer also to Andr\'e~\cite[Appendix]{Andre} for a discussion from the arithmetic algebraization point of view --- the method that we build upon in our present paper --- on the Grothendieck--Katz $p$-curvature conjecture. Eventually the more precise
expectation crystallized (see Eholzer~\cite[\emph{Conjecture} on page~628]{Eholzer})  that all rational
conformal field theory graded twisted characters are in fact classical modular forms for a  \emph{congruence} subgroup of $\SL_2(\Z)$, which is more precise than Anderson and Moore's conjectured algebraicity over the modular ring $\Z[E_4, E_6]$.

 This conjecture became known as the \emph{congruence property in conformal field theory}, and was proved  in the eponymous paper of Dong, Lin and Ng~\cite{DongLinNg}, after  landmark progresses  from many authors (for some history, including notably Bantay's solution~\cite{Bantay} under a certain heuristic assumption, the \emph{orbifold covariance principle}~\cite{BantayCov, BantayPerm, Xu},  we refer the reader to the introduction of~\cite{DongLinNg}). Finally, the congruence property for the McKay--Thompson series in the full equivariant setting (orbifold theory) $V^G$ of a finite group $G$ of automorphisms of a rational, $C_2$-cofinite vertex operator algebra $V$ (the prime example being  the Fischer--Griess Monster group operating on the moonshine module of Frenkel--Lepowski--Meurman~\cite{FLM}) was
 proved  by Dong and Ren~\cite{DongRen} by a reduction to the special case $G = \{1\}$ that is~\cite{DongLinNg}. 

Our paper, via Theorem~\ref{ubd for vvmf} below for the vector valued extension of the congruence property, inherits a new proof of these modularity theorems. The connection was engineered by Knopp and Mason~\cite{KnoppMason}, with their formalization  of generalized modular forms for $\SL_2(\Z)$, and fine tuned by Kohnen and Mason~\cite[\SSS~4]{KohnenMason}, who
brought forward the idea of a purely arithmetic approach --- based on the integrality properties of the Fourier coefficients, that record a graded dimension  and are hence integers --- for a part of Borcherds' theorem~\cite{Borcherds} (the Conway--Norton ``monstrous moonshine" conjecture).  Namely, suppressing the Hauptmodul property, for the  classical modularity  --- under a congruence subgroup of $\SL_2(\Z)$ --- of all the various McKay--Thompson series for the Monster group over the moonshine module $V^{\sharp}$.  Whereas Borcherds' proof, based on his own generalized Kac--Moody algebras that go outside of the general framework of vertex operator algebras, is rather particular to the Monster vertex algebra and genus $0$ arithmetic groups, Kohnen and Mason proposed that an arithmetic abstraction from the integrality of Fourier coefficients might open up a window on the modularity and congruence properties to apply just as well in the equivariant setting to any rational $C_2$-cofinite vertex operator algebra --- this theorem, eventually proved in~\cite{DongLinNg, DongRen} by other means, was an  open problem at the time of~\cite{KohnenMason}. 

It is precisely this arithmetic scheme that we are able to complete with our paper.

\subsection{Unbounded denominators for the solutions of certain ODEs}

In the language of Anderson--Moore~\cite[page~445]{AndersonMoore}, the functions   occurring below are said to be \emph{quasi-automorphic} for the modular group $\PSL_2(\Z)$, while in Knopp--Mason~\cite{KnoppMasonFourier} or Gannon~\cite{GannonVVMF}, they arise  as component functions   of vector-valued modular forms for $\SL_2(\Z)$. 
We firstly take up the holonomic viewpoint and give a yet another formulation, in the equivalent
language of linear ODEs on the triply punctured projective line, where we think of $x$ as the modular function $\lambda(\tau) / 16 \in q + q^2\Z\llbracket q\rrbracket$, where $q = \exp(\pi i \tau)$,  and of $\mathbb{P}^1 \setminus \{0,1/16,\infty\}$ as the 
modular curve $Y(2) = \H / \Gamma(2)$.  This answers the question raised in~\cite[Appendix, A.5]{Andre}. For simplicity of exposition, we only consider the case of a power series expansion $f(x) \in \Z\llbracket x \rrbracket$ here, as opposed to a general Puiseux expansion (see Remark~\ref{Puiseux extension}).

\begin{theorem} \label{ODE form}
Let $f(x) \in \Z\llbracket  x \rrbracket$ be an integer coefficients formal power series solution of $L(f) = 0$, where $L$ is a linear differential operator 
without singularities\footnote{Similarly to~\ref{holonomy ring}, the proof allows for singularities  on $\mathbb{P}^1 \setminus \{0, 1/16,\infty\}$ provided their local monodromy is trivial.}  on $\mathbb{P}^1 \setminus \{0,1/16,\infty\}$. If the $x = 0$ local monodromy of $L$ is finite, then $f(x)$ 
is algebraic, and more precisely, the function $f(\lambda(\tau)/16)$ on $\H$ is automorphic for some congruence subgroup $\Gamma(N)$ of $\SL_2(\Z)$.
\end{theorem}

\begin{proof}
Our condition is that the $x=0$ local monodromy group is $\Z / N$ for some $N \in \NwithoutzeroA$.
Then the formal function $g(x) := f(x^N)$ is in $\Z\llbracket x \rrbracket$ and fulfills a linear ODE on $\mathbb{P}^1 \setminus \{16^{-1/N}\mu_N, \infty\}$. In our notation of Corollary~\ref{holonomy}, that means $g \in \mathcal{H}(\C \setminus 16^{-1/N}\mu_N , \Z)$. Hence, denoting again by $F_N : D(0,1) \to \C \setminus \mu_N$ the universal covering map taking $F_N(0) = 0$, recalling  our exact uniformization radius formula in Theorem~\ref{exactconformal} giving in particular the strict lower bound
$$
|F_N'(0)| = \sqrt[N]{16} \, \Big(  1 + \frac{\zeta(3)}{2N^3}  + \frac{3\zeta(5)}{8N^5} + \cdots  \Big) > \sqrt[N]{16},
$$
 and letting then
$$
\varphi(z) := 16^{-1/N} F_N(rz)
$$
 for some parameter $r$ with $ \sqrt[N]{16} \big/ |F_N'(0)|  < r < 1$, Corollary~\ref{holonomy} implies that $g(x) \in \Z\llbracket x \rrbracket$ is an algebraic power series. Hence $f(x) = g(\sqrt[N]{x})$ is algebraic. 

 At this point we know that $f(\lambda(\tau)/16)$ is automorphic for some finite index subgroup $\Gamma \subset \Gamma(2)$. Theorem~\ref{theorem:main} then upgrades this to automorphy under some congruence modular group $\Gamma(M)$, for some $M \equiv 0 \mod{N}$, and the result follows upon replacing $N$ with $M$.
\end{proof}

\begin{remark} \label{Puiseux extension}
To include Puiseux series $f(x) \in \C\llbracket x^{1/m} \rrbracket$, the statement and proof apply verbatim on replacing the integrality condition $f(x) \in \Z\llbracket x \rrbracket$ by $f(\lambda(\tau)/16) \in \Z\llbracket \lambda(\tau/m) / 16 \rrbracket \otimes \C$.
\end{remark}

\begin{remark} \label{log remark}
The condition in Theorem~\ref{ODE form} that the linear differential operator $L$ has a finite local monodromy at $x = 0$ is essential for algebraicity. The canonical and explicit transcendental example, which is given in~\cite[Appendix, A.5]{Andre} and we have already mentioned in our introduction~\SSSS~\ref{intro sketch}, is the Gauss hypergeometric series or complete elliptic integral of the first kind
$$
\frac{2}{\pi} K(x) := \pFq{2}{1}{1/2,1/2}{1}{16x} = \sum_{n=0}^{\infty} \binom{2n}{n}^2 x^n, 
$$
that is the Hadamard square of $(1-4x)^{-1/2}$ and has the Jacobi theta function parametrization making
\numequation \label{theta series}
 \pFq{2}{1}{1/2,1/2}{1}{\lambda(q)} = \Big( \sum_{n \in \Z} q^{n^2} \Big)^2
\end{equation}
a weight one modular form for the congruence group $\Gamma(2)$. The modularity streak is not an accident: 
more generally, to get $\Z\llbracket x \rrbracket$ holonomic functions on $\mathbb{P}^1 \setminus \{0, 1/16, \infty\}$ with infinite $x = 0$ local monodromy, we may reversely start with any congruence modular form of a weight $k > 0$, such
as  for instance Ramanujan's (discriminant) weight~$12$ modular form $\Delta (\tau) = q\prod_{n=1}^{\infty}(1-q^n)^{24} \in q \Z\llbracket q\rrbracket$, and express it formally into a power series in $x:= \lambda(\tau) /16$, using $\Z\llbracket q\rrbracket = \Z\llbracket x \rrbracket$ as in~\SSSS~\ref{intro sketch}. It is then a classical fact, cf. Stiller~\cite{Stiller} or Zagier~\cite[\SSS~5.4]{Zagier123}, that  the resulting formal power series fulfills a linear ODE on a finite  \'etale cover of $\mathbb{P}^1 \setminus \{0,1/16,\infty\} \cong Y(2)$,  of order $k+1$ and monodromy group commensurable with $\mathrm{Sym}^k \, \SL_2(\Z) \hookrightarrow \SL_{k+1}(\Z)$.

It remains to us an open question whether a complete description of all integral solutions $f \in \Z\llbracket x \rrbracket$ on dropping the $x = 0$ finite local monodromy condition in Theorem~\ref{ODE form} should arise in this way from a classical congruence modular form expressed into a holonomic function in $x = \lambda/16$. We formulate the precise statement in Question~\ref{logarithmic UDC} below.
\end{remark}

\subsection{Vector-valued modular forms}  \label{Mason vvmf}

We close our paper by  another formulation of Theorem~\ref{ODE form}, translated now over to the language of vector-valued modular forms.  The following definition is a special case of the vector-valued modular forms studied in~\cite[\SSS~2]{FrancMasonVvmf}.

\begin{df} \label{df vvmf}
A \emph{vector-valued modular form} of weight $k \in \Z$ and dimension $n$ for $\PSL_2(\Z)$ is  a pair $(F,\rho)$ made of a holomorphic mapping $F = (F_1, \ldots, F_n) : \H \to \C^n$ and an $n$-dimensional
complex representation
$$
\rho : \PSL_2(\Z) \to \GL_n(\C)
$$ 
obeying the following properties:
\begin{itemize}
\item  For all $\gamma \in \PSL_2(\Z)$, we have 
$
F^{\mathrm{t}} \,|_k \gamma = \rho(\gamma) F^\mathrm{t}
$.
\item The matrix
$\rho \left( \begin{pmatrix}  1 & 1 \\ 0 & 1 \end{pmatrix}  \right)  \in \GL_n(\C)$
is semisimple.
\item All components $F_j : \H \to \C$ have \emph{moderate growth in vertical strips}: for all $a < b$ and $C > 0$, 
there exist $A,B > 0$ such that
$$
 \tau \in \H, \quad a \leq \mathrm{Re} \, \tau \leq b, \quad  \mathrm{Im}\, \tau \geq C
\quad \Longrightarrow \quad |F_j(\tau)| \leq A e^{B \, \mathrm{Im} \, \tau}. 
$$
\end{itemize}
Here, as usual, $|_k$ is used to denote the componentwise right action of $\gamma = \begin{pmatrix} a & b \\ c & d \end{pmatrix}$ 
via the usual automorphy factor $j_k(\gamma, \tau) = (c\tau + d)^{-k}$:
$$
f(\tau) \, |_k \gamma := j_k(\gamma, \tau)  f(\gamma\tau) = (c\tau+d)^{-k} f(\gamma \tau).
$$
\end{df}

\begin{remark} \label{expansion form of the regular singularities condition}
Taken together (see, for example,~\cite[\SSS~2.A]{AndersonMoore}) the semisimplicity and moderate growth  conditions are equivalent to the existence of generalized Puiseux formal expansions (except in general with irrational exponents: but without $\log{q}$ terms, due to semisimplicity) of each component function $F_j(\tau)$ at the cusp $q = 0$.  More precisely, via a change of basis (see the equivalent notion in~\cite{FrancMasonVvmf}), we may assume that $\rho  \left( \begin{pmatrix}  1 & 1 \\ 0 & 1 \end{pmatrix} \right) $ is a diagonal matrix. If $F_j$ is a $\lambda$-eigenvector of $\rho \left( \begin{pmatrix}  1 & 1 \\ 0 & 1 \end{pmatrix} \right)$, then $F_j=\sum_{n\in \Z_{\geq n_0}} a_{n,j}q^{n+\mu}$ for some $n_0\in \Z$, where $q=e^{2 \pi i \tau}$ and we choose a $\mu\in \C$ such that $\lambda=e^{2 \pi i \mu}$.
\end{remark}

Thus, the classical (scalar-valued) modular forms  $M_k(\Gamma(1), \chi)$ attached to a finite-order character $\chi : \Gamma(1) \to U(1)$ are precisely the special case $n = 1$ of one-dimensional vector-valued modular forms and a unitary character $\rho$.  In a reverse direction, any classical (scalar-valued) modular form for a finite index subgroup $\Gamma \subseteq \PSL_2(\Z)$ can be considered as the first component of a vector-valued modular form for $\PSL_2(\Z)$ of dimension $[\Gamma(1):\Gamma]$. From that point of view, there is no loss of generality in Definition~\ref{df vvmf} to limit to the representations of the ambient group $\PSL_2(\Z)$. Here and in the following, we continue to denote by~$\rho$ the extended homomorphism~$\rho : \Gamma(1) = \SL_2(\Z) \to \PSL_2(\Z) \to \GL_n(\C)$, which (by convention) contains~$Z(\SL_2(\Z)) = \{\pm I\}$ in its kernel. 

Knopp and Mason's \emph{generalized modular forms}~\cite{KnoppMason} are the case, intermediate in generality, where the representation $\rho$ is \emph{monomial}: that is, induced from a linear character $\chi : \Gamma \to \C^{\times}$ on a finite index subgroup  $\Gamma \subset \PSL_2(\Z)$. 
If that character $\chi$ is unitary, then in fact it has finite image and all components of $F$ are classical modular forms of weight $k$ for a finite index subgroup~\cite{KnoppMason}. The general (non-unitary) case does come up for the partition function and correlation functions of a rational conformal field theory~\cite{KnoppMason}, to which the point of contact is supplied by  Zhu's modularity theorem~\cite{Zhu}, 
%(see also Codogni~\cite{Codogni} for a recent different proof and a generalization),
 and its extension to the equivariant setting by Dong, Li and Mason~\cite{DongLiMason}.

To make the connection to Theorem~\ref{ODE form}, note upon restricting the representation $\rho$ to the free subgroup 
$$\Z \ast \Z 
= \left< 
  \begin{pmatrix}  1 & 2 \\ 0 & 1 \end{pmatrix},  \begin{pmatrix}  1 & 0 \\ 2 & 1 \end{pmatrix}  \right>
\subset \Gamma(2) \subset \Gamma(1) =\SL_2(\Z)$$
 that
the case of weight $k = 0$ and finite-order element $\rho \left( \begin{pmatrix}  1 & 1 \\ 0 & 1 \end{pmatrix} \right)$  is equivalent to exactly the situation of~\ref{ODE form}: a local system on the triply-punctured projective line $Y(2) \cong \mathbb{P}^1 \setminus \{0, 1/16, \infty\}$ \emph{that has a finite local monodromy  around the puncture $x = 0$}.
(Note that the map from~$\Z \ast \Z$ to~$\Gamma(2)/\pm I \subset \PSL_2(\Z)$ is an isomorphism.) Concretely, the local system with integrable connection~$(\mathcal{E}, \nabla)$ over~$Y(2) \cong \mathbb{P}^1 \setminus \{0,1/16, \infty\}$ is defined by taking for~$\nabla$ the derivation~$d/dx$ in the coordinate~$x := \lambda/16$ of the base curve~$Y(2)$, and the vector bundle~$\mathcal{E} \to Y(2)$  over the base algebraic curve~$Y(2) = \spec \Z[ x, 16/x, 1/(1-16x) ]$ to be defined by the  rank-$6n$ free $\Z[x,1/x,1/(1-16x)]$-module spanned by the functions~$F_j | \gamma$, where~$F_1, \ldots,F_n$ range over the components of the vector-valued modular form~$F$ on the modular group~$\Gamma(1) = \SL_2(\Z)$, and~$\gamma$ runs through all six cosets for~$\Gamma(2)$ in~$\Gamma(1)$ (with the stroke action here being in weight~$0$). The multiplier system~$\rho$ features as the monodromy representation~$\rho|_{\Gamma(2)} : \Gamma(2) \to \GL_n(\C)$. The condition in Remark~\ref{expansion form of the regular singularities condition} on the existence of a Fourier expansion at the cusp of~$\PSL_2(\Z)$ means that the local system with integrable connection~$(\mathcal{E},\nabla)$ has regular singularities with semisimple local monodromies around the three cusps of~$Y(2)$. 
 We refer the reader to~\cite{BantayGannon} and~\cite{GannonVVMF} for the bridge between these two equivalent points of view. 

Our general result on unbounded denominators for components of vector-valued modular forms is the following. 

\begin{thm}  \label{ubd for vvmf}
Let $(F,\rho)$ be a vector-valued modular form for $\PSL_2(\Z)$ of dimension~$n$ and weight~$k$. Suppose that some component function $F_j(\tau) : \H \to \C$ of $F = (F_1, \ldots, F_n) : \H \to \C^n$ has at $\tau = i \infty$
a formal Fourier expansion lying in $\Z\llbracket q \rrbracket = \Z\llbracket e^{\pi i \tau} \rrbracket$. Then that component $F_j(\tau)$ is a classical modular form of weight $k$ on a congruence subgroup of $\SL_2(\Z)$.
\end{thm}

\begin{proof}
After some standard theorems from the theory of $G$-functions to reduce to the case that
the semisimple matrix $\rho \left( \begin{pmatrix} 1 & 1 \\ 0 & 1 \end{pmatrix}  \right) \in \GL_n(\C)$ is in fact
of finite order, this is an equivalent expression of Theorem~\ref{ODE form}. 
The transition is as follows. First, by taking the componentwise product 
$$F(\tau)g(\tau) \left( \frac{\lambda(\tau)} {16\Delta(\tau/2)} \right)^{\frac{k+k'}{12}},$$
 where we choose a non-zero scalar-valued modular form $g(\tau)\in \Z \llbracket q \rrbracket$ of weight $k'$ to have $12 \mid k+k'$,
 we reduce to the case of a vector-valued modular form of weight~$k = 0$ on~$\PSL_2(\Z)$. 
 
We restrict that form from~$\PSL_2(\Z)$ to its index-$6$ torsion-free subgroup~$\Gamma(2) / \langle \pm I \rangle$. The remarks immediately preceding the statement of the theorem-under-proof construct a rank~$6n$ local system with integrable connection~$(\mathcal{E},\nabla)$ over the modular curve~$Y(2) = \spec{\Z[x,16/x,1/(1-16x)}]$. The monodromy representation of that local system is the homomorphism~$\rho|_{\Gamma(2)} : \Gamma(2) \to \GL_n(\C)$. As we study the specific component function~$F_j$ and its monodromy unfoldings in this local system, upon replacing the range~$\GL_n(\C)$ by a lower-dimensional general linear group, we lose no generality in assuming that the local system is irreducible. 
 
 With these reductions, we have realized our original~$\PSL_2(\Z)$ vector-valued modular form component~$F_j \in \Z\llbracket q \rrbracket = \Z \llbracket \lambda/16 \rrbracket$ of interest as one of the~$6n$ component power series in a complex Puiseux series vector solution to an irreducible rank-$6n$ system of first-order linear homogeneous ODEs over $\Q[\lambda, 1/\lambda, 1/(1-\lambda)]$. One of the components --- namely, $F_j$ --- in the solution vector to this irreducible linear differential system is a~$G$-function~\cite[page~xiii]{DworkGerottoSullivan}. 
David and Gregory Chudnovsky's fundamental~$G$-functions theorem~\cite[Theorem~VIII.1.5]{DworkGerottoSullivan} implies that this linear differential system~$(\mathcal{E},\nabla)$ satisfies the Galo\v{c}kin (finite  global operator height $\sigma(\nabla) < \infty$) condition~\cite[VII.2.(2.3) on page~227]{DworkGerottoSullivan}, hence by the Bombieri--Andr\'e theorem~\cite[Theorem~VII.2.1]{DworkGerottoSullivan}, it satisfies the
Bombieri (finite generic global inverse radius $\rho(\nabla) < \infty$) condition~\cite[VII.2.(2.1) on page~226]{DworkGerottoSullivan}, and is therefore globally nilpotent.  At this point Katz's local monodromy
theorem~\cite{Katz} (see also~\cite[Theorem~III 2.3 (ii)]{DworkGerottoSullivan}) proves that $(\mathcal{E},\nabla)$ has 
quasi-unipotent local monodromies. Since, as we already observed from Remark~\ref{expansion form of the regular singularities condition}, our local system~$(\mathcal{E},\nabla)$ also has semisimple local monodromies, 
it follows that these local monodromies have finite order. 

Thus we find that~$f := F_j \in \Z\llbracket q \rrbracket = \Z \llbracket x \rrbracket$ (in the coordinate~$x := \lambda/16$) satisfies a linear ODE~$L(f) = 0$ over~$Y(2) = \mathbb{P}^1 \setminus \{0,1/16,\infty\}$ with a finite local monodromy at $x = 0$. The result now follows on 
applying Theorem~\ref{ODE form} to $f(x) = F_j(\tau)$. 
\end{proof}

\begin{cor}[Mason's conjecture] \label{Mason conjecture}
If all components of a vector-valued modular form $(F,\rho)$ for $\PSL_2(\Z)$ have  Fourier expansions with bounded denominators, then the representation $\rho$ has a finite image, and more precisely $\ker(\rho) \supseteq \Gamma(N)$ for some $N \in \NwithoutzeroA$.
\end{cor}

\subsection{Some questions and concluding remarks}  \label{questions and remarks}

\subsubsection{A brief survey of the literature on Mason's conjecture}
 Mason's conjecture, as discussed in~\cite{MasonConjectured, KohnenMason, KohnenMason2},  concerned the stronger condition in Corollary~\ref{Mason conjecture}, namely that \emph{all} components $F_1, \ldots, F_n$ have bounded denominators.
These are the cases emerging in conformal field theories, and apart from Gottesman's result~\cite[Theorem~1.7]{Gottesman} resolving a strong form of the conjecture for a  class of two-dimensional vector-valued modular forms on $\Gamma_0(2)$,
  the literature  on the vector-valued case has focused on the stronger assumption for the full vector of components $F$. We review some of this work here. 

Originally Kohnen and Mason~\cite{KohnenMason,KohnenMason2} focused on the particular case (\emph{generalized modular forms})  that the representation $\rho$ is monomial. They used the Rankin--Selberg method to prove the conjecture in the case of a generalized modular function (weight~$0$) without any zeros or poles on the extended upper-half plane~\cite[Theorem~1]{KohnenMason}. In fact Selberg's paper~\cite{Selberg} that they used here had already considered vector-valued modular forms for the purpose of extending the Rankin--Selberg estimate into the noncongruence case (see also~\S~\ref{Ramanujan query} below). Kohnen and Mason~\cite[Theorem~2]{KohnenMason}, again based on the Rankin--Selberg $L$-function method but now with a finer input from the Eichler--Shimura--Weil bound on Fourier coefficients of congruence cusp forms in weight~$2$, also proved that 
when $\rho$ is induced from a linear character of a \emph{congruence} subgroup of $\PSL_2(\Z)$, the same result on generalized modular function units also holds if the condition on integer coefficients is relaxed to $S$-integer coefficients: a case that goes beneath the scope of our results here.  

In a sequel work~\cite{KohnenMason2}, Kohnen and Mason  used the Knopp--Mason canonical factorization~\cite{KnoppMasonFactorization} $f = f_0f_1$ (over $\C$) of a parabolic generalized modular function $f$ on a congruence subgroup of $\PSL_2(\Z)$, where $f_0$ is a parabolic generalized modular function of a \emph{unitary} character $\chi$, while $f_1$ is a parabolic generalized modular function without zeros or poles  on the extended upper-half plane~\cite{KnoppMason}. Combining to  their earlier method from~\cite{KohnenMason}, they thus proved that the unbounded denominators conjecture for the case of \emph{parabolic} GMF is equivalent to the algebraicity of the first ``few'' Fourier coefficients of the component $f_1$ in the canonical factorization of $f$. As an application they proved Mason's unbounded denominators conjecture for the case of a cuspidal parabolic GMF of weight $0$  on a congruence group.  

In the case
of two-dimensional representations of~$\Gamma(2)$, Mason's conjecture was settled by Franc and Mason~\cite{MasonConjectured,FrancMasonDim2}, and extended further by Franc, Gannon and Mason~\cite{FrancGannonMason} to the stronger sense of only requiring the $p$-adic boundedness of the coefficients for a density one set of primes~$p$. Their proof relies on the special property
 that the power series in~$\Q \llbracket x \rrbracket$ which arise in their context are hypergeometric functions.
 It is conceivable that the algebraicity part (over $\Q(x)$, respectively over the ring of classical modular forms) in Theorems~\ref{ODE form} and~\ref{ubd for vvmf} could likewise hold under a similar loosening of the integrality condition; but our proof does not imply this.  On the other hand, for representations of dimension $n \geq 3$, it is plain that the congruence property ceases to hold as in~\cite{FrancMasonDim2} if we relax $\Z\llbracket q \rrbracket$ to $\Z[1/S]\llbracket q \rrbracket$. 
 Another example where hypergeometric functions arise (this time for  three-dimensional representations of~$\PSL_2(\Z)$) 
 appears in the work of
 %The hypergeometric method  was extended to three-dimensional representations ($n=3$) of $\SL_2(\Z)$ by 
 Franc--Mason~\cite{FrancMasonVvmf} and Marks~\cite{Marks}, and was employed by~\cite{FrancMason} to derive certain cases of the original unbounded denominators conjecture.

\subsubsection{Logarithmic vector-valued modular forms}
If one drops the semisimplicity stipulation on $\rho  \left(\begin{pmatrix} 1 & 1 \\ 0 & 1 \end{pmatrix}\right)$ in the 
definition of a vector-valued modular form,  the resulting structure has been named a \emph{logarithmic vector-valued modular form} by Knopp and Mason~\cite{KnoppMasonLog}. They also do arise in conformal field theories, termed \emph{logarithmic} (in place of rational). See, for example, Fuchs--Schweigert~\cite{FuchsSchweigert}.  
The components of  a weight zero logarithmic vector-valued modular form with bounded denominators can now be classical
(congruence) modular forms of higher weight, and so certainly transcendental over~$\C(\lambda)$
(see Remark~\ref{log remark}).
In Question~\ref{logarithmic UDC} below, we give
an extension of the unbounded denominators problem over to the logarithmic setting. It remains outside the scope of our method as far as we could see. 
Before stating this question, we recall some basic facts concerning quasi-modular forms.
Recall that the ring of  \emph{quasi-modular forms}~$\WM(\Gamma)$ for~$\Gamma \subset \PSL_2(\Z)$ may be identified with the ring generated by~$E_2$
over the ring of classical holomorphic modular forms~$M(\Gamma)$ of integral weight for~$\Gamma$~\cite[Prop~20(ii)]{Zagier123}, and 
by~\cite[Prop~20(i)]{Zagier123} it is stable
under the operator
$$\theta = q \cdot \frac{d}{dq} =  \frac{1}{\pi i } \cdot \frac{d}{d \tau}.$$
Let~$M^{!}(\Gamma)$ denote the ring of weakly holomorphic modular forms for~$\Gamma$; that is,
the meromorphic modular forms which are holomorphic away from the cusps.

\begin{df} The ring of \emph{weakly holomorphic quasi-modular forms}~$\WM^{!}(\Gamma)$ for~$\Gamma \subset \PSL_2(\Z)$ is  the 
ring~$\WM(\Gamma)[1/\Delta]$.
\end{df}

\begin{lemma} The ring~$\WM^{!}(\Gamma)$ is the smallest ring which contains~$M^{!}(\Gamma)$
and which is closed under~$\theta$.
\end{lemma}

\begin{proof}  If~$f \in M^{!}(\Gamma)$ then~$f \Delta^m \in M(\Gamma) \subset \WM(\Gamma)$
for some~$m$ and thus~$M^{!}(\Gamma) \subset \WM^{!}(\Gamma)$.
Recall that~$\theta \Delta = E_2 \Delta$.
If~$g \in \WM^{!}(\Gamma)$, then~$h = \Delta^m g \in \WM(\Gamma)$ for some~$m$,
and thus
$$ \theta g =  \theta \frac{h}{\Delta^m} =
 \frac{\theta h}{\Delta^m} - \frac{m E_2 h}{\Delta^m} \in \WM^{!}(\Gamma),$$
 and hence~$\WM^{!}(\Gamma)$ is closed under~$\theta$. 
 Finally, any ring containing~$M^{!}(\Gamma)$ and closed under~$\theta$ contains both~$\Delta^{-1}$
 and~$E_2 = \theta \Delta/\Delta$ and thus contains~$\WM^{!}(\Gamma)$.
 \end{proof}

\begin{question}  \label{logarithmic UDC}
If a component $F_j(\tau)$ of a logarithmic vector-valued modular form for $\PSL_2(\Z)$ has a $\Z\llbracket q\rrbracket$ Fourier expansion, does $F_j(\tau)$ belong to the ring of 
weakly holomorphic quasi-modular forms~$\WM^{!}(\Gamma)$ 
for some \emph{congruence} subgroup~$\Gamma \subset \PSL_2(\Z)$?
\end{question}

Recall the classical Jacobi theta functions:
$$\vartheta_2 =  \sum_{n \in \Z} q^{(n+1/2)^2}, \quad
\vartheta_3 =  \sum_{n \in \Z} q^{n^2}, 
\quad
\vartheta_4 =   \sum_{n \in \Z}(-1)^n q^{n^2}.$$
By Jacobi's triple product identity, these functions~$\vartheta_i$ have explicit representations in terms of the Dedekind~$\eta$ function:
$$\vartheta_2 =   \frac{2 \eta^2(2 \tau)}{\eta(\tau)},
\qquad
\vartheta_3 =  \frac{\eta^5(\tau)}{\eta^2(\tau/2) \eta^2(2 \tau)}, 
\qquad
\vartheta_4 =    \frac{\eta^2(\tau/2)}{\eta(\tau)},$$
and hence they are holomorphic modular forms of weight~$1/2$ without any zeros on $\H$. 
Consequently, all the Laurent monomials $\vartheta_2^a\vartheta_3^b\vartheta_4^c$ (with $a,b,c \in \Z$) of an even degree~$a+b+c$ 
belong to~$ \WM(\Gamma)$
for some fixed congruence subgroup~$\Gamma$ (one can take~$\Gamma = \Gamma(12)$,
although some monomials are invariant under smaller groups, for example: $(\vartheta_2/\vartheta_3)^4 = \lambda$ by Equation~(\ref{Legendre})).
Finally, we also have~$2 (\theta \vartheta_i)/\vartheta_i =  (\theta \vartheta^2_i)/\vartheta^2_i \in \WM^{!}(\Gamma)$.

\medskip

We now turn to some basic examples hinting towards  a positive answer to question~\ref{logarithmic UDC}.
Complementing  Example~\ref{log remark} is the $\lambda$-pullback of the complete elliptic integral of the second kind: 
$$ 
\frac{2}{\pi}E(\lambda(q)) := \pFq{2}{1}{{1/2}, \, {-1/2}}{1}{\lambda(q)} = 1 - 4q + 20q^2 -64q^3 + 164q^4 - 392q^5 + \cdots \in \Z\llbracket q\rrbracket,
$$
 clearly a component of a logarithmic vector-valued modular form on $\PSL_2(\Z)$,  whose $q$-expansion is in $\Z\llbracket q \rrbracket$.
 But one can indeed verify that
$$
\frac{2}{\pi}E(\lambda(q)) =
\frac{\vartheta_3 \vartheta^4_4 + 4 \theta \vartheta_3}{\vartheta^3_3}
% \frac{\Big( \sum_{n \in \Z}q^{n^2} \Big) \Big( \sum_{n \in \Z}(-1)^n q^{n^2} \Big)^4   + 4 q\frac{d}{dq} \Big( \sum_{n \in \Z}q^{n^2} \Big) }{\Big( \sum_{n \in \Z}q^{n^2} \Big)^3}
$$
is also an element of~$\WM^{!}(\Gamma)$ for the congruence subgroup~$\Gamma 
= \Gamma(12)$.
%the quasi-modular ring for the congruence level $\Gamma(2)$. 
%\subsubsection{Example}
One can express $E$ in terms of $K$ and its integral: 
$$
\frac{2}{\pi} E(16x)    =   (16x-1) \, \frac{2}{\pi}  K(16x) - 8 \int_0^{x} \frac{2}{\pi}K(16t) \, dt, 
$$
where  one finds that 
$$
\int_0^x \frac{2}{\pi}K(16t) \, dt = \sum_{n=0}^{\infty} \frac{1}{n+1}\binom{2n}{n}^2 x^{n+1} \in \Z\llbracket x \rrbracket,
$$
the integrality of the coefficients now manifested by the Catalan numbers $C_n = \frac{1}{n+1} \binom{2n}{n} \in \Z$. 
One further integration still has integer coefficients:
$$
\int_0^{x} \frac{dy}{y} \int_0^{y} \frac{2}{\pi}K(16t) \, dt = \sum_{n=0}^{\infty} 
\frac{1}{(n+1)^2} \binom{2n}{n}^2x^{n+1}  = \sum_{n =0}^{\infty} C_n^2 \, x^{n+1} \in \Z\llbracket x \rrbracket,
$$
and $\sum \frac{1}{(n+1)^2} \binom{2n}{n}^2 (\lambda/16)^n \in \Z\llbracket q \rrbracket$ is a component of a logarithmic vector-valued modular form with a $\Z\llbracket q \rrbracket$ expansion. Zudilin has pointed out to us the formula
\begin{equation*}
\sum_{n=0}^{\infty} \frac{1}{(n+1)^2} \binom{2n}{n}^2 \big(\lambda(q) \big/ 16  \big)^n 
= \frac{4}{\vartheta^4_2} 
\left( 4 \vartheta^2_3  \cdot \frac{\theta \vartheta_2}{\vartheta_2}
+ 4 \vartheta_3 \cdot \theta \vartheta_3
-  \vartheta^4_3 \right)
\end{equation*}
exhibiting this $\Z\llbracket q \rrbracket$ power series as an element 
of~$\WM^{!}(\Gamma)$ for the congruence subgroup~$\Gamma = \Gamma(12)$, in accordance with Question~\ref{logarithmic UDC}.

\subsubsection{Some variations}  \label{the easy ones}
Our proof of Theorems~\ref{theorem:main} and~\ref{ubd for vvmf} is readily refined to yield a further precision in two regards: 

Firstly, the condition on $\Z\llbracket q^{1/N} \rrbracket$ Fourier coefficients can be relaxed to $\Z\llbracket q^{1/N} \rrbracket \otimes \C$ Fourier coefficients.

Secondly, the condition that the modular form $f(\tau)$, respectively the vector-valued modular form $F(\tau)$ are holomorphic on $\H$ can be relaxed to the condition of meromorphy on $\H$. 

We leave it to the interested reader to  fill in the details of these further extensions of our results.

\subsubsection{Beyond~$\SL_2(\Z)$} Much less obvious is how to extend our results to arithmetic groups other than $\SL_2(\Z)$. Here 
are two possible settings one could consider.

 Firstly, the group $\SL_2(\F_q[t])$ in function field arithmetic  and its attendant theory of Drinfeld--Goss modular forms. 
See Pellarin~\cite{Pellarin} for a recent  survey of this area. 
Here, in the analogy with $\SL_2(\Z)$ where the congruence kernels of these two arithmetic groups are similarly large, it would be interesting to decide whether the modular forms on a finite index subgroup of $\SL_2(\F_q[t])$ that have (up to a $\F_q(t)^{\times}$ scalar multiple) a $u$-expansion~\cite[\SSS~4.7.1]{Pellarin} with coefficients in $A = \F_q[t]$ are likewise the congruence modular forms. 

 Secondly, the mapping class groups $\Gamma_{g,n} = \mathrm{Mod}(S_{g,n})$ in  signatures $(g,n)$ 
other than $(1,1), (1,0)$ or $(0,4)$ that we have implicitly been limiting to.  Recall that $\Gamma_{1,1}\cong \Gamma_{1,0}  = \mathrm{Mod}(\mathbb{T}^2) = \SL_2(\Z)$ and $\Gamma_{0,4} \cong \PSL_2(\Z) \ltimes (\Z/2 \times \Z/2)$, and correspondingly the discussion in the rational conformal field theory under~\SSSS~\ref{cft} has been for the $1$-loop partition function with  a complex torus ($g = 1$) as the worldsheet~\cite{Gannon}. 
In a more recent
research stream in two-dimensional conformal field theory, a  higher genus extension of Zhu's modularity theorem 
would  associate to any holomorphic vertex operator  algebra a Teichm\"uller modular form (as defined in~\cite{Ichikawa})
in every signature $(g,n)$; this is a section of a tensor power
  $\lambda^{\otimes k}$ of the Hodge bundle over $\overline{\mathcal{M}_{g,n}}$.
  One could ask about extending the cruder algebraicity proviso of our Theorem~\ref{ubd for vvmf} over to
the more general setting of a component of a vector-valued Teichm\"uller modular form that has an appropriate integrality property.

\subsubsection{The Ramanujan--Petersson question for cuspidal vector-valued modular forms} \label{Ramanujan query} In the bulk of our paper, our theorems were stated for modular forms of an arbitrary integral\footnote{They even hold for the modular forms of half-integral weight, upon multiplying by a weight-$1/2$ theta function.} weight, but the path to arithmetic algebraization methods was always through a straightforward reduction to weight~$0$.  This is because it did not make a difference in our integrality questions as to whether or not the forms in question were cuspidal. At the same time, all our theorems can be equivalently (without loss of generality) stated for the cuspidal forms. This offers a common global framework for unifying our Theorem~\ref{ubd for vvmf} with the classical Ramanujan--Petersson conjectures. 

 Consider a representation~$\rho : \PSL_2(\Z) \to \GL_n(\C)$. 
For each integer~$k \in \Z$,  denote by~$S_k(\rho)$ the~$\Q$-vector space of the weight~$k$ vector-valued modular form whose multiplier system is~$\rho$ and whose~$n$ components all belong to~$q^{1/N} \Q\llbracket q^{1/N} \rrbracket$ for some (unspecified)~$N \in \NwithoutzeroA$. We can focus our question on the~$q$-expansion coefficients of an arbitrary~$q \, \Q\llbracket q \rrbracket$ component~$f(q) = \sum_{l =1}^{\infty} a_l q^l$ of an element of~$S_k(\rho)$. 
The growth of those coefficients acquires a global, multicolored meaning upon completing~$\Q \hookrightarrow \Q_v$ at the different places~$v \in M_{\Q}$ of~$\Q$. Denote by~$|\cdot|_v$ the usual absolute value normalized by~$|e| = e$, if~$v = \infty$, and by~$|p|_p = 1/p$, if~$v$ is the~$p$-adic place.  We can define a set~$\Sigma(\rho) \subset M_{\Q}$ of the \emph{deficient places} for the multiplier system~$\rho$ by declaring~$v \in \Sigma(\rho)$ if and only if there exists such a Fourier series component~$f(q) = \sum_{l=1}^{\infty} a_l q^l \in q \, \Q\llbracket q \rrbracket$, for some weight~$k \in \Z$ and some cuspidal vector-valued modular form~$F \in S_k(\rho)$ to the multiplier system~$\rho$ that has~$f$ as one of its~$n$ component functions, such that
$$
\left\{ \begin{array}{ll}  |a_l|_v\big/l^{\frac{k-1+\varepsilon}{2}} \textrm{ is unbounded for some~$\varepsilon > 0$},  & \textrm{ if } v = \infty; \\ |a_l|_v \textrm{ is unbounded}, & \textrm{ if } v \in M_{\Q}^{\mathrm{fin}}.  \end{array} \right.
$$
The conjunction of Shimura's integrality theorem~\cite[\S~8]{ShimuraEichler} and Deligne's resolution~\cite[{\it Th\'eor\`eme~8.2}]{DeligneWeilI} of the Ramanujan--Petersson conjecture are exactly packaged together into the statement that~$\Sigma(\rho) = \emptyset$ (no deficiency places) for the case that the representation~$\rho : \PSL_2(\Z) \to \GL_n(\C)$ has kernel a congruence subgroup of~$\PSL_2(\Z)$. On the other hand, our proof of Theorem~\ref{ubd for vvmf} established that~$\Sigma(\rho) \neq \emptyset$ (and, more precisely, contains at least one non-archimedean place) in every other case, to wit: whenever the kernel~$\ker(\rho)$ of the multiplier representation~$\rho : \PSL_2(\Z) \to \GL_n(\C)$ is not a congruence subgroup of~$\PSL_2(\Z)$. As far as we are aware, it is an open question whether there exists any~$\rho$ with~$\infty \in \Sigma(\rho)$, but also
whether there exists any~$\rho$ with~$\infty \notin \Sigma(\rho)$ but yet with~$\ker(\rho)$ \emph{not} a congruence subgroup of~$\PSL_2(\Z)$. We note however Selberg's result~\cite[\S~2]{Selberg} that at least~$|a_l|_{\infty}  = O\left( l^{\frac{k}{2}-\frac{1}{5}} \right)$ (an improvement over Hecke's trivial~$\ll l^{k/2}$ bound, based on the Rankin--Selberg method) is in place for a completely arbitrary~$\rho : \PSL_2(\Z) \to \GL_n(\C)$. Another question is whether the deficient places set~$\Sigma(\rho)$ is always finite for a given multiplier representation~$\rho : \PSL_2(\Z) \to \GL_n(\C)$.  Eisenstein's theorem~\cite[\S11.4]{BombieriGubler} guarantees that this is so in the case that the representation~$\rho$ has a finite image. 

\subsubsection{Algebraic fundamental groups} 
Finally we return to our introductory outline~\SSSS~\ref{intro sketch} where we acknowledged that our approach to the unbounded denominators conjecture has been  particularly inspired by the papers of Ihara~\cite{Ihara} and Bost~\cite{Bost} on arithmetic algebraization and Lefschetz theorems
in Arakelov geometry. Our central overconvergence boost emerged from the isogeny $[N]$ of $\mathbb{G}_m$ to trade a Bely\u{\i} map, or more
generally a local system on $\mathbb{P}^1 \setminus \{0, 1, \infty\}$ that has a $\Z/N$ local monodromy around $x = 0$, for a local system on $\mathbb{P}^1 \setminus \{\mu_N \cup \infty\}$:  the step of
extending through the falsely apparent singularity at $ x = 0$.   This is directly inspired by Ihara's employment of an arithmetic rationality theorem 
of Harbater~\cite[\SSS~1~Lemma]{Ihara} to derive $\pi_1$ results on certain arithmetic schemes, including for instance a Diophantine analysis proof of Saito's example of
$\pi_1 \big( \spec \, \Z[x, 1/x, 1/(x-1)] \big) = \{1\}$. In a similar fashion,
our Theorem~\ref{theorem:main} can be used to establish a $\pi_1$ result
in the style of Bost~\cite{Bost}.

\begin{theorem}  \label{pi1}
Let $N \in \NwithoutzeroA$, let $K /\Q(\mu_N)$ be a finite extension, and let  $\pi: \mathcal{X}(N) \to \spec{O_K}$ (``connected N\'eron model'') be the connected component containing the cusp $\infty$ in the smooth part of the minimal regular model of $X(N)$ over $\spec{O_K}$.  Thus the cusp $\infty$ extends to a morphism  $\varepsilon : \spec{O_K} \to \mathcal{X}(N)$. 

  Then, for every geometric point $\eta$ of $\spec{O_K}$, the maps of algebraic fundamental groups
$$
\pi_* : \pi_1(\mathcal{X}(N), \varepsilon(\eta))  \to \pi_1(\spec{O_K}, \eta)
$$
and
$$
\varepsilon_* :    \pi_1(\spec{O_K}, \eta)   \to  \pi_1(\mathcal{X}(N), \varepsilon(\eta))
$$
are mutually inverse isomorphisms. 
\end{theorem}

\begin{proof}[Proof (a sketch)]
This follows rather formally by the argument of~\cite[\SSS~4 on page 252]{Ihara} and also  \cite[proof of Theorem~1 \emph{loc.cit.} on pages~248--249]{Ihara}, upon replacing Ihara's  function field $k(t)$  by the modular function field~$K(X(N))$ and Ihara's formal power series ring $\mathfrak{O}\llbracket t \rrbracket$ by $O_K\llbracket \lambda(\tau/N)/16 \rrbracket$, taking account of Remark~\ref{remark:numberfield},
  and on using our Theorem~\ref{theorem:main} in place of Harbater's arithmetic rationality input~\cite[Claim~1A on page 248]{Ihara}.  
  \end{proof}

\begin{remark}
Very recently, Bost and Charles~\cite{BostCharles} have obtained new relative $\pi_1$ finiteness theorems for certain quasi-projective
arithmetic surfaces  $\mathcal{X} \to \spec{O_K}$, including for the case~\cite[\SSS~9.3.4]{BostCharles} of the affine modular scheme $\mathcal{Y}(N)^{\mathrm{arith}} \to  \spec \Z$ that represents the functor ``full level $N$ structure'' ($N \geq 3$) in the sense of isomorphisms $\iota : ( \mu_N \times \Z / N\Z )_S \stackrel{\simeq}{\longrightarrow} \mathcal{E}[N]$ of finite flat group schemes over a test scheme $S$. 
\end{remark}

\begin{remark}
Another $\pi_1$ interpretation of the unbounded denominators conjecture, in terms of the Galois theory of the Tate curve and the congruence kernel of
$\SL_2(\Z)$, was given by Chen~\cite[Conjecture~5.5.10]{Chen}.
\end{remark}

Similarly to our choice of the isogeny $[N] : \mathbb{G}_m \to \mathbb{G}_m$, one could perhaps more directly consider the modular covering $X(2N) \to X(2)$ and use that it is totally ramified of index~$N$ over the three cusps of $X(2)$. Thus a local system $(\mathcal{E},\nabla)$ on the modular curve $Y(2) \cong \mathbb{P}^1 \setminus \{0,1,\infty\}$ that has  $\Z/N$ local monodromies around the three singularities  has its pullback $g^*\mathcal{E}$ under the modular covering $g: Y(2N) \to Y(2)$ extend through the cusps of $Y(2N)$ to a local system on the projective curve $X(2N)$.  See also Andr\'e~\cite[II \SSS~8.3]{Andre}, for a more general setting. Another natural approach to the unbounded denominators conjecture 
would then be to aim directly for rationality on the curve $X(2N)$, instead of for a tight algebraicity or holonomicity rank bound over $X(2)$. Certainly at least the algebraicity clause of Theorems~\ref{ODE form} and~\ref{ubd for vvmf} is also possible by this alternative higher genus route to an arithmetic
algebraization. 

It is tempting to approach Theorem~\ref{pi1} or the congruence property directly using the arithmetic rationality theorem of 
Bost and Chambert-Loir~\cite{BCL}, although we were unable to do so.
In these optics, it may be of some interest to remark that the case of Theorem~\ref{pi1} with $N = 6$ and $K$ a sufficiently large number field to attain semistable reduction is  contained in~\cite[Corollary~1.3 with Example~7.2.2 (i)]{Bost}. Indeed, the modular curve $X(6)$ has genus $1$ and turned into an elliptic curve using the cusp $\infty$ for the origin. Since this elliptic curve contains the automorphism $\begin{pmatrix} 1 & 1 \\ 0 & 1 \end{pmatrix}$ of order $6$, it has $j$-invariant $0$ and is analytically isomorphic with the complex torus $\C / \Z[\omega]$, $\omega = e^{\pi i/3} = \frac{1 + \sqrt{-3}}{2}$, with complex multiplication by the Eisenstein integers $\Z[\omega]$, 
and in particular extending to a (smooth, proper) abelian scheme over $\spec{O_K}$. Its Faltings height is 
$$
-\frac{1}{2} \log   \Big\{  \frac{1}{\sqrt{3}} \Big( \frac{\Gamma(1/3)}{\Gamma(2/3)} \Big)^3 \Big\} = -0.749\ldots < -0.05\ldots = \frac{1}{2}\log\frac{\pi}{
4 \, \mathrm{Im} \, \omega},
$$
by the Lerch--Chowla--Selberg formula making Bost's capacitary condition~\cite[Corollary~1.3]{Bost} apply, and this is the isolated minimum value of the Faltings height across all elliptic curves.  In practice this means that this complex torus has a ``large'' \emph{univalent} complex-analytic  uniformization (in the sense of conformal size from the origin $[\infty]$ and potential theory), sufficient to place this particular case of Theorem~\ref{pi1}  within the framework of arithmetic \emph{rationality} --- as opposed to algebraicity
or holonomicity --- theorems~\cite{Bost, BCL} on the algebraic curve $X(N)$.  Can such an approach be continued to all $N$?

\section{Acknowledgments}We would like
to thank Yves Andr\'{e} for a number of insightful remarks
on the first version of this manuscript, leading in particular to~\SSSS~\ref{lem_lexi} as an alternative and simplified approach to Theorem~\ref{abstract form}.
We would also like to thank 
 Michael Barz, Jean-Beno\^{i}t Bost, Fran\c{c}ois Charles, Pierre Deligne, Cameron Franc, Igor Frenkel, Javier Fres\'an, Jayce Getz, Kenz Kallal, Mark Kisin, Geoffrey Mason, Peter Sarnak, Alex Smith,  Richard Taylor, John Voight, 
and Wadim Zudilin
for useful remarks, suggestions, and corrections.

\bibliographystyle{amsalpha}
\bibliography{UDC}

\renewcommand{\MR}[1]{}
\providecommand{\bysame}{\leavevmode\hbox to3em{\hrulefill}\thinspace}
\providecommand{\MR}{\relax\ifhmode\unskip\space\fi MR }
\providecommand{\MRhref}[2]{%
  \href{http://www.ams.org/mathscinet-getitem?mr=#1}{#2}
}
\providecommand{\href}[2]{#2}
\begin{thebibliography}{ASVV10}

\bibitem[AM88]{AndersonMoore}
Greg Anderson and Greg Moore, \emph{Rationality in conformal field theory},
  Commun. Math. Physics \textbf{117} (1988), 119--136.

\bibitem[Ami75]{Amice}
Yvette Amice, \emph{Les nombres {$p$}-adiques}, Collection SUP: ``Le
  Math\'{e}maticien'', vol.~14, Presses Universitaires de France, Paris, 1975,
  Pr\'{e}face de Ch. Pisot. \MR{0447195}

\bibitem[And89]{AndreG}
Yves Andr\'e, \emph{{$G$}-{F}unctions and {G}eometry}, Aspects of Mathematics,
  no. E13, Friedr. Vieweg Sohn, Braunschweig, 1989. \MR{0990016}

\bibitem[And04]{Andre}
\bysame, \emph{Sur la conjecture des {$p$}-courbures de {G}rothendieck--{K}atz
  et un probl\`eme de {D}work}, Geometric Aspects of Dwork Theory, vol. I, de
  Gruyter, Berlin, 2004, pp.~55--112. \MR{2023288}

\bibitem[AS92]{AS}
Milton Abramowitz and Irene~A. Stegun (eds.), \emph{Handbook of {M}athematical
  {F}unctions with {F}ormulas, {G}raphs, and {M}athematical {T}ables}, Dover
  Publications, Inc., New York, 1992, Reprint of the 1972 edition. \MR{1225604}

\bibitem[ASD71]{ASD}
Arthur Oliver~Lonsdale Atkin and Henry Peter~Francis Swinnerton-Dyer,
  \emph{Modular forms on noncongruence subgroups}, Proc. Symposia Pure Math.:
  Combinatorics, vol. XIX, American Mathematical Society, 1971, pp.~1--26.
  \MR{0337781}

\bibitem[ASVV10]{ASVV}
Greg~D. Anderson, Toshiyuki Sugawa, Mavina~K. Vamanamurthy, and Matti~K.
  Vuorinen, \emph{Twice-punctured hyperbolic sphere with a conical singularity
  and generalized elliptic integral}, Math. Z. \textbf{266} (2010), no.~1,
  181--191. \MR{2670678}

\bibitem[Ban00]{BantayCov}
Peter Bantay, \emph{Frobenius--{S}chur indicators, the {K}lein-bottle
  amplitude, and the principle of orbifold covariance}, Phys. Lett. B
  \textbf{488} (2000), no.~2, 207--210. \MR{1782167}

\bibitem[Ban02]{BantayPerm}
\bysame, \emph{Permutation orbifolds}, Nuclear Phys. B \textbf{633} (2002),
  no.~3, 365--378. \MR{1910268}

\bibitem[Ban03]{Bantay}
\bysame, \emph{The kernel of the modular representation and the {G}alois action
  in {RCFT}}, Comm. Math. Phys. \textbf{233} (2003), no.~3, 423--438.
  \MR{1962117}

\bibitem[BC22]{BostCharles}
Jean-Beno\^{i}t Bost and Fran\c{c}ois Charles, \emph{Quasi-projective and
  formal-analytic arithmetic surfaces}, 2022,
  \color{blue}{\url{https://arxiv.org/abs/2206.14242}}\color{black},
  pp.~165+xx.

\bibitem[BCL09]{BCL}
Jean-Beno\^{i}t Bost and Antoine Chambert-Loir, \emph{Analytic curves in
  algebraic varieties over number fields}, Algebra, arithmetic and geometry: in
  honor of Yu. I. Manin. Vol.~I, Birkh\"auser Boston, Boston, MA, 2009,
  pp.~69--124. \MR{2641171}

\bibitem[Ber94]{Berger}
Gabriel Berger, \emph{Hecke operators on noncongruence subgroups}, C. R. Acad.
  Sci. Paris S\'{e}r. I Math. \textbf{319} (1994), no.~9, 915--919.
  \MR{1302789}

\bibitem[BG06]{BombieriGubler}
Enrico Bombieri and Walter Gubler, \emph{Heights in {D}iophantine {G}eometry},
  Cambridge New Mathematical Monographs, no.~4, Cambridge University Press,
  2006. \MR{2216774}

\bibitem[BG07]{BantayGannon}
Peter Bantay and Terry Gannon, \emph{Vector-valued modular functions for the
  modular groups and the hypergeometric equation}, Commun. Number Theory Phys.
  \textbf{1} (2007), 651--680. \MR{2412268}

\bibitem[Bil97]{Bilu}
Yuri Bilu, \emph{Limit distribution of small points on algebraic tori}, Duke
  Math. J. \textbf{89} (1997), no.~3, 465--476. \MR{1470340}

\bibitem[Bir94]{Birch}
Bryan Birch, \emph{Noncongruence subgroups, covers and drawings}, The
  Grothendieck theory of dessins d'enfants (Luminy 1993, ed. L. Schneps),
  Cambridge University Press, Cambridge, 1994, London Math. Soc. Lecture Note
  Series, vol.~{\bf 200}, pp.~25--46. \MR{1305392}

\bibitem[BK01]{BenbourenaneKorhonen}
Djamel Benbourenane and Risto Korhonen, \emph{On the growth of the logarithmic
  derivative}, Comput. Methods Funct. Theory \textbf{1} (2001), no.~2,
  301--310. \MR{1941127}

\bibitem[Bor94]{Borel}
\'Emile Borel, \emph{Sur une application d'un th\'eor\`eme de {M}. {H}adamard},
  Bulletin des sciences math\'ematiques \textbf{18} (1894), 22--25.

\bibitem[Bor92]{Borcherds}
Richard~E. Borcherds, \emph{Monstrous moonshine and monstrous {L}ie
  superalgebras}, Invent. Math. \textbf{109} (1992), 405--444. \MR{1172696}

\bibitem[Bos99]{Bost}
Jean-Beno\^{i}t Bost, \emph{Potential theory and {L}efschetz theorems for
  arithmetic surfaces}, Ann. Sci. \'Ecole Norm. Sup. (4) \textbf{32} (1999),
  241--312. \MR{1681810}

\bibitem[Bos04]{BostGerms}
Jean-Beno\^{\i}t Bost, \emph{Germs of analytic varieties in algebraic
  varieties: canonical metrics and arithmetic algebraization theorems},
  Geometric aspects of {D}work theory. {V}ol. {I}, Walter de Gruyter, Berlin,
  2004, pp.~371--418. \MR{2023294}

\bibitem[Bos13]{BostAlg}
\bysame, \emph{Algebraization, transcendence, and {$D$}-group schemes}, Notre
  Dame J. Form. Log. \textbf{54} (2013), no.~3-4, 377--434. \MR{3091663}

\bibitem[Bos20]{BostBook}
\bysame, \emph{Theta invariants of {E}uclidean lattices and
  infinite-dimensional {H}ermitian vector bundles over arithmetic curves},
  Progress in Mathematics, vol. 334, Birkh\"{a}user/Springer, 2020.
  \MR{4180991}

\bibitem[Car54]{Caratheodory}
C.~Carath\'eodory, \emph{Theory of {F}unctions of a {C}omplex {V}ariable.
  {V}ol. 2}, Chelsea Publishing Co., New York, 1954, Translated by F.
  Steinhardt. \MR{0064861}

\bibitem[CE11]{CEgrowth}
Frank Calegari and Matthew Emerton, \emph{Mod-{$p$} cohomology growth in
  {$p$}-adic analytic towers of 3-manifolds}, Groups Geom. Dyn. \textbf{5}
  (2011), no.~2, 355--366. \MR{2782177}

\bibitem[CE12]{survey}
\bysame, \emph{Completed cohomology---a survey}, Non-abelian fundamental groups
  and {I}wasawa theory, London Math. Soc. Lecture Note Ser., vol. 393,
  Cambridge Univ. Press, Cambridge, 2012, pp.~239--257. \MR{2905536}

\bibitem[CE16]{MR3466858}
\bysame, \emph{Homological stability for completed homology}, Math. Ann.
  \textbf{364} (2016), no.~3-4, 1025--1041. \MR{3466858}

\bibitem[Che18]{Chen}
William~Yun Chen, \emph{Moduli interpretations for noncongruence modular
  curves}, Math. Ann. \textbf{371} (2018), 41--126. \MR{3788845}

\bibitem[CL02]{ACL}
Antoine Chambert-Loir, \emph{Th\'eor\`emes d'alg\'ebricit\'e en g\'eom\'etrie
  diophantienne (d'apr\`es {J.-B.} {B}ost, {Y}. {A}ndr\'e, {D}. \& {G}.
  {C}hudnovsky)}, S\'eminaire Bourbaki. Ast\'erisque \textbf{282} (2002),
  175--209. \MR{1975179}

\bibitem[CV19]{CalegariVenkatesh}
Frank Calegari and Akshay Venkatesh, \emph{A torsion {J}acquet-{L}anglands
  correspondence}, Ast\'{e}risque (2019), no.~409, x+226. \MR{3961523}

\bibitem[CY01]{CherryYe}
William Cherry and Zhuan Ye, \emph{Nevanlinna's {T}heory of {V}alue
  {D}istribution}, Springer Monographs in Mathematics, Springer Verlag, Berlin,
  2001. \MR{1831783}

\bibitem[DDT97]{DDT}
Henri Darmon, Fred Diamond, and Richard Taylor, \emph{Fermat's last theorem},
  Elliptic curves, modular forms \& {F}ermat's last theorem ({H}ong {K}ong,
  1993), Int. Press, Cambridge, MA, 1997, pp.~2--140. \MR{1605752}

\bibitem[Del74]{DeligneWeilI}
Pierre Deligne, \emph{La conjecture de {W}eil. {I}}, Inst. Hautes \'Etudes Sci.
  Publ. Math. (1974), no.~43, 273--307. \MR{340258}

\bibitem[DGS94]{DworkGerottoSullivan}
Bernard Dwork, Giovanni Gerotto, and Francis~J. Sullivan, \emph{An introduction
  to {$G$}-functions}, Annals of Mathematics Studies, Princeton University
  Press, Princeton, NJ, 1994. \MR{1274045}

\bibitem[DLM00]{DongLiMason}
Chongying Dong, Haisheng Li, and Geoffrey Mason, \emph{Modular-invariance of
  trace functions in orbifold theory and generalized {M}oonshine}, Comm. Math.
  Phys. \textbf{214} (2000), 1--56. \MR{1794264}

\bibitem[DLN15]{DongLinNg}
Chongying Dong, Xingjun Lin, and Siu-Hung Ng, \emph{Congruence property in
  conformal field theory}, Algebra and Number Theory \textbf{9} (2015),
  2121--2166. \MR{MR3435813}

\bibitem[DR18]{DongRen}
Chongying Dong and Li~Ren, \emph{Congruence property in orbifold theory}, Proc.
  Amer. Math. Soc. \textbf{146} (2018), no.~2, 497--506. \MR{3731686}

\bibitem[dSG16]{SaintGervais}
Henri~Paul de~Saint-Gervais, \emph{Uniformization of {R}iemann surfaces:
  revisiting a hundred-year-old theorem}, Heritage of European Mathematics,
  European Mathematical Society (EMS), Z\"urich, 2016, By A. Alvarez, Ch.
  Bavard, F. B\'eguin, N. Bergeron, M. Bourrigan, B. Deroin, S. Dumitrescu, Ch.
  Frances, \'E. Ghys, A. Guilloux, F. Loray, P. Popescu-Pampu, P. Py, B.
  S\'evennec, and J.-C. Sikorav. Translated from the 2010 French original by
  Robert G. Burns. \MR{3494804}

\bibitem[DT97]{DrmotaTichy}
Michael Drmota and Robert~F. Tichy, \emph{Sequences, discrepancies and
  applications}, Lecture Notes in Mathematics, vol. 1651, Springer-Verlag,
  Berlin, 1997. \MR{0470456}

\bibitem[EGM98]{ElstrodtGrunewaldMennicke}
J\"urgen Elstrodt, Fritz Grunewald, and Jens Mennicke, \emph{Groups acting on
  hyperbolic space. {H}armonic analysis and number theory}, Springer Monographs
  in Mathematics, Springer-Verlag, Berlin, 1998. \MR{1483315}

\bibitem[Eho95]{Eholzer}
Wolfgang Eholzer, \emph{On the classification of modular fusion algebras},
  Commun. Math. Physics \textbf{172} (1995), 623--659.

\bibitem[FF22]{FioriFranc}
Andrew Fiori and Cameron Franc, \emph{The unbounded denominators conjecture for
  the noncongruence subgroups of index 7}, J. Number Theory \textbf{240}
  (2022), 611--640. \MR{4458253}

\bibitem[FGM18]{FrancGannonMason}
Cameron Franc, Terry Gannon, and Geoffrey Mason, \emph{On unbounded
  denominators and hypergeometric series}, J. Number Theory \textbf{192}
  (2018), 197--220. \MR{3841552}

\bibitem[FLM88]{FLM}
Igor Frenkel, James Lepowsky, and Arne Meurman, \emph{Vertex operator algebras
  and the {M}onster}, Pure and Applied Mathematics, vol. 134, Academic Press,
  Inc., Boston, MA, 1988. \MR{0996026}

\bibitem[FM14]{FrancMasonDim2}
Cameron Franc and Geoffrey Mason, \emph{Fourier coefficients of vector-valued
  modular forms of dimension~$2$}, Canad. Math. Bull. \textbf{57} (2014),
  485--494. \MR{3239110}

\bibitem[FM16a]{FrancMasonVvmf}
\bysame, \emph{Hypergeometric series, modular linear differential equations,
  and vector-valued modular forms}, Ramanujan J. \textbf{41} (2016), no.~1--3,
  233--267. \MR{MR3574630}

\bibitem[FM16b]{FrancMason}
\bysame, \emph{Three-dimensional imprimitive representations of the modular
  group and their associated modular forms}, J. Number Theory \textbf{160}
  (2016), 186--214. \MR{MR3425204}

\bibitem[FS19]{FuchsSchweigert}
J\"urgen Fuchs and Christoph Schweigert, \emph{Full logarithmic conformal field
  theory --- an attempt at a status report}, Fortschr. Phys. (2019), no.~8--9,
  Special issue: Proceedings of the LMS/PESR Durham Symposium on Higher
  Structures in M-theory, 1910018, 12 pp. \MR{4016024}

\bibitem[Gan06]{Gannon}
Terry Gannon, \emph{Moonshine beyond the {M}onster: The bridge connecting
  algebra, modular forms and physics}, Cambridge Monographs on Mathematical
  Physics, Cambridge University Press, 2006. \MR{2257727}

\bibitem[Gan14]{GannonVVMF}
\bysame, \emph{The theory of vector-valued modular forms for the modular
  group}, Conformal {F}ield {T}heory, {A}utomorphic {F}orms and {R}elated
  {Topics}, Springer-Verlag, 2014, pp.~247--286.

\bibitem[GG76]{GoldbergGrinshtein}
A.A. Goldberg and V.A. Grinshtein, \emph{The logarithmic derivative of a
  meromorphic function}, Math. Notes \textbf{19} (1976), 320--323.

\bibitem[Gol69]{Goluzin}
Gennadiy~Mikhailovich Goluzin, \emph{Geometric {T}heory of {F}unctions of a
  {C}omplex {V}ariable}, Translations of Mathematical Monographs, Vol. 26,
  American Mathematical Society, Providence, R.I., 1969. \MR{0247039}

\bibitem[Got20]{Gottesman}
Richard Gottesman, \emph{The arithmetic of vector-valued modular forms on
  {$\Gamma_0(2)$}}, Int. J. Number Theory \textbf{16} (2020), no.~2, 241--289.
  \MR{4077422}

\bibitem[Hay64]{Hayman}
Walter~K. Hayman, \emph{Meromorphic {F}unctions}, Oxford Mathematical
  Monographs, Clarendon Press, Oxford, 1964. \MR{0164038}

\bibitem[Hem88]{Hempel}
Joachim~A. Hempel, \emph{On the uniformization of the {$n$}-punctured sphere},
  Bull. London Math. Soc. \textbf{20} (1988), no.~2, 97--115. \MR{924235}

\bibitem[Hin92]{Hinkkanen}
Aimo Hinkkanen, \emph{A sharp form of {N}evanlinna's second main theorem},
  Invent. Math. \textbf{108} (1992), 549--574. \MR{1163238}

\bibitem[Hup67]{Huppert}
B.~Huppert, \emph{Endliche {G}ruppen. {I}}, Die Grundlehren der mathematischen
  Wissenschaften, Band 134, Springer-Verlag, Berlin-New York, 1967.
  \MR{0224703}

\bibitem[Ich94]{Ichikawa}
Takashi Ichikawa, \emph{On {T}eichm\"{u}ller modular forms}, Math. Ann.
  \textbf{299} (1994), no.~4, 731--740. \MR{1286895}

\bibitem[Iha94]{Ihara}
Yasutaka Ihara, \emph{Horizontal divisors on arithmetic surfaces associated
  with {B}ely\u{\i} uniformizations}, The Grothendieck theory of dessins
  d'enfants (Luminy 1993, ed. L. Schneps), Cambridge University Press,
  Cambridge, 1994, London Math. Soc. Lecture Note Series, vol.~{\bf 200},
  pp.~245--254. \MR{1305400}

\bibitem[Kat70]{Katz}
Nicholas~M. Katz, \emph{Nilpotent connections and the monodromy theorem:
  Applications of a result of {T}urrittin}, Inst. Hautes \'Etudes Sci. Publ.
  Math. \textbf{39} (1970), 175--232. \MR{0291177}

\bibitem[Kat73]{Katz349}
\bysame, \emph{{$p$}-adic properties of modular schemes and modular forms},
  Modular functions of one variable, {III} ({P}roc. {I}nternat. {S}ummer
  {S}chool, {U}niv. {A}ntwerp, {A}ntwerp, 1972), 1973, pp.~69--190. Lecture
  Notes in Mathematics, Vol. 350. \MR{0447119}

\bibitem[KF17]{FrickeKlein}
Felix Klein and Robert Fricke, \emph{Lectures on the theory of elliptic modular
  functions. {V}ol. 1}, CTM. Classical Topics in Mathematics, vol.~1, Higher
  Education Press, Beijing, 2017, Translated from the German original by Arthur
  M. DuPre. \MR{3838340}

\bibitem[Kir05]{Kirsch}
Siegfried Kirsch, \emph{Transfinite diameter, {C}hebyshev constant and
  capacity}, Handbook of complex analysis: geometric function theory. {V}ol. 2,
  Elsevier Sci. B. V., Amsterdam, 2005, pp.~243--308. \MR{2121861}

\bibitem[KL08]{KL1}
Chris~A. Kurth and Ling Long, \emph{Modular forms for some noncongruence
  subgroups of $\mathrm{SL}(2,\mathbb{Z})$}, J. Number Theory \textbf{128}
  (2008), 1989--2009. \MR{MR2423745}

\bibitem[KL09]{KL2}
\bysame, \emph{On modular forms for some noncongruence subgroups of
  $\mathrm{SL}(2,\mathbb{Z})$. {II.}}, Bull. London Math. Soc. \textbf{41}
  (2009), 589--598. \MR{MR2521354}

\bibitem[KM03a]{KnoppMason}
Marvin Knopp and Geoffrey Mason, \emph{Generalized modular forms}, J. Number
  Theory \textbf{99} (2003), 1--28. \MR{MR1957241}

\bibitem[KM03b]{KnoppMasonFourier}
\bysame, \emph{On vector-valued modular forms and their {F}ourier
  coefficients}, Acta Arithmetica \textbf{110} (2003), 117--124. \MR{MR2008079}

\bibitem[KM08]{KohnenMason}
Winfried Kohnen and Geoffrey Mason, \emph{On generalized modular forms and
  their applications}, Nagoya Math. J. \textbf{192} (2008), 119--136.
  \MR{2477614}

\bibitem[KM09]{KnoppMasonFactorization}
Marvin Knopp and Geoffrey Mason, \emph{Parabolic generalized modular forms and
  their characters}, Int. J. Number Theory \textbf{5} (2009), no.~5, 845--857.
  \MR{2553511}

\bibitem[KM11]{KnoppMasonLog}
\bysame, \emph{Logarithmic vector-valued modular forms}, Acta Arithmetica
  \textbf{147} (2011), 261--282. \MR{2273205}

\bibitem[KM12]{KohnenMason2}
Winfried Kohnen and Geoffrey Mason, \emph{On the canonical decomposition of
  generalized modular functions}, Proc. Amer. Math. Soc. \textbf{140} (2012),
  no.~4, 1125--1132. \MR{2869098}

\bibitem[KR16]{KrausRoth}
Daniela Kraus and Oliver Roth, \emph{Sharp lower bounds for the hyperbolic
  metric of the complement of a closed subset of the unit circle and theorems
  of {S}chwarz-{P}ick-, {S}chottky- and {L}andau-type for analytic functions},
  Constr. Approx. \textbf{43} (2016), no.~1, 47--69. \MR{3439233}

\bibitem[KRS11]{KrausRothSugawa}
Daniela Kraus, Oliver Roth, and Toshiyuki Sugawa, \emph{Metrics with conical
  singularities on the sphere and sharp extensions of the theorems of {L}andau
  and {S}chottky}, Math. Z. \textbf{267} (2011), no.~3-4, 851--868.
  \MR{2776061}

\bibitem[Laz65]{Lazard}
Michel Lazard, \emph{Groupes analytiques {$p$}-adiques}, Inst. Hautes
  \'{E}tudes Sci. Publ. Math. (1965), no.~26, 389--603. \MR{209286}

\bibitem[LL12]{LiLong}
W.-C.W. Li and L.~Long, \emph{Fourier coefficients of noncongruence cuspforms},
  Bull. London Math. Soc. \textbf{44} (2012), no.~3, 591--598. \MR{2967004}

\bibitem[Lon08]{Long}
Ling Long, \emph{Finite index subgroups of the modular group and their modular
  forms}, Fields Inst. Commun. \textbf{54} (2008), 83--102. \MR{MR2454321}

\bibitem[Mar15]{Marks}
Christopher Marks, \emph{Fourier coefficients of three-dimensional
  vector-valued modular forms}, Commun. Number Theory Phys. \textbf{9} (2015),
  387--412. \MR{3361298}

\bibitem[Mas12]{MasonConjectured}
Geoffrey Mason, \emph{On the {F}ourier coefficients of 2-dimensional
  vector-valued modular forms}, Proc. Amer. Math. Socl. \textbf{140} (2012),
  no.~6, 1921--1930. \MR{2888179}

\bibitem[Men67]{Mennicke}
Jens~L. Mennicke, \emph{On {I}hara's modular group}, Invent. Math. \textbf{4}
  (1967), 202--228. \MR{0225894}

\bibitem[Nev70]{Nevanlinna}
Rolf Nevanlinna, \emph{Analytic {F}unctions}, Die Grundlehren der
  mathematischen Wissenschaften, Springer-Verlag, New York-Berlin, 1970.
  \MR{0164038}

\bibitem[NS10]{NgSchauenburg}
Siu-Hung Ng and Peter Schauenburg, \emph{Congruence subgroups and generalized
  {F}robenius-{S}chur indicators}, Comm. Math. Phys. \textbf{300} (2010),
  no.~1, 1--46. \MR{2725181}

\bibitem[Pel21]{Pellarin}
Federico Pellarin, \emph{From the {C}arlitz exponential to {D}rinfeld modular
  forms}, Arithmetic and geometry over local fields, Lecture Notes in Math.,
  vol. 2275, de Gruyter, Berlin, 2021, pp.~93--177. \MR{4238488}

\bibitem[Rib84]{RibetICM}
Kenneth~A. Ribet, \emph{Congruence relations between modular forms},
  Proceedings of the {I}nternational {C}ongress of {M}athematicians, {V}ol. 1,
  2 ({W}arsaw, 1983), PWN, Warsaw, 1984, pp.~503--514. \MR{804706}

\bibitem[Ru21]{MinRu}
Min Ru, \emph{Nevanlinna theory and its relation to {D}iophantine
  approximation}, World Scientific Publishing Co. Pte. Ltd., Hackensack, NJ,
  2021, Second edition. \MR{4265173}

\bibitem[Sel65]{Selberg}
Atle Selberg, \emph{On the estimation of {F}ourier coefficients of modular
  forms}, Proc. Sympos. Pure Math., American Mathematical Society, Providence,
  R.I., 1965, pp.~1--15. \MR{0182610}

\bibitem[Ser70]{SerreSL2}
Jean-Pierre Serre, \emph{Le probl\`eme des groupes de congruence pour
  {$\mathrm{SL}_2$}}, Ann. of Math. \textbf{2} (1970), 489--527. \MR{0272790}

\bibitem[Ser80]{Trees}
\bysame, \emph{Trees}, Springer-Verlag, Berlin-New York, 1980, Translated from
  the French by John Stillwell. \MR{607504}

\bibitem[Shi59]{ShimuraEichler}
Goro Shimura, \emph{Sur les int\'egrales attach\'ees aux formes automorphes},
  J. Math. Soc. Japan \textbf{11} (1959), 291--311. \MR{120372}

\bibitem[Shi71]{Shimura}
\bysame, \emph{Introduction to the arithmetic theory of automorphic functions},
  Kan\^{o} Memorial Lectures, No. 1, Publications of the Mathematical Society
  of Japan, No. 11. Iwanami Shoten, Publishers, Tokyo; Princeton University
  Press, Princeton, N.J., 1971. \MR{0314766}

\bibitem[Sti84]{Stiller}
Peter Stiller, \emph{Special values of {D}irichlet series, monodromy, and the
  periods of automorphic forms}, Mem. Amer. Math. Soc. \textbf{49} (1984),
  no.~299, iv+116. \MR{743544}

\bibitem[Swa60]{Swan}
Richard~G. Swan, \emph{The {$p$}-period of a finite group}, Illinois J. Math.
  \textbf{4} (1960), 341--346. \MR{122856}

\bibitem[SZ12]{SommerhauserZhu}
Yorck Sommerh\"{a}user and Yongchang Zhu, \emph{Hopf algebras and congruence
  subgroups}, Mem. Amer. Math. Soc. \textbf{219} (2012), no.~1028, vi+134.
  \MR{2985696}

\bibitem[Tao12]{Tao}
Terence Tao, \emph{Topics in random matrix theory}, Graduate Studies in
  Mathematics, vol. 132, American Mathematical Society, Providence, RI, 2012.
  \MR{2906465}

\bibitem[Tho89]{SerreThompson}
J.~G. Thompson, \emph{Hecke operators and noncongruence subgroups}, Group
  theory ({S}ingapore, 1987), de Gruyter, Berlin, 1989, Including a letter from
  J.-P. Serre, pp.~215--224. \MR{981844}

\bibitem[Tsu52]{Tsuji}
Masatsugu Tsuji, \emph{Theory of {F}uchsian groups}, Jpn. J. Math. \textbf{21}
  (1952), 1--27. \MR{0054718}

\bibitem[Woh64]{Wohlfahrt}
Klaus Wohlfahrt, \emph{An extension of {F}. {K}lein's level concept}, Illinois
  J. Math. \textbf{8} (1964), 529--535. \MR{167533}

\bibitem[Xu06]{Xu}
Feng Xu, \emph{Some computations in the cyclic permutations of completely
  rational sets}, Comm. Math. Phys. \textbf{267} (2006), no.~3, 757--782.
  \MR{MR2249790}

\bibitem[Zag08]{Zagier123}
Don Zagier, \emph{Elliptic modular forms and their applications}, The 1-2-3 of
  modular forms, Universitext, Springer, Berlin, 2008, pp.~1--103. \MR{2409678}

\bibitem[Zhu96]{Zhu}
Yongchang Zhu, \emph{Modular invariance of characters of vertex operator
  algebras}, J. Amer. Math. Soc. \textbf{9} (1996), no.~1, 237--302.
  \MR{MR1317233}

\end{thebibliography}

\end{document}